\documentclass[a4paper, oneside, 11pt]{amsart}

\usepackage{import}
\usepackage[english]{babel}
\usepackage{amsmath}
\usepackage{amsthm}
\usepackage[T1]{fontenc}
\usepackage{commath}
\usepackage{vwcol}
\usepackage[a4paper]{geometry}
 
\usepackage[utf8]{inputenc}
\usepackage{graphicx}
\usepackage{mathrsfs}
\usepackage{nicefrac}
\usepackage{xfrac}
\usepackage{multicol}
\usepackage{multirow}
\usepackage{float}
\usepackage[overload]{empheq}
\usepackage{epstopdf}
\usepackage{epsfig}
\usepackage{listings}
\usepackage{color}
\usepackage{mathtools}
\usepackage{tikz-cd}
\usepackage{relsize}
\usepackage{comment}
\usepackage{dsfont}
\usepackage{url}
\tikzset{shorten <>/.style={shorten >=#1,shorten <=#1}}
\DeclareRobustCommand{\SkipTocEntry}[9]{}

\usepackage[hidelinks]{hyperref}

\newfloat{Diagram}{htbp}{dia}

\theoremstyle{definition}
\newtheorem{definition}{Definition}[section]
\newtheorem{proposition}[definition]{Proposition}
\newtheorem{lemma}[definition]{Lemma}
\newtheorem{theorem}[definition]{Theorem}
\newtheorem*{theorem*}{Theorem}
\newtheorem*{definition*}{Definition}
\newtheorem*{example*}{Example}
\newtheorem{corollary}[definition]{Corollary}
\newtheorem{example}[definition]{Example}
\newtheorem{remark}[definition]{Remark}

\newtheorem*{rep@theorem}{\rep@title}
\newcommand{\newreptheorem}[2]{%
\newenvironment{rep#1}[1]{%
 \def\rep@title{#2 \ref{##1}}%
 \begin{rep@theorem}}%
 {\end{rep@theorem}}}
\makeatother

\newreptheorem{theorem}{Theorem}

\newtheorem*{rep@definition}{\rep@title}
\newcommand{\newrepdefinition}[2]{%
\newenvironment{rep#1}[1]{%
 \def\rep@title{#2 \ref{##1}}%
 \begin{rep@definition}}%
 {\end{rep@definition}}}
\makeatother

\newrepdefinition{definition}{Definition}

\newtheorem*{rep@example}{\rep@title}

\newcommand{\newrepexample}[2]{%
  \newenvironment{rep#1}[2]{%
    \def\rep@title{#2 \ref{##1}--\ref{##2}}%
    \begin{rep@example}%
  }{%
    \end{rep@example}%
  }%
}
\makeatother

\newrepexample{example}{Example}

\RequirePackage{tikz-cd}
\RequirePackage{amssymb}
\usetikzlibrary{calc}
\usetikzlibrary{decorations.pathmorphing}

\tikzset{curve/.style={settings={#1},to path={(\tikztostart)
			.. controls ($(\tikztostart)!\pv{pos}!(\tikztotarget)!\pv{height}!270:(\tikztotarget)$)
			and ($(\tikztostart)!1-\pv{pos}!(\tikztotarget)!\pv{height}!270:(\tikztotarget)$)
			.. (\tikztotarget)\tikztonodes}},
	settings/.code={\tikzset{quiver/.cd,#1}
		\def\pv##1{\pgfkeysvalueof{/tikz/quiver/##1}}},
	quiver/.cd,pos/.initial=0.35,height/.initial=0}

\tikzset{
  trim node/.default=1cm,
  trim node/.style={
    overlay,
    append after command={% restore smaller bounding box
      ([xshift={+#1}]\tikzlastnode.north west)
      ([xshift={+-#1}]\tikzlastnode.south east)}},
  down and trim/.default=1cm,
  down and trim/.style={
    yshift=-(\pgfmatrixcurrentcolumn-1)*1.5\baselineskip,
    trim node={#1}},
  downup and trim/.default=1cm,
  downup and trim/.style={
    yshift=iseven(\pgfmatrixcurrentcolumn) ? -1.5\baselineskip : 0pt,
    trim node={#1}},
  -|/.style={to path={-|(\tikztotarget)\tikztonodes}},
  |-/.style={to path={|-(\tikztotarget)\tikztonodes}},
  -| sl/.style={-|, xslant=-1},
  |- sl/.style={|-, xslant= 1},
  center picture/.style={
    trim left=(current bounding box.center),
    trim right=(current bounding box.center)}}

\tikzset{tail reversed/.code={\pgfsetarrowsstart{tikzcd to}}}
\tikzset{2tail/.code={\pgfsetarrowsstart{Implies[reversed]}}}
\tikzset{2tail reversed/.code={\pgfsetarrowsstart{Implies}}}
\tikzset{no body/.style={/tikz/dash pattern=on 0 off 1mm}}
\tikzset{
	labl/.style={anchor=south, rotate=90, inner sep=.5mm}
}

\newcommand{\Rep}{\text{Rep}}
\renewcommand{\Vec}{\text{Vec}}
\newcommand{\End}{\text{End}}
\newcommand{\Z}{\mathbb{Z}}

\newcommand{\Hom}{\text{Hom}}

\newcommand{\id}{\text{id}}

\newcommand{\1}{\mathds{1}}

\newcommand{\CC}{\mathbb{C}}

\newcommand{\R}{\mathbb{R}}

\newcommand{\Ccal}{\mathcal{C}}
\newcommand{\Dcal}{\mathcal{D}}
\newcommand{\Scal}{\mathcal{S}}
\newcommand{\Lcal}{\mathcal{L}}

\newcommand{\Hilb}{\text{Hilb}}

\newcommand{\MorC}{\mathrm{C}^*\text{Alg}}
\newcommand{\Ecal}{\mathcal{E}}
\newcommand{\Wcat}{\mathrm{W}^*\text{Cat}}
\newcommand{\Corr}{\text{Corr}}

\newcommand{\Tcal}{\mathcal{T}}
\newcommand{\Ycal}{\mathcal{Y}}
\newcommand{\TYcal}{\Tcal\Ycal}
\newcommand{\Fcal}{\mathcal{F}}

\newcommand{\Ffrak}{\mathfrak{F}}
\newcommand{\Bcal}{\mathscr{B}}

\title{Continuous Tambara-Yamagami Tensor Categories}
\author{Adrià Marín-Salvador}
\date{}
\email{marin@maths.ox.ac.uk}

\begin{document}

\begin{abstract}
We introduce continuous tensor categories as algebra objects in the Morita bicategory of $\mathrm{C}^*$-algebras to study ``semisimple'' tensor categories with a topological space of objects and continuous fusion rules and associators. In this setting, we generalize the construction of Tambara-Yamagami tensor categories from finite abelian groups to locally compact abelian groups, and provide a classification of continuous Tambara-Yamagami tensor categories for a locally compact abelian group $G$. A continuous Tambara-Yamagami tensor category associated to a locally compact abelian group $G$ is a continuous tensor category that has a single non-invertible simple object $\tau$ such that $\tau\otimes \tau$ decomposes as a direct integral indexed over $G$, meaning $\tau\otimes\tau  \cong L^2(G)$. We show that continuous Tambara-Yamagami tensor categories for $G$ are classified by a continuous symmetric nondegenerate bicharacter $\chi: G\times G\to U(1)$ and a sign $\xi\in\{\pm 1\}$. We also prove that, if a $\mathrm{W}^*$-tensor category $\Ccal$ obeys the Tambara-Yamagami fusion rules, then its associators are automatically continuous in the sense that $\Ccal$ is obtained from a continuous tensor category by ``forgetting its topology''.
\end{abstract}

\maketitle

\tableofcontents

\section{Introduction}
\addtocontents{toc}{\protect\setcounter{tocdepth}{-1}}

\subsection*{Background and motivation}Fusion (i.e. finite, rigid, semisimple tensor) categories appear in different areas of mathematics, including representation theory \cite{NV}, conformal field theory (CFT) \cite{frsI, frsII, frsIII, frsIV, frsV}, topological quantum field theory \cite{TV, barretwestbury, DSPS}, quantum groups \cite{andersen, andersenparadowski}, or vertex operator algebras \cite{huang}, and they have been extensively studied, in part thanks to their very combinatorial nature. However, in many of these areas, one has to restrict attention to a particular class of objects, like finite groups in representation theory or strongly rational CFTs, in order to obtain categories with the finiteness and semisimplicity conditions imposed in the definition of fusion categories. It is therefore reasonable to aim to work with a larger class of tensor categories which still preserve some of the convenient features of fusion categories but include a wider class of examples which arise naturally. We provide a step towards this generalization.

Our main motivation comes from 2-dimensional unitary CFT. To such a CFT, one can associate a tensor category of quantum symmetries. While for some CFTs the categories of symmetries are fusion, there are many CFTs for which the quantum symmetries form non-trivial topological spaces \cite{categoriosities, fredenhagen}. As an example, the massless boson on a line and on a plane have categories of symmetries whose simple objects are parametrized by $\R$ and $\R^2$ respectively \cite{fredenhagen, RepHeis}. In addition, the orbifold of the latter under the action of $SO(2)$ has simple objects whose tensor product is a direct integral of simple objects. Topological spaces of simple objects and direct integrals also appear in the study of positive representations of the quantum groups $U_q(\mathfrak{sl}_n(\R))$ \cite{Ponsot01, schrader2017continuoustensorcategoriesquantum}. Positive representations of $U_q(\mathfrak{sl}_n(\R))$ are in bijection with points in the Weyl chamber of $\mathfrak{sl}_n(\R)$, and direct integrals with respect to certain measures on the Weyl chamber are needed to describe their tensor product. From a different viewpoint, an interest in generalizations of $\Vec\,G$ (the tensor category of $G$-graded vector spaces) from finite groups to Lie groups, as well as possible definitions and induced TQFTs, can also be found in \cite{Walker}.

The examples above illustrate the need for a mathematical treatment of tensor categories which are still semisimple but which have a topological space of simple objects. However, asking for a topology on the space of simple objects of a tensor category is a rather weak requirement, as one would also want to encode that the fusion rules and associators are ``continuous'' with respect to this topology. The first mathematical approach to categories of this type appeared, as a tool to study Chern-Simons theory, in \cite{FHLT}. Freed, Hopkins, Lurie, and Teleman introduced the categorified group ring $\Vec^\omega T$ of vector spaces graded over a torus $T$, where the associator is twisted by a cohomology class $\omega\in H^4(BT; \mathbb{Z})$. The category $\Vec^\omega T$ is defined as the category of skyscraper sheaves on $T$ with finite support.

The most exhaustive study of tensor categories with a smooth family of simple objects appears in \cite{CW} under the name of manifold tensor categories. The proposed model builds on the skyscraper sheaf approach to define manifold tensor categories as stacks of skyscraper sheaves on the site of smooth manifolds, which are further allowed to be twisted by a gerbe on the manifold of simple objects. Whilst the categories assigned to the point manifold are the categories of \cite{FHLT}, the rest of the stack allows one to keep track of smooth families of finitely many simple objects. Weis also defines smooth analogues of categorical constructions such as algebra objects or the Drinfeld centre, and discusses rigidity in this context. Under this definition, Weis gives a structure of a manifold tensor category to the category $\Vec^\omega G$ for $G$ a Lie group and $\omega\in H^3_{SM}(BG; \mathbb{C}^\times)$ a class in Segal-Mitchison cohomology, as well as to some of the categories appearing in \cite{categoriosities}, and other examples.

The skyscraper sheaf model of \cite{FHLT} and \cite{CW}, however, does not allow for the limit of simple objects to be non-simple (which happens, for example, in the representation theory of dimension drop algebras \cite{MR1680321}), or for direct integrals of objects, as only finite support sheaves are considered. Hence, one cannot discuss the examples of \cite{fredenhagen, Ponsot01, schrader2017continuoustensorcategoriesquantum} where the tensor product of two simple objects yields a direct integral of simple objects. On the other hand, the treatment in \cite{fredenhagen, Ponsot01, schrader2017continuoustensorcategoriesquantum} does not provide a model in which the natural topology on the space of simple objects can be exploited. The present paper introduces a new framework for ``semisimple'' tensor categories with continuously many objects which naturally incorporates the topology on the space of simple objects and also allows for direct integrals and non-simple limits of simple objects.

\subsection*{Continuous tensor categories}
Given a $\mathrm{C}^*$-algebra $A$, its category of representations $\Rep(A)$ has a natural topology, the Fell topology, and one can take direct integrals of representations of $A$. If one only has access to $\Rep(A)$ as opposed to $A$, these two structures are no longer~well defined. We may therefore think of $A$ as a continuous category topologizing its underlying linear category $\Rep(A)$. Given another $\mathrm{C}^*$-algebra $B$, the functors $\Rep(A)\to \Rep(B)$ which are compatible with the topology are those given by tensoring with a $B-A$-correspondence.

\begin{repdefinition}{def: ctstensorcats}
    A continuous tensor category is a (coherently associative) unitary algebra object in the bicategory $\MorC$ of $\mathrm{C}^*$-algebras, $\mathrm{C}^*$-correspondences and intertwiners.
\end{repdefinition}

A continuous tensor category is therefore a $\mathrm{C}^*$-algebra $A$, together with a $\mathrm{C}^*$-correspondence $\Tcal$ between $A$ and $A\otimes A$ encoding the tensor product, and a unitary intertwiner encoding the associator, plus a unit and unitors. The $\mathrm{C}^*$-correspondence $\Tcal$ of a continuous tensor category induces a tensor product on the underlying linear category $\Rep(A)$ of $A$. This assignment extends to a functor $\mathfrak{F}$ from the bicategory of continuous tensor categories into the bicategory of linear tensor categories, which we do not know to be full. Actually, the underlying linear tensor category of a continuous tensor category is a $\mathrm{W}^*$-tensor category, meaning, in particular, that its endomorphism algebras are von Neumann algebras. Let $G$ be a locally compact group. 

\begin{repexample}{ex: HilbG}{Ex: GroupC*Algebra}
    \begin{enumerate}
        \item The tensor category $\Hilb\,G$ of $G$-graded Hilbert spaces is naturally a continuous tensor category. Its underlying continuous category is the $\mathrm{C}^*$-algebra $C_0(G)$ of $\CC$-valued continuous functions on $G$ vanishing at infinity. The tensor $C_0(G)-C_0(G\times G)$-correspondence is obtained by letting $C_0(G)$ act on $C_0(G\times G)$ by pointwise multiplication after pulling back along the operation $m:G\times G\to G$. The simple objects of $\Hilb\,G$ are parametrized by $G$, and they fuse according to $m$.
        \item Given $\omega\in H^3_{SM}(BG;U(1))$ a class in Segal-Mitchison cohomology, we obtain a continuous tensor category $\Hilb^\omega G$ by twisting the tensor product and the associator above~by~$\omega$.
\item The tensor category $\Rep(G)$ of unitary representations of $G$ is naturally a continuous tensor category. Its underlying continuous category is the group $\mathrm{C}^*$-algebra $C^*(G)$, which is naturally a Hopf $\mathrm{C}^*$-algebra. The Hopf structure provides the tensorator.
    \end{enumerate}
\end{repexample}
 
 \subsection*{Continuous Tambara-Yamagami categories} A more involved example of continuous tensor categories are continuous Tambara-Yamagami tensor categories, which we introduce next.

 Before discussing the continuous version, let us recall the definition of finite Tambara-Yamagami fusion categories. In \cite{TY}, Tambara and Yamagami studied those fusion categories which have a unique non-invertible simple object $\tau$ and $\tau\otimes\tau$ is a direct sum of invertible simple objects. These categories are equivalent to $\Vec\,G\oplus \Vec\cdot \tau$ for $G$ a finite abelian group (whose group operation we write additively), with tensor product given by, taking $g,h\in G$,
\begin{equation}\label{eq: TYfusionrules}
g\otimes h \cong g+h \hspace{1cm} g\otimes\tau \cong\tau\otimes g \cong \tau\hspace{1cm} \tau\otimes\tau \cong \bigoplus_{g\in G}g.
\end{equation}
Here, $\Vec\cdot \tau$ is the category isomorphic to $\Vec$ generated by the symbol $\tau$. Fusion categories with the fusion rules \eqref{eq: TYfusionrules} are classified by a symmetric nondegenerate bicharacter $\chi:G\times G\to \CC^\times$ controlling the $\Vec\,G - \Vec\,G$-bimodule structure of $\Vec\cdot \tau$ and a choice of a square root $\xi\in \{\pm\nicefrac{1}{\sqrt{|G|}}\}$ controlling the $(\tau,\tau,\tau)$-associator. The fusion category associated to such a pair $(\chi, \xi)$ is denoted $\TYcal(G, \chi,\xi).$

\begin{theorem*}[\cite{TY}]
    Let $\Ccal$ be a fusion category with a unique non-invertible simple object $\tau$ whose square is a direct sum of all invertible simple objects. Then, $\Ccal$ is equivalent to $\TYcal(G, \chi, \xi)$ for a finite abelian group $G$, a unique symmetric nondegenerate bicharacter $\chi:G\times G\to \CC^\times$ (up to the action of $\text{Aut}(G)$) and a unique choice of a square root $\xi\in \{\pm\nicefrac{1}{\sqrt{|G|}}\}$.
\end{theorem*}

We exploit our definition of continuous tensor categories to generalize the construction of Tambara-Yamagami tensor categories from finite groups to topological groups, and we classify continuous Tambara-Yamagami tensor categories. The reasons to treat these continuous tensor categories are multiple. If we want to understand continuous tensor categories, it is reasonable to start by understanding simple examples, as one does in the fusion setting. In addition, Tambara-Yamagami categories are some of the first examples in which one needs access to sheaves of infinite support on the space of simple objects (as $\tau\otimes\tau$ will be a direct integral over the $G$-indexed simple invertible objects), and hence they are a setting in which we can exploit the benefits of our new model over previous definitions in terms of skyscraper sheaves. Explicit interest in continuous generalizations of Tambara-Yamagami tensor categories has recently appeared in the literature in \cite[Problem 7]{galindob} and \cite{galindo24} when discussing $\mathbb{Z}/2$-crossed-braided graded extensions of $\Vec\,G$ with a braiding, for $G$ finite, via condensation from the conjectural Tambara-Yamagami category for $\mathbb{R}^d$. The role of continuous Tambara-Yamagami tensor categories in CFT is outlined at the end of this introduction.

We next introduce the definition of continuous Tambara-Yamagami tensor categories. We provide a $\mathrm{C}^*$-algebra giving the underlying continuous category and a correspondence witnessing the tensor product. Let $G$ be a locally compact abelian group and define $A:=C_0(G)\oplus \CC$. The correspondence encoding the tensor product on $A$ is given as follows. Note that $A\otimes A\cong C_0(G\times G)\oplus C_0(G)\oplus C_0(G)\oplus\CC$ and we can construct the $A-A\otimes A$-correspondence
 \[
 \TYcal_G:= C_0(G\times G)\oplus C_0(G)\oplus C_0(G)\oplus L^2(G),
 \]
 where each of the summands of $A\otimes A$ acts on the right on the corresponding summand of $\TYcal_G$ by multiplication. The left $A$-action on $\TYcal_G$ is given by declaring that $C_0(G)$ acts on $C_0(G\times G)$ by pullback along the group operation $m:G\times G\to G$, on $L^2(G)$ by pointwise multiplication, and by zero on the two copies of $C_0(G)$. On the other hand, the summand $\CC$ of $A$ acts on the two copies of $C_0(G)$ by multiplication and by zero on the leftmost and rightmost summands of $\TYcal_G$. The $A-A\otimes A$-correspondence $\TYcal_G$ mimics the Tambara-Yamagami tensor structure. Indeed, the continuous tensor subcategory on $C_0(G)$ becomes $\Hilb\,G$, as described above, and the summand $\Rep(\CC) = \Hilb$ is acted upon on the left and on the right by $\Rep(C_0(G))$ by tensoring with the underlying Hilbert space of a representation. Denoting by $\tau$~the~canonical simple object of the summand $\Rep(\CC) = \Hilb$, it holds that $\tau\otimes \tau = L^2(G)\in \Rep(C_0(G))$.
 \begin{repdefinition}{def: ctsTY}
    A continuous Tambara-Yamagami tensor category for $G$ is a continuous tensor category whose underlying $\mathrm{C}^*$-algebra is $C_0(G)\oplus \CC$ and whose correspondence encoding the tensor product is $\TYcal_G$.     
 \end{repdefinition}

 We classify continuous Tambara-Yamagami tensor categories for $G$ as follows. Let $\chi:G\times G\to U(1)$ be a continuous symmetric bicharacter which is nondegenerate in the sense that it implements an isomorphism $x\mapsto \chi(x,-)$ from $G$ to its Pontryagin dual. In addition, pick $\xi\in \{\pm1\}.$ This data provides an associator for the tuple $(C_0(G)\oplus\CC, \TYcal_G)$ in an analogous way to the finite setting, and hence defines a continuous Tambara-Yamagami tensor category $\TYcal(G,\chi,\xi)$. The bicharacter $\chi$ encodes the $\Hilb\,G-\Hilb\,G$-bimodule structure on $\Hilb\cdot \tau$ and the associator $\alpha_{\tau,\tau,\tau}$ is given by the Fourier transform on $G$ induced by $\chi$, with sign $\xi$.

 Note that the group $\text{Aut}(G)$ of continuous automorphisms of $G$ acts on the set of continuous symmetric nondegenerate bicharacters on $G$ by pullback.
 \begin{reptheorem}{thm: finalTY2}
   Let $G$ be a locally compact abelian group. There is a bijection 
\[\hspace{-.5cm}
\nicefrac{\left\{\begin{array}{l} (\chi, \xi)\ |\ 
    \text{$\chi: G\times G\to U(1)$ a}\\
    \text{continuous symmetric} \\\text{nondegenerate bicharacter}\\\text{ and $\xi\in\{\pm1\}$}
  \end{array}\right\}}{\text{Aut}(G)}\xrightarrow{\TYcal(G,-,-)}\nicefrac{\left\{\begin{array}{l}
    \text{continuous Tambara-Yamagami}\\
    \text{tensor categories for $G$}
  \end{array}\right\}}{\text{continuous $\otimes$ equiv.}} 
\]
\end{reptheorem}
 
 As a consequence of the proof of Theorem \ref{thm: finalTY2}, we obtain the following characterization of linear tensor categories which are the image under $\mathfrak{F}$ of a continuous Tambara-Yamagami tensor category. We say that a Tambara-Yamagami $\mathrm{W}^*$-tensor category is a $\mathrm{W}^*$-tensor (not necessarily fusion or semisimple) category whose underlying category is $\Rep(C_0(G))\oplus\Hilb\cdot \tau$ and such that the tensor product makes the first summand $\Hilb\,G$, the first summand acts on $\Hilb\cdot \tau$ on the left and on the right by the forgetful functor $\Rep(C_0(G))\to\Hilb$, and $\Hom(\tau\otimes\tau, \tau) = 0$. Every continuous Tambara-Yamagami tensor category induces a Tambara-Yamagami $\mathrm{W}^*$-tensor category via the functor $\mathfrak{F}$. We show the following result, also as part of Theorem \ref{thm: finalTY2}.
 \begin{theorem*}
     Every Tambara-Yamagami $\mathrm{W}^*$-tensor category for $G$ is equivalent to $\mathfrak{F}(\mathcal{TY}(G,\chi, \xi))$ for a continuous symmetric nondegenerate bicharacter $\chi: G\times G\to U(1)$ and a sign $\xi\in\{\pm 1\}$. In addition, $\mathfrak{F}(\mathcal{TY}(G,\chi, \xi))$ and $\mathfrak{F}(\mathcal{TY}(G,\chi', \xi'))$ are tensor equivalent if and only if $\xi = \xi'$ and there exists a continuous group automorphism $\phi: G\to G$ such that $\chi' = \chi\circ(\phi\times\phi)$.
 \end{theorem*}

This result can be thought of as an automatic continuity of Tambara-Yamagami $\mathrm{W}^*$-tensor categories. That is, the Tambara-Yamagami fusion rules imply the continuity of the associators. In addition, Tambara-Yamagami $\mathrm{W}^*$-tensor categories are equivalent if and only if they are equivalent as continuous tensor categories.

\subsection*{Continuous tensor categories in CFT}
As we have mentioned before, one of our main motivations to introduce continuous tensor categories is their potential application to conformal field theory. One of the most developed approaches to 2-dimensional unitary chiral CFT is the theory of conformal nets. Given a conformal net, its category of representations has a topological space of (simple) objects and it admits direct integrals. We believe, as suggested by computations in \cite{bmt88, fredenhagen, MR2031030, MR3627412}, that categories of representations of conformal nets are canonically continuous tensor categories. In our work \cite{RepHeis}, we show that the categories of representations of the Heisenberg conformal net (which represents the free boson on a line) and its $\Z/2$-fixed points are continuous tensor categories. Actually, the latter is a $\Z/2$-equivariantization of a continuous Tambara-Yamagami tensor category for $G = \R$. The tools developed in the current paper are key in the characterization of these tensor~categories.

 \section*{Acknowledgements}
I am grateful to André Henriques and Christopher Douglas for their help in this article, and to Thomas Wasserman for our continuing conversations about this and related projects, from which this paper has greatly benefited. I thank Lucas Hataishi for his help with operator algebraic arguments. I am also grateful to Nivedita and Sofía Marlasca Aparicio for discussions on various aspects of this work. I am thankful to Miquel Saucedo Cuesta for his help in the proofs of Lemmas \ref{lemm: SolutionToEquationL2} and \ref{Lemm: Function P}. This work has been funded by the EPSRC grant EP/W524311/1.

For the purpose of Open Access, the author has applied a CC BY public copyright license to any Author Accepted Manuscript version arising from this submission.

\addtocontents{toc}{\protect\setcounter{tocdepth}{5}}
\section{Preliminaries}
\subsection{Tambara-Yamagami tensor categories}

Let $\Ccal$ be a $\CC$-linear category. Recall that a \emph{tensor structure} on $\Ccal$ consists of a bifunctor $\otimes: \Ccal \times \Ccal \to \Ccal$ together with a unit object $\1\in \Ccal$ and natural associators and unitors satisfying the pentagon and triangle identities. All the data is required to be compatible with the linear structure on $\Ccal$. We refer the reader to \cite{EGNO} for more details. A tensor category $\Ccal$ is called \emph{rigid} if every object $X\in \Ccal$ admits both a left and a right dual, that is, objects $X^*$ and $\prescript{*}{}{X}$ together with morphisms $\text{ev}_X:X^*\otimes X\to \1$, $\text{coev}_X: \1\to X\otimes X^*$ and $\text{ev}_{\prescript{*}{}{X}}: X\otimes\prescript{*}{}{X}\to \1$, $\text{coev}_{\prescript{*}{}{X}}: \1\to \prescript{*}{}{X}\otimes X$ satisfying the snake relations. An object $X$ is called \emph{invertible} if $\text{ev}_X$ and $\text{coev}_X$ are isomorphisms. A rigid tensor category $\Ccal$ is \emph{fusion} if $\Ccal$ has finitely many simple objects, including $\1$, all the $\Hom$ vector spaces are finite-dimensional, and $\Ccal$ is semisimple in the sense that every object is isomorphic to a finite direct sum of simples. 

In \cite{TY}, Tambara and Yamagami classified all fusion categories with exactly one non-invertible simple object $\tau$ up to isomorphism, which further satisfies that $\tau\otimes\tau$ decomposes as a finite direct sum of invertible simple objects. The construction is given as follows. Let $G$ be a finite abelian group (whose group operation we write additively) and $\chi:~G\times~G\to~\CC^\times$ a symmetric nondegenerate bicharacter on $G$. Furthermore, let $\xi$ be a choice of a square root of $\frac{1}{|G|}$. Then, the fusion category $\TYcal(G, \chi,\xi)$ is defined by:
\begin{enumerate}
	\item objects are finite direct sums of elements of $G\sqcup\{\tau\}$,
	\item tensor products of simple objects are given by, for $x,y\in G$,
	\[
	x\otimes y = x+y,\hspace{1cm} x\otimes \tau = \tau,\hspace{1cm}  \tau\otimes x = \tau,\hspace{1cm}  \tau\otimes \tau  =\bigoplus_{x\in G}x,
	\]
	\item associators for simple objects are given by, for $x,y,z\in G$,
\begin{align*}
		\alpha_{x,y,z} &= \id_{x+y+z}\\
		\alpha_{x,y,\tau} = \alpha_{\tau,x,y} &= \id_\tau\\
		\alpha_{x,\tau,y} &= \chi(x,y)\id_\tau\\
		\alpha_{x,\tau,\tau} = \alpha_{\tau,\tau,x} &= \bigoplus_{y\in G} \id_y\\
		\alpha_{\tau,x,\tau} & = \bigoplus_{y\in G} \chi(x,y)\id_y\\
		\alpha_{\tau,\tau,\tau}& = (\xi \cdot \chi(x,y)^{-1}\cdot \id_\tau)_{x,y}:\bigoplus_{x\in G} \tau\to \bigoplus_{x\in G} \tau
	\end{align*}
	\item the unit object is the unit $e\in G$ and unitors are identities.
\end{enumerate}
\begin{theorem}[\cite{TY}]
	Every fusion category with one non-invertible simple object (up to isomorphism) $\tau$ and such that $\tau\otimes\tau\neq0$ and $\Hom(\tau\otimes\tau, \tau) = 0$ is equivalent to $\TYcal(G, \chi,\xi)$ for some finite abelian group $G$, some symmetric nondegenerate bicharacter $\chi$ on $G$ and a choice $\xi$ of a square root of $\nicefrac{1}{|G|}$. In addition, $\TYcal(G, \chi,\xi)$ and $\TYcal(G', \chi',\xi')$ are equivalent if and only if $\xi = \xi'$ and $\chi(x,y) = \chi'\big(\phi(x), \phi(y)\big)$ for some isomorphism $\phi: G\to G'$.
\end{theorem}

\subsection{Fourier analysis on locally compact abelian groups}

For the rest of this paper, all topological spaces are assumed to be locally compact, Hausdorff, and second countable. In the finite case, as discussed in \cite{TY} and in the previous section, every finite abelian group $G$ admits a Tambara-Yamagami category of the form $\text{Vec}G\oplus \text{Vec}\cdot \tau$. In the continuous setting, this will no longer be the case, the reason being that not all locally compact abelian groups are self-Pontryagin dual.

\begin{definition}
	Let $G$ be a locally compact abelian group. The \emph{Pontryagin dual} of $G$~is~the~group
	\[
	\hat{G} : = \Hom\big(G, U(1)\big)
	\]
	of continuous group homomorphisms from $G$ into $U(1)$, equipped with pointwise multiplication and the compact-open topology. We say $G$ is \emph{self-Pontryagin dual} if there is an isomorphism of topological groups
	\(
	G\cong \hat{G}.
	\)
\end{definition}

\begin{example} 
	The following groups are self-Pontryagin dual:
	\begin{enumerate}
		\item finite abelian groups,
		\item the additive group $\mathbb{R}^n$,
		\item the product group $\mathbb{Z}\times U(1).$ 
	\end{enumerate}
\end{example}

Recall that, on every locally compact group $G$, there exists a non-negative regular measure $\mu$, the Haar measure on $G$, which is not identically zero and which is (left) translation invariant. When considering $L^p$-functions on a group, we will always be referring to the Haar measure, unless specified otherwise. We write $L^2(G)$ and $L^2(\mu)$ interchangeably.

\begin{definition}
	Let $G$ be a locally compact abelian group and $f\in L^1(G)$. The function on~$\hat{G}$
	\[
	\Fcal(f): \eta\mapsto \int_Gf(x)\eta(-x)d\mu(x)
	\]
	is called the \emph{Fourier transform} of $f$.
\end{definition}

The Fourier transform can be extended to $L^2(G)$, see \cite[Thm. 1.6.1]{rudin}.

\begin{theorem}[Plancherel Theorem]
	Let $G$ be a locally compact abelian group. There is a Haar measure $\hat{\mu}$ on $\hat{G}$ for which the Fourier transform can be extended uniquely to a unitary $$\Fcal:L^2(G, \mu)\cong L^2(\hat{G}, \hat{\mu{}}).$$
\end{theorem}

Every time we consider $L^2$-functions on $\hat{G}$, the underlying measure is the dual Haar measure $\hat{\mu}$, unless stated otherwise. We will need the following properties of the Fourier transform, which can be found for $L^1$-functions in (the proof of) \cite[Thm. 1.2.4]{rudin}, and can be extended to $L^2(G)$.

\begin{proposition} \label{prop: PropertiesF}
	Let $f\in L^2(G)$, $x_0\in G$ and $\eta_0\in \hat{G}$.
	\begin{enumerate}
		\item If $g(x) = \eta_0(x)f(x)$, then $\Fcal(g)(\eta) = \Fcal(f)(\eta-\eta_0)$, \label{prop: PropertiesF1}
		\item If $g(x) = f(x-x_0)$, then $\Fcal(g)(\eta) = \eta(-x_0)\Fcal(f)(\eta)$. \label{prop: PropertiesF2}
	\end{enumerate}
\end{proposition}

\section{Continuous tensor categories}
\subsection[The Morita bicategory of C*-algebras and correspondences]{The Morita bicategory of $\mathrm{C}^*$-algebras and correspondences}
Our definition of continuous tensor categories will strongly rely on the theory of $\mathrm{C}^*$-algebras. We give a short introduction to the topic here and direct the reader to \cite{OAbook} for more details. Recall that a \emph{$\mathrm{C}^*$-algebra} is a Banach algebra which is further equipped with an involution $-^*:A\to A$ such that $||a^*a|| = ||a||^2$. An element $a\in A$ is called \emph{positive} if $a = b^*b$ for some $b\in A$. We use the terminology \emph{$\mathrm{C}^*$-algebra} to refer to separable $\mathrm{C}^*$-algebras, meaning that their underlying Banach algebra is separable. All Hilbert spaces are also assumed to be separable and all measure spaces, standard measure spaces.

A \emph{representation} of a $\mathrm{C}^*$-algebra $A$ is a $*$-homomorphism $\pi: A\to B(H)$, for a Hilbert space $H$. We say that $\pi$ is \emph{nondegenerate} if $\pi(A)H$ is dense in $H$. From now on, we will only consider nondegenerate representations of $\mathrm{C}^*$-algebras. Given a $\mathrm{C}^*$-algebra $A$, we denote its category of nondegenerate representations by $\Rep(A).$ The set of isomorphism classes of irreducible representations of $A$ is called its \emph{spectrum}. We will call these categories ``semisimple'', as every short exact sequence splits. Indeed, let $0\to H\xrightarrow{\iota} K\xrightarrow{\pi} R\to 0$ be a sequence of maps in $\Rep(A)$ such that $\iota$ is injective, $\pi$ is surjective and $\text{Im}(\iota) =\ker(\pi).$ Then, the orthogonal complement $\iota(H)^\perp\subset K$ is an $A$-representation which is isomorphic to $R$. Hence, $K = \iota(H)\oplus\iota(H)^\perp\cong H\oplus R$.

The bicategory of continuous semisimple categories will be a Morita bicategory of $\mathrm{C}^*$-algebras. Hence, we introduce the appropriate notion of $\mathrm{C}^*$-algebra bimodules, see \cite[II.7.1]{OAbook}.

\begin{definition}
    Let $A$ be a $\mathrm{C}^*$-algebra and $\Ecal$ an algebraic right $A$-module. An $A$\emph{-valued pre-inner product} on $\Ecal$ is a function $\langle - , - \rangle : \Ecal\times \Ecal\to A$ such that, for all $\xi,\eta, \zeta\in \Ecal$ and $a\in A$, $c\in \CC$,
    \begin{enumerate}
        \item $\langle \xi, c\eta+\zeta\rangle = c\langle \xi, \eta\rangle + \langle \xi, \zeta\rangle$,
        \item $\langle \xi, \eta a\rangle = \langle \xi, \eta\rangle a$,
        \item $\langle \eta, \xi\rangle = \langle \xi, \eta\rangle^*$,
        \item $\langle \xi, \xi\rangle$ is positive as an element of $A$.
    \end{enumerate}
    A \emph{(right) Hilbert $A$-module} is a right $A$-module $\Ecal$ with a pre-inner product $\langle-,-\rangle$ such that $\langle\xi, \xi\rangle = 0\implies \xi = 0$ for all $\xi\in\mathcal{E}$ and $(\Ecal, \lvert\lvert\langle-,-\rangle\rvert\rvert^{1/2})$ is complete.
\end{definition}

Given two Hilbert $A$-modules $\Ecal$ and $\Dcal$, we say that a $\CC$-linear operator $u:\Ecal\to \Dcal$ is \emph{adjointable} if there exists an adjoint $u^*:\Dcal \to \Ecal$ with respect to the $A$-valued inner product in the sense that $\langle u- ,- \rangle = \langle -,u^*-\rangle$. We write $\Lcal_A(\Ecal, \Dcal)$ for the space of adjointable operators from $\Ecal$ to $\Dcal$ and $\Lcal_A(\Ecal):=\Lcal_A(\Ecal, \Ecal)$. Note that an adjointable operator is automatically $A$-linear and bounded. Given two $\mathrm{C}^*$-algebras $B$ and $A$, a $B-A$-\emph{correspondence} is a Hilbert $A$-module $\Ecal$ together with a nondegenerate homomorphism $\phi: B\to \mathcal{L}_A(\Ecal)$ from $B$ into the adjointable endomorphisms of $\Ecal$. We will denote a correspondence as such by $(\Ecal, \phi)$ or simply by $\Ecal$. If $\Dcal$ is another $B-A$-correspondence, an \emph{intertwiner} from $\Ecal$ to $\Dcal$ is a morphism $u\in \mathcal{L}_A(\Ecal, \Dcal)$ such that the following diagram commutes for all $b\in B$,
    % https://q.uiver.app/#q=WzAsNCxbMCwwLCJcXEVjYWwiXSxbMSwwLCJcXEVjYWwiXSxbMCwxLCJcXERjYWwiXSxbMSwxLCJcXERjYWwiXSxbMCwxLCJcXHBoaShhKSJdLFsyLDMsIlxccHNpKGEpIiwyXSxbMCwyLCJ1IiwyXSxbMSwzLCJ1Il1d
\[\begin{tikzcd}
	\Ecal & \Ecal \\
	\Dcal & \Dcal.
	\arrow["{\phi(b)}", from=1-1, to=1-2]
	\arrow["u"', from=1-1, to=2-1]
	\arrow["u", from=1-2, to=2-2]
	\arrow["{\psi(b)}"', from=2-1, to=2-2]
\end{tikzcd}\]
We denote by $\Corr(B, A)$ the category whose objects are $B-A$-correspondences and whose morphisms are intertwiners.

Correspondences can be composed as follows. Given three $\mathrm{C}^*$-algebras $A,B$ and $C$, together with an $A-B$-correspondence $(\Ecal, \phi)$ and a $B-C$-correspondence $(\Dcal, \psi)$, we can obtain an $A-C$-correspondence $(\Ecal,\phi)\otimes_B(\Dcal, \psi)$ denoted
\[
(\Ecal\otimes_\psi\Dcal, \phi\otimes \text{I}).
\]
The Hilbert $C$-module $\Ecal\otimes_\psi\Dcal$ is constructed from the quotient of the algebraic tensor product $\Ecal\odot \Dcal$ by the subspace spanned by $\{\xi b\otimes \eta-\xi\otimes \psi(b)\eta\ |\ \xi\in\Ecal,\ \eta\in \Dcal,\ b\in B\}$ by further completing it (after quotienting by the null vectors) with respect to the norm induced by the pre-inner product $\langle \xi_1\otimes \eta_1, \xi_2\otimes \eta_2\rangle := \langle \eta_1,\psi(\langle\xi_1,\xi_2\rangle)\eta_2\rangle$, see \cite[II.7.4.4]{OAbook}. This procedure is known as the \emph{Rieffel tensor product of $\Ecal$ and $\Dcal$}. Note that, given a $\mathrm{C}^*$-algebra $A$, there is a canonical $A-A$-correspondence given by seeing $A$ as a Hilbert module over itself with the pre-inner product $\langle a,a'\rangle:= a^*a'$, and where the left and right actions are given by left and right multiplication respectively. The pre-inner product is an inner product by the $\mathrm{C}^*$-identity $||a^*a|| = ||a||^2$. We denote this correspondence by $A$ and note that, for any other $B-A$-correspondence $\Ecal$ and $A-B$-correspondence $\Dcal$, we have
\[
\Ecal\otimes_A A \cong \Ecal\hspace{1cm} A\otimes_A \Dcal \cong \Dcal.
\]
More generally, any nondegenerate $*$-homomorphism $B\to A$ produces the structure of a correspondence\footnote{Even more generally, any nondegenerate $*$-homomorphism $\phi: B\to M(A)$ into the multiplier algebra of $A$ defines one such correspondence.} on the trivial $A$-Hilbert module $A$. 

Given a third $\mathrm{C}^*$-algebra $C$, the Rieffel tensor product provides a functor
\[
-\otimes_B -: \Corr(C, B)\times \Corr(B, A)\to \Corr(C, A).
\]
At the level of morphisms, if $u:(\Ecal, \phi)\to(\Ecal', \phi')$ is an intertwiner in $\Corr(B, A)$ and $v:(\Dcal, \psi)\to(\Dcal', \psi')$ is an intertwiner in $\Corr(C, B)$, we define 
\[
v\otimes_B u: (\Dcal\otimes_\psi\Ecal, \phi\otimes I)\to (\Dcal'\otimes_{\psi'}\Ecal', \phi'\otimes I)
\]
as the morphism induced by $v\odot u: \Dcal\odot \Ecal\to \Dcal'\odot\Ecal'$.

Equipped with the Rieffel tensor product as composition, $\mathrm{C}^*$-algebras, together with correspondences and intertwiners, form a bicategory \cite{landsman}.

\begin{definition}\label{def: 2Calg}
    We denote by $\MorC$ the bicategory whose objects are $\mathrm{C}^*$-algebras and whose category of morphisms between two $\mathrm{C}^*$-algebras $A$ and $B$ is $\Corr(B,A).$ The composition of 1- and 2-morphisms is given by the Rieffel tensor product and the identity 1-morphisms are the correspondences $A$ for $A$ a $\mathrm{C}^*$-algebra. 
\end{definition}

\begin{remark}
    The invertible 1-morphisms in $\MorC$ are known as \emph{imprimitivity bimodules} in the literature.
\end{remark}

The Morita bicategory $\MorC$ models the bicategory of ``semisimple'' continuous categories. Given a $\mathrm{C}^*$-algebra $A$, the set of simple objects of $\Rep(A)$ up to isomorphism has a canonical topology. Given another $\mathrm{C}^*$-algebra $B$, a continuous functor between $\Rep(A)$ and $\Rep(B)$ is encoded by a $B-A$-correspondence. In order to make this picture precise, we need to provide a way to, given $A\in \MorC$, recover its associated linear category, and given a $B-A$-correspondence, obtain a functor between $\Rep(A)$ and $\Rep(B)$. 

Recall that a \emph{von Neumann algebra} on a Hilbert space $H$ is a $*$-subalgebra $A$ of $B(H)$ such that $A'' = A$, where $A' = \{b\in B(H)\ |\ ab = ba \text{ for all $a\in A$}\}$, and $A'':=(A')'$. A $*$-algebra is a \emph{von Neumann} algebra if it admits a faithful $*$-action on a Hilbert space $H$ exhibiting it as a von Neumann algebra on $H$. A \emph{$*$-category} is a $\CC$-linear category equipped with a dagger structure $*:\Hom(X,Y)\to \Hom(Y,X)$ which is $\CC$-antilinear and satisfies $f^{**} = f$ and $(f\circ g)^* = g^*\circ f^*.$ Given a $*$-category $\Ccal$, we write $\Ccal^\oplus$ for the category whose objects are formal finite direct sums $\oplus_{i\in I} X_i$ and whose morphisms are $\Hom_{\Ccal^\oplus}(\oplus_{i\in I}X_i, \oplus_{j\in J}Y_j):=\oplus_{i\in I, j\in J}\Hom_{\Ccal}(X_i, Y_j).$
\begin{definition}
 A $\mathrm{W}^*$\emph{-category} is a $*$-category $\Ccal$ such that $\End_{\Ccal^\oplus}(X)$ is a von Neumann algebra for all $X\in \Ccal^{\oplus}.$
\end{definition}

Given two von Neumann algebras $A$ and $B$, a $*$-algebra homomorphism $f:A\to B$ is called \emph{normal} if $f(\sup a_i) = \sup f(a_i)$ for every bounded increasing net $\{a_i\}_i$ of positive elements of $A$. A \emph{functor of $\mathrm{W}^*$-categories} is a functor $F:\Ccal\to \Dcal$ of $\CC$-linear categories that preserves the involution and such that it induces normal homomorphisms $\End_{\Ccal^\oplus}(X)\to \End_{\Dcal^\oplus}(F(X))$ for all $X\in \Ccal^\oplus.$ A natural transformation $\alpha$ between $\mathrm{W}^*$-functors is called \emph{bounded} if $||\alpha||:=\sup_{X\in\Ccal}||\alpha_X||<\infty$, where the norm of a morphism $f:X\to Y$ in a $\mathrm{W}^*$-category is defined to be the norm of $f$ in $\End_{\Ccal^{\oplus}}(X\oplus Y)$. In a $\mathrm{W}^*$-category, there exists the notion of orthogonal direct sums, extending that of Hilbert spaces \cite[Def. 3.7]{henriques2024completewcategories}.

\begin{definition}
    We write $\Wcat$ for the bicategory of idempotent complete $\mathrm{W}^*$-categories  admitting all countable orthogonal direct sums, $\mathrm{W}^*$-functors and bounded natural transformations.
\end{definition}

All $\mathrm{W}^*$-categories appearing in this paper are assumed to admit all countable orthogonal direct sums. It is well known that, if $A$ is a $\mathrm{C}^*$-algebra, its category of representations $\Rep(A)$ is a $\mathrm{W}^*$-category. Indeed, if $H$ is an $A$-representation, then $\End_{\Rep(A)}(H)$ is the commutant of $A$ in $B(H)$ and is therefore a von Neumann algebra. The assignment $A\mapsto \Rep(A)$ can be extended to a functor $\Ffrak:\MorC\to \Wcat$.
Let $A,B$ be $\mathrm{C}^*$-algebras and $\Ecal$ a $B-A$-correspondence. We define the functor
\[
\Ffrak(\Ecal):\Rep(A)\to \Rep(B)
\]
as follows. A representation $H\in \Rep(A)$ is equivalently an $A-\CC$-correspondence, and we can take its Rieffel tensor product with $\Ecal$ to obtain a $B-\CC$-correspondence $\Ecal \otimes_A H$, that is, an element of $\Rep(B)$. Therefore, we can define the functor $\Ffrak(\Ecal):\Rep(A)\to \Rep(B)$ as the functor
\[
 (\Ecal, \phi)\otimes_A - : \Rep(A) = \Corr(\CC, A)\to \Corr(\CC, B) = \Rep(B).
\]
At the level of morphisms, $\Ffrak(\Ecal)$ sends a morphism $f: H\to K$ in $\Rep(A)$ to $\Ffrak(\Ecal)(f) : = \id_\mathcal{E}\otimes_A f: \mathcal{E}\otimes_A H\to \mathcal{E}\otimes_A K$. Hence, $\Ffrak(\Ecal)$ is a $\mathrm{W}^*$-functor. If $v$ is an intertwiner between $B-A$-correspondences $\Ecal$ and $\Dcal$, we define $\Ffrak(v)_H:\Ecal \otimes_A H\to \Dcal \otimes_A H$ as the morphism induced by $v\odot_A \id_H$.
\begin{proposition}\label{prop: functorF}
    The data above produces a functor of bicategories 
    \[
    \Ffrak: \MorC\to \Wcat.
    \]
    The functor is faithful at the level of $2$-morphisms.
\end{proposition}
\begin{proof}
The compatibility data of the Rieffel tensor product with the composition of functors in $\Wcat$ is given by the associator of the Rieffel tensor product. The unitor data is also given by the unitor data of the Rieffel tensor product. To show faithfulness at the level of 2-morphisms, let $A$ and $B$ be $\mathrm{C}^*$-algebras and $u, v:\Ecal\to \Dcal$ be 2-morphisms in $\Corr(B,A)$. Assume the corresponding natural transformations $\Ffrak(u)$ and $\Ffrak(v)$ are the same in $\Hom_{\Wcat}(\Rep(A), \Rep(B)).$ That is, for every $H\in \Rep(A)$ and every $e\in \Ecal$, we have the equality
\[
u(e)\otimes h = v(e)\otimes h
\]
in $\Dcal \otimes_A H$ for every $h\in H$. Since $A$ is separable, there exists a faithful nondegenerate representation $(H, \pi)$ of $A$. Let $w :=u-v$. Then, for every $e\in \mathcal{E}$ and $h\in H$, we have
\[
w(e)\otimes h = (w\otimes_A\id_H)(e\otimes h) = 0,
\]
implying that $0 = ||w(e)\otimes h||^2 = \langle h, \pi\big(\langle w(e), w(e)\rangle\big)(h)\rangle$. Therefore, $\pi\big(\langle w(e), w(e)\rangle\big) = 0$ and, since $\pi$ is faithful, $\langle w(e), w(e)\rangle = 0$. Therefore, $w = 0$.
\end{proof}

If we think of $\MorC$ as encoding continuous ``semisimple'' categories, Proposition \ref{prop: functorF} says that $\Ffrak$ is the forgetful functor that recovers the underlying $\mathrm{W}^*$-category. The image of
$$\Ffrak: \Corr(B,A)\to \Hom_{\Wcat}(\Rep(A), \Rep(B))$$ should then be thought of as those $\mathrm{W}^*$-functors which are continuous.

The bicategory $\MorC$ can be upgraded to a symmetric monoidal bicategory as follows. Given two $\mathrm{C}^*$-algebras $A$ and $B$, there is a poset of $\mathrm{C}^*$-norms on their algebraic tensor product $A\odot B$. This poset has a unique maximal element $||-||_{\text{max}}$ given by
    \[
    || \sum\limits_{i = 1}^n a_i\otimes b_i||_{\text{max}} := \sup ||\pi\big(
     \sum\limits_{i = 1}^n a_i\otimes b_i\big)||,
    \]
    where $\sup$ runs over all representations $\pi$ of $A\odot B$. The completion of $A\odot B$ with respect to this norm is denoted $A\otimes B$ and called the \emph{maximal tensor product of $A$ and $B$}. The maximal tensor product satisfies the following universal property: given $\phi: A\to C$ and $\psi: B\to C$ homomorphisms of $\mathrm{C}^*$-algebras with commuting images, there is a unique morphism $\rho: A\otimes B\to C$ such that $\rho(a\otimes b) = \phi(a)\psi(b)$ for all $a\in A$ and $b\in B$.

\begin{proposition}\label{prop: maximaltensor}
 The maximal tensor product can be extended to a functor of bicategories
\[
-\otimes - : \MorC\times\MorC\to \MorC.
\]   
Moreover, this defines a symmetric monoidal bicategory structure on $\MorC$.
\end{proposition}
\begin{proof}
This appears partially as \cite[Thm. 5.8]{Meyer}. The functor at the level of 1-morphisms is given by the external maximal tensor product of $\mathrm{C}^*$-correspondences. This is constructed as follows. Let $A, A'$ and $B, B'$ be $\mathrm{C}^*$-algebras and $(\Ecal, \phi)\in \Corr(A, A')$ and $(\Dcal, \psi)\in\Corr(B,B')$. The external tensor product of the Hilbert modules $\Ecal$ and $\Dcal$ is given by the completion of the right $A'\odot B'$-module $\Ecal\odot \Dcal$ with respect to the norm induced by the inner product
\[
\langle \xi_1\odot \eta_1, \xi_2\odot \eta_2 \rangle := \langle \xi_1,\xi_2\rangle_{A'}\odot \langle \eta_1, \eta_2\rangle_{B'},
\]
where we denote as a subscript where the inner product takes values. See \cite[Prop. 5.7]{Meyer} for the complete construction. We have a canonical map $A\to \Lcal_{A'\otimes B'}(\Ecal\otimes \Dcal)$ given by $a\mapsto (e\otimes d\mapsto \phi(a)e\otimes d)$, and similarly for $B\to \Lcal_{A'\otimes B'}(\Ecal\otimes \Dcal)$. These morphisms have commuting images and hence by the universal property of the maximal tensor product they induce a map $A\otimes B\to \Lcal_{A'\otimes B'}(\Ecal\otimes \Dcal)$. The procedure outlined above produces an $(A\otimes B)-(A'\otimes B')$-correspondence that we will denote
\[
(\Ecal, \phi)\otimes (\Dcal, \psi) = (\Ecal\otimes \Dcal, \phi\otimes \psi).
\]

Given two 2-morphisms $u:(\Ecal, \phi)\to (\Ecal', \phi')$ and $v:(\Dcal, \psi)\to (\Dcal', \psi')$, by the same procedure as above we obtain a morphism $u\otimes v: \Ecal\otimes \Dcal\to \Ecal'\otimes \Dcal'$ extending $u\odot v$, which is clearly a morphism of $\mathrm{C}^*$-correspondences.

Note that, in order to define $-\otimes -$ as a functor of bicategories, we need to further provide a 2-morphism witnessing the compatibility of the Rieffel tensor product with $-\otimes -$. That is, for $A,A',B,B',C,C'\in \MorC$ and 
\[
\Ecal\in \Corr(C,B),\hspace{1cm}\Ecal'\in\Corr(C',B')\hspace{1cm}\Dcal\in\Corr(B, A), \hspace{1cm}\Dcal'\in\Corr(B',A'),
\]
we need an intertwiner
\begin{equation}\label{eq: RieffelVSTensor}
(\Ecal\otimes \Ecal')\otimes_{B\otimes B'}(\Dcal\otimes \Dcal')\xrightarrow{\cong}(\Ecal\otimes_B\Dcal)\otimes (\Ecal'\otimes_{B'}\Dcal').
\end{equation}
 The flip map $\Ecal'\odot \Dcal\cong \Dcal\odot \Ecal'$ induces an isomorphism
 \(
 (\Ecal\odot \Ecal')\odot_{(B\odot B')}(\Dcal\odot \Dcal')\cong (\Ecal\odot_B\Dcal)\odot (\Ecal'\odot_{B'}\Dcal')
 \)
 which is easily seen to be continuous with respect to the norms on the left and the right-hand sides of \eqref{eq: RieffelVSTensor}. These morphisms are natural and hence produce the needed 2-cell. There is an analogous 2-cell witnessing compatibility on identities. For the associator, note that the maximal tensor product already endows the 1-category of $\mathrm{C}^*$-algebras and $\mathrm{C}^*$-algebra homomorphisms with a monoidal structure \cite[II.9.2.6]{OAbook}. Hence, one can write an associator for $-\otimes -$ on $\MorC$ which consists of correspondences coming from $\mathrm{C}^*$-algebra isomorphisms. Then, the pentagon equation is satisfied on the nose and one can choose the pentagonator to be the canonical reparenthesization in the Rieffel tensor product. The same holds for unitor data. The flip map $A\otimes B\cong B\otimes A$ induces a symmetric unitary braiding on $\MorC$.
\end{proof}

\subsection{Definition and first examples}

In this section, we define continuous tensor categories as algebra objects in $\MorC$. Given a monoidal bicategory $(\mathfrak{C}, \otimes)$ with unit $\mathbf{1}$, an \emph{algebra object in $\mathfrak{C}$} consists of an object $X\in \mathfrak{C}$ equipped with a multiplication 1-morphism $m: X\otimes X\to X$ and a unit morphism $u:\mathbf{1}\to X$ together with invertible 2-cells
% https://q.uiver.app/#q=WzAsOCxbMCwwLCJYXFxvdGltZXMgWFxcb3RpbWVzIFgiXSxbMiwwLCJYXFxvdGltZXMgWCJdLFsyLDEsIlgiXSxbMCwxLCJYXFxvdGltZXMgViJdLFszLDAsIlgiXSxbNSwwLCJYXFxvdGltZXMgWCJdLFs3LDAsIlgiXSxbNSwxLCJYIl0sWzAsMSwibVxcb3RpbWVzXFxpZF9YIl0sWzEsMiwibSJdLFswLDMsIlxcaWRfWFxcb3RpbWVzIG0iLDJdLFsxLDMsIlxcYWxwaGEiLDIseyJzaG9ydGVuIjp7InNvdXJjZSI6MjAsInRhcmdldCI6MzB9LCJsZXZlbCI6Mn1dLFszLDIsIm0iLDJdLFs0LDUsInVcXG90aW1lc1xcaWRfWCJdLFs2LDUsIlxcaWRfWFxcb3RpbWVzIHUiLDJdLFs1LDcsIm0iLDJdLFs0LDcsIlxcaWRfWCIsMl0sWzYsNywiXFxpZF9YIl0sWzUsMTYsIlxcbGFtYmRhIiwyLHsic2hvcnRlbiI6eyJ0YXJnZXQiOjIwfX1dLFs1LDE3LCJcXHJobyIsMCx7InNob3J0ZW4iOnsidGFyZ2V0IjoyMH19XV0=
\[\begin{tikzcd}
	{X\otimes X\otimes X} && {X\otimes X} & X && {X\otimes X} && X \\
	{X\otimes X} && X &&& X
	\arrow["{m\otimes\id_X}", from=1-1, to=1-3]
	\arrow["{\id_X\otimes m}"', from=1-1, to=2-1]
	\arrow["\alpha"', shorten <=11pt, shorten >=16pt, Rightarrow, from=1-3, to=2-1]
	\arrow["m", from=1-3, to=2-3]
	\arrow["{u\otimes\id_X}", from=1-4, to=1-6]
	\arrow[""{name=0, anchor=center, inner sep=0}, "{\id_X}"', from=1-4, to=2-6]
	\arrow["m"', from=1-6, to=2-6]
	\arrow["{\id_X\otimes u}"', from=1-8, to=1-6]
	\arrow[""{name=1, anchor=center, inner sep=0}, "{\id_X}", from=1-8, to=2-6]
	\arrow["m"', from=2-1, to=2-3]
	\arrow["\lambda"', shorten >=4pt, Rightarrow, from=1-6, to=0]
	\arrow["\rho", shorten >=4pt, Rightarrow, from=1-6, to=1]
\end{tikzcd}\]
called the associator and the left and right unitors respectively. The 2-cells are required to satisfy their own coherence conditions known as the pentagon and triangle diagrams, see \cite[Sec. 3]{Pseudomonoids}. A morphism between algebra objects $(X,m,u,\alpha,\lambda, \rho)$ and $(X',m',u',\alpha',\lambda', \rho')$ is a morphism $f:X\to X'$ in $\mathfrak{C}$ together with invertible 2-cells $s: m'\circ(f\otimes f)\xrightarrow{\cong}f\circ m$ and $u'\xrightarrow{\cong} f\circ u$ which are compatible with the associators and with left and right unitors. A 2-morphism between $(f,s)$ and $(f',s')$ is a 2-morphism $\eta:f\to f'$ compatible with $s$ and $s'$. For $\mathfrak{C} = \MorC$, we require algebra objects to have unitary associators and unitary unitors, and call these unitary algebra objects. Algebra morphisms and 2-morphisms between unitary algebra objects are also required to have unitary structure data.

\begin{definition}\label{def: ctstensorcats}
    The \emph{bicategory of continuous ``semisimple'' tensor categories} is the bicategory of (unitary) algebra objects in $(\MorC, \otimes)$. We call 1- and 2-morphisms between continuous tensor categories \emph{continuous tensor functors} and \emph{continuous monoidal natural transformations} respectively.
\end{definition}

Let us make explicit the data of an algebra object in $\MorC$. A continuous tensor category consists of
\begin{enumerate}
    \item A $\mathrm{C}^*$-algebra $A\in \MorC$,
    \item An $A-A\otimes A$-correspondence $(\Tcal, \tau)$ that encodes the tensor product,
    \item An $A-\CC$-correspondence $\1$, that is, an $A$-representation $\1$ that encodes the unit,
    \item A unitary intertwiner $\alpha: (\Tcal, \tau)\otimes_{A\otimes A}\big((\Tcal, \tau)\otimes A\big)\xrightarrow{\cong} (\Tcal, \tau)\otimes_{A\otimes A}\big(A\otimes (\Tcal, \tau)\big)$ that encodes the associator,
    \item Left and right unitors given by unitary intertwiners $\lambda: (\Tcal, \tau)\otimes_{A\otimes A}(A\otimes \1)\xrightarrow{\cong} A$ and $\rho: (\Tcal, \tau)\otimes_{A\otimes A}(\1\otimes A)\xrightarrow{\cong} A$.
\end{enumerate}
All this data satisfies the pentagon and triangle identities. The underlying linear category of a continuous tensor category is naturally a $\mathrm{W}^*$-tensor category, defined as follows.

\begin{definition}
    A $\mathrm{W}^*$\emph{-tensor category} is a $\mathrm{W}^*$-category $T$ with a monoidal structure whose tensor functor
    \[
    \otimes:T\times T\to T
    \]
    is a bilinear functor of $\mathrm{W}^*$-categories and whose associators and unitors are unitary. 
\end{definition}
A tensor functor between two $\mathrm{W}^*$-tensor categories $T$ and $S$ is a $\mathrm{W}^*$-functor $F:T\to S$ together with unitary coherences between the tensor product on $T$ and on $S$ and also between the image of the unit in $T$ and the unit in $S$, see \cite{henriques2024completewcategories}. A monoidal natural transformation between two $\mathrm{W}^*$-tensor functors is a natural transformation intertwining the coherence data.

Given a continuous semisimple tensor category, we can produce the underlying $\mathrm{W}^*$-tensor category via the forgetful functor~$\Ffrak$.

\begin{proposition}\label{Prop: ForgetE1}
    The functor $\Ffrak:\MorC\to\Wcat$ provides an assignment from continuous tensor categories, continuous tensor functors and continuous monoidal natural transformations to $\mathrm{W}^*$-tensor categories, $\mathrm{W}^*$-tensor functors and monoidal natural transformations.
\end{proposition}
\begin{proof}
    Let $(A, \Tcal, \1, \alpha, \lambda, \rho)$ be a continuous tensor category. We can define a bilinear $\mathrm{W}^*$-functor
    \[
    -\otimes -: \Rep(A)\times \Rep(A)\to \Rep(A\otimes A)\xrightarrow{\Ffrak(\Tcal)} \Rep(A),
    \]
    where the first functor takes two $A$-representations $H$ and $K$ and constructs the $(A\otimes A)$-representation on $H\otimes K$, through the universal property of the maximal tensor product. For the associator, we use $\Ffrak(\alpha)$ together with the canonical natural transformations $(H\otimes K)\otimes R\cong H\otimes (K\otimes R)$ in $\Rep(A)$. The rest of the data trivially carries over and so do the various coherence identities by functoriality of $\Ffrak$. Given a tensor functor $(\Scal, \eta): (A_1, \Tcal_1)\to (A_2, \Tcal_2)$, we construct the $\mathrm{W}^*$-tensor functor $\Ffrak(\Scal):\Rep(A_1)\to \Rep(A_2)$ equipped with the natural isomorphism
% https://q.uiver.app/#q=WzAsNixbMCwwLCJcXFJlcChBXzEpXFx0aW1lc1xcUmVwKEFfMSkiXSxbMSwwLCJcXFJlcChBXzFcXG90aW1lcyBBXzEpIl0sWzAsMSwiXFxSZXAoQV8yKVxcdGltZXNcXFJlcChBXzIpIl0sWzEsMSwiXFxSZXAoQV8yXFxvdGltZXMgQV8yKSJdLFszLDAsIlxcUmVwKEFfMSkiXSxbMywxLCJcXFJlcChBXzIpIl0sWzAsMiwiXFxGY2FsKFxcU2NhbClcXHRpbWVzIFxcRmNhbChcXFNjYWwpIl0sWzEsMywiXFxGY2FsKFxcU2NhbFxcb3RpbWVzXFxTY2FsKSJdLFswLDFdLFsyLDNdLFs0LDUsIlxcRmNhbChcXFNjYWwpIiwxXSxbNCw1XSxbMyw1LCJcXEZjYWwoXFxUY2FsXzIpIiwyXSxbMSw0LCJcXEZjYWwoXFxUY2FsXzEpIl0sWzQsMywiXFxGY2FsKFxcZXRhKSIsMSx7InNob3J0ZW4iOnsic291cmNlIjoyMCwidGFyZ2V0IjoyMH0sImxldmVsIjoyfV0sWzEsMiwiIiwxLHsic2hvcnRlbiI6eyJzb3VyY2UiOjIwLCJ0YXJnZXQiOjIwfSwibGV2ZWwiOjJ9XV0=
\[\begin{tikzcd}
	{\Rep(A_1)\times\Rep(A_1)} & {\Rep(A_1\otimes A_1)} && {\Rep(A_1)} \\
	{\Rep(A_2)\times\Rep(A_2)} & {\Rep(A_2\otimes A_2)} && {\Rep(A_2)},
	\arrow[from=1-1, to=1-2]
	\arrow["{\Ffrak(\Scal)\times \Ffrak(\Scal)}"', from=1-1, to=2-1]
	\arrow["{\Ffrak(\Tcal_1)}", from=1-2, to=1-4]
	\arrow[shorten <=10pt, shorten >=10pt, Rightarrow, from=1-2, to=2-1]
	\arrow["{\Ffrak(\Scal\otimes\Scal)}"', from=1-2, to=2-2]
	\arrow["{\Ffrak(\eta)}"{description}, shorten <=10pt, shorten >=10pt, Rightarrow, from=1-4, to=2-2]
	\arrow["{\Ffrak(\Scal)}", from=1-4, to=2-4]
	\arrow[from=1-4, to=2-4]
	\arrow[from=2-1, to=2-2]
	\arrow["{\Ffrak(\Tcal_2)}"', from=2-2, to=2-4]
\end{tikzcd}\]
where the left-hand side invertible 2-morphism is induced by the intertwiner in \eqref{eq: RieffelVSTensor}. A continuous monoidal natural transformation $\omega$ clearly defines a monoidal natural transformation $\Ffrak(\omega).$
\end{proof}

\begin{example}[$\Hilb\,G$]\label{ex: HilbG}
    Let $G$ be a locally compact group and consider the $\mathrm{C}^*$-algebra $C_0(G)$ of continuous functions on $G$ vanishing at infinity. Then $C_0(G)\otimes C_0(G) \cong C_0(G\times G)$ and we can construct a $C_0(G)-C_0(G\times G)$-correspondence on the trivial $C_0(G\times G)$-Hilbert module $C_0(G\times G)$. The left $C_0(G)$-action is given by $(g\cdot f)(x,y) = g(xy)f(x,y)$ for all $g\in C_0(G)$ and $f\in C_0(G\times G)$. Taking the obvious isomorphism $C_0((G\times G)\times G)\cong C_0(G\times (G\times G))$ we obtain a continuous tensor category which we still denote by $C_0 (G)$. Then, $\Hilb\,G:=\Ffrak(C_0(G))$ is the $\mathrm{W}^*$-tensor category whose objects are direct integrals of skyscraper sheaves of Hilbert spaces on $G$, with tensor product given by convolution.
\end{example}

The previous example can be generalized to allow for a twist in the tensor product and the associator, in analogy to the fusion categories $\Vec^\omega G$ for $G$ finite and $\omega\in H^3(BG;\mathbb{C}^\times)$ a class in group cohomology. When $G$ is a locally compact topological group, the relevant cohomology theory is Segal-Mitchison cohomology, which we recall for convenience. Let $\Bcal G$ be the simplicial topological space whose $n$-simplices are $\Bcal G_n := G^n$ and whose face maps are given by
\[
d_i: (g_1,\ldots, g_n)\mapsto \begin{cases}
     (g_2,\ldots, g_n),& i = 0 \\
     (g_1,\ldots, g_ig_{i+1},\ldots ,g_n) & 0<i<n\\
      (g_1,\ldots, g_{n-1}) & i = n.
\end{cases}
\]
The degeneracy maps $s_i: G^n\to G^{n+1}$ are given by including the identity $e\in G$. A simplicial cover $Y_\bullet\twoheadrightarrow \Bcal G$ is a simplicial topological space $Y_\bullet$ together with covering maps $Y_n \to \Bcal G_n$ commuting with the face maps. Associated to a simplicial cover $Y_\bullet\twoheadrightarrow \Bcal G$ there is a simplicial double complex whose $(p,q)$-th entry is $C(Y_q^{[p]}, U(1))$, the space of continuous functions from the $p$-fold product of $Y_q$ over $G^q$, denoted $Y_q^{[p]}$,  into $U(1)$. The vertical differential is the \v{C}ech differential $\check{d}$ and the horizontal differential is the alternating sum of the pullback maps 
\[
d^*_i:C(Y_q^{[p]}, U(1))\to C(Y_{q+1}^{[p]}, U(1)),
\]
where we also denote by $d_i$ the horizontal maps of the simplicial space $Y_\bullet$. We write $H^\bullet_{SM}(\mathscr{B}G; U(1); Y_\bullet)$ for the cohomology of the totalization of the double complex $\big(C(Y^{[*]}_{*}, U(1)), d_i,\check{d} \big)$.

A simplicial cover is called locally finite if each cover is locally finite. The set of locally finite simplicial covers is a poset with partial order given by refinement. The Segal-Mitchison cohomology of $G$ with coefficients in $U(1)$ is 
\[
H^\bullet_{SM}(\mathscr{B}G; U(1)) := \varinjlim\limits_{Y_\bullet\to \mathscr{B}
G} H^\bullet_{SM}(\mathscr{B}G; U(1); Y_\bullet),
\]
where the direct limit runs over the poset of locally finite simplicial covers of $\mathscr{B}G$.

Classes in Segal-Mitchison cohomology induce continuous tensor categories via the following construction. Let $X$ be a paracompact space with a locally finite cover $Y\to X$. A continuous map $\lambda: Y^{[3]}\to U(1)$ satisfying $\check{d}\lambda = 0$ represents a class in the second \v{C}ech cohomology of $X$. In addition, $\lambda$ can be assumed to be alternating without altering its cohomology class \cite[Prop. 4.41]{ContinuousTrace}. Then, we obtain a $\mathrm{C}^*$-algebra with spectrum $X$ as follows, see \cite[Sec. 5.3]{ContinuousTrace}. The vector space $C_c(Y_1^{[2]})$ of continuous compactly supported $\CC$-valued functions on $Y_1^{[2]}$ admits an associative multiplication given~by
    \[
    (f*g)(x,z) = \sum\limits_{(x,y,z)\in Y_1^{[3]}}f(x,y)g(y,z)\overline{\lambda(x,y,z)}
    \]
    and an involution $f\mapsto f^*$ defined by $f^*(x,y) = \overline{f(y,x)}$. This algebra can be completed with respect to a certain norm to produce a $\mathrm{C}^*$-algebra we denote by $C_0(X, \lambda)$, and whose space of irreducible representations is homeomorphic to $X$, see \cite[Ex. 5.20]{ContinuousTrace}. Cohomologous cocycles in \v{C}ech cohomology give Morita equivalent $\mathrm{C}^*$-algebras. Indeed, let $Z\to X$ be another locally finite cover and $\nu: Z^{[3]}\to U(1)$ another continuous function with $\check{d}\nu = 0$. Denote $W : = Y\times_X Z$ and let $\mu: W^{[2]}\to U(1)$ be a continuous function such that 
    \begin{equation}\label{eq: cocycleeq}
    \check{d}\mu = (p_Y^*\lambda)(p_Z^*\nu)^{-1}
    \end{equation}
    for $p_Y: W^{[3]}\to Y^{[3]}$ and $p_Z: W^{[3]}\to Z^{[3]}$ the projections. Let $C_c(W)$ denote the vector space of continuous compactly supported functions on $W$. Then, $C_c(W)$ is a $C_c(Y^{[2]})-C_c(Z^{[2]})$-bimodule as follows. For $f\in C_c(Y^{[2]})$, $g\in C_c(Z^{[2]})$ and $\xi\in C_c(W)$, we define
    \begin{align*}
    (f\xi)(y,z) &= \sum\limits_{(y,y')\in Y^{[2]}} f(y,y')\xi(y',z)\overline{\mu((y,z), (y',z))}\\(\xi g)(y,z) &= \sum\limits_{(z',z)\in Z^{[2]}} \xi(y,z')g(z',z)\overline{\mu((y,z'), (y,z))}
    \end{align*}
    for every $(y,z)\in W$. In addition, we define maps $_Y\langle - , -\rangle: C_c(W)\times C_c(W)\to C_c(Y^{[2]})$ and $\langle - , -\rangle_Z: C_c(W)\times C_c(W)\to C_c(Z^{[2]})$ by 
\begin{align*}
_Y\langle \xi , \zeta\rangle(y_1, y_2) &= \sum\limits_{(y_1, z)\in W} \xi(y_1, z)\overline{\zeta(y_2, z)}\mu((y_1, z), (y_2, z))\\ 
\langle \xi , \zeta\rangle_Z(z_1, z_2) &= \sum\limits_{(y, z_1)\in W} \overline{\xi(y, z_1)}{\zeta(y, z_2)}\mu((y, z_1), (y, z_2))
\end{align*}
for all $(y_1,y_2)\in Y^{[2]}$ and $(z_1, z_2)\in Z^{[2]}$. It is a routine check that Equation \eqref{eq: cocycleeq} implies that, with the structure above, $C_c(W)$ is a pre-Hilbert module for both $C_c(Y^{[2]})$ (on the left) and $C_c(Z^{[2]})$ (on the right). In addition, both actions commute. Since both $Y,Z\to X$ are covers, the maps $W\to Y$ and $W\to Z$ are surjective. Since $X$ is paracompact, we can find a locally finite partition of unity subordinate to $Y$ and, for each $(y_1,y_2)\in Y^{[2]}$, a $z\in Z$ such that $(y_i,z)\in W$. Taking bump functions supported near the points $(y_1,z)$ and $(y_2,z)$ of $W$, we obtain functions $\xi, \zeta\in C_c(W)$ with $_Y\langle \xi, \zeta\rangle$ approximating any given compactly supported function $f\in C_c(Y^{[2]})$ arbitrarily well. Hence, the linear span of $_Y\langle-,-\rangle$ is dense in $C_0(X, \lambda)$. A similar argument shows that the linear span of $\langle-,-\rangle_Z$ is dense in $C_0(X,\nu)$. Hence, by \cite[Lem. 2.16]{ContinuousTrace}, $C_c(W)$ can be completed to an invertible $C_0(X,\lambda)-C_0(X, \nu)$-correspondence $\mathcal{M}_\mu$. If the cocycle $\lambda$ is trivial in \v{C}ech cohomology, then $C_0(X, \lambda)$ is Morita equivalent to $C_0(X)$, hence justifying our notation.

If $\mu': W^{[2]}\to U(1)$ is another continuous map satisfying $\check{d}\mu' = (p_Y^*\lambda)(p_Z^*\nu)^{-1}$, and $\omega: W\to U(1)$ is a continuous map such that $\check{d}\omega = \mu(\mu')^{-1}$, the completion of
\[
\begin{array}{ccc}
    C_c(W) &\to &C_c(W)  \\
    \xi & \mapsto &\overline{\omega}\xi
\end{array}
\]
is a unitary intertwiner between the $C_0(X, \lambda)-C_0(X, \nu)$-correspondences $\mathcal{M}_\mu$ and $\mathcal{M}_{\mu'}.$ This construction can be straightforwardly generalized to the case where $\mu$ and $\mu'$ are defined on different covers $W$ and $W'$, and $\omega$ is defined on the refinement $W\times_XW'$.

\begin{example}[$\Hilb^\omega G$]
\label{ex: HilbOmega}
    Let $G$ be a locally compact group and $\omega\in H_{SM}^3(\Bcal G; U(1))$ a class in Segal-Mitchison cohomology. Then, there is a locally finite simplicial cover $Y_\bullet\to \Bcal G$ so that the class $\omega$ is represented by a triple of continuous functions
    \[
    \lambda: Y_1^{[3]}\to U(1) \hspace{1cm} \mu: Y_2^{[2]}\to U(1)\hspace{1cm}\Omega: Y_3\to U(1).
    \]
    These can be assumed to be normalized, meaning that their pullbacks along degeneracy maps $s_i$ are trivial. These functions satisfy
    \begin{align*}
        \check{d}\lambda &= 1\\
        \check{d}\mu &= d_0^*\lambda\cdot d_1^*\lambda^{-1}\cdot d_2^*\lambda\\
        \check{d}\Omega & = d_0^*\mu^{-1}\cdot d_1^*\mu\cdot d_2^*\mu^{-1}\cdot d_3^*\mu\\
         d_1^*\Omega\cdot d_3^*\Omega & = d_0^*\Omega\cdot d_2^*\Omega\cdot d_4^*\Omega.
    \end{align*}
    The condition for $\lambda$ on $Y_1^{[4]}$ is exactly that it defines a cocycle in the \v{C}ech cohomology of $G$ relative to the cover $Y_1$, and we can further assume that $\lambda$ is alternating, see \cite[Prop. 4.41]{ContinuousTrace}. As described above, such a cocycle defines a $\mathrm{C}^*$-algebra $C_0(G, \lambda)$. 
    
    We next produce a $C_0(G,\lambda)-C_0(G,\lambda)\otimes C_0(G,\lambda)$-correspondence witnessing the tensor product. It holds that $C_0(G,\lambda)\otimes C_0(G,\lambda)\cong C_0(G\times G,p_1^*\lambda \cdot p_2^* \lambda)$ by nuclearity of $C_0(G,\lambda)$ (see for example \cite[Cor. B.44]{ContinuousTrace}), where the right-hand side is the $\mathrm{C}^*$-algebra constructed as above for the data $Y_1\times Y_1\to G\times G$ with cocycle
   \[
    Y_1^{[3]}\times Y_1^{[3]}\xrightarrow{\lambda\times\lambda} U(1)\times U(1)\xrightarrow{\mathrm{m}} U(1),
    \]
    for $\mathrm{m}$ the multiplication on $U(1)$. We first need to find a suitable cohomologous cocycle to $p_1^*\lambda \cdot p_2^*\lambda$ with respect to the cover $Y_2\to G\times G$. We take the cocycle 
    \[
    Y_2^{[3]}\xrightarrow{(d_0, d_2)}Y_1^{[3]}\times Y_1^{[3]}\xrightarrow{\lambda\times \lambda} U(1)\times U(1)\xrightarrow{\mathrm{m}}U(1),
    \]
    which we denote by $d_0^*\lambda \cdot d_2^*\lambda$, and which produces another $\mathrm{C}^*$-algebra $C_0(G\times G, d_0^*\lambda\cdot d_2^*\lambda)$. The fact that the cocycles $p_1^*\lambda\cdot p_2^*\lambda$ and $d_0^*\lambda\cdot d_2^*\lambda$ are canonically cohomologous produces a canonical invertible correspondence $\mathcal{E}$ between $C_0(G\times G, d_0^*\lambda \cdot d_2^*\lambda)$ and $C_0(G\times G, p_1^*\lambda\cdot p_2^*\lambda)$. Similarly, the cocycle $\lambda$ on $G$ can be pulled back to a cocycle $m^*\lambda$ on $G\times G$ with respect to the cover $m^*Y_1\to G\times G$. The cohomology class of $m^*\lambda$ can also be represented with respect to the cover $Y_2\to G\times G$ by the cocycle $d_1^*\lambda:= \lambda\circ d_1$. These produce two new $\mathrm{C}^*$-algebras $C_0(G\times G, m^*\lambda)$ and $C_0(G\times G, d_1^*\lambda)$, which are equivalent in $\MorC$ via a canonical invertible correspondence $\Dcal$.

    The compatibility condition between $\mu$ and $\lambda$ on $Y_2^{[3]}$ is exactly that $\mu$ is a witness of the fact that $d_0^*\lambda\cdot d_2^*\lambda$ and $d_1^*\lambda$ are cohomologous. This provides an invertible $\mathrm{C}^*$-correspondence $\mathcal{M}_\mu$ between $C_0(G\times G, d_1^*\lambda)$ and $C_0(G\times G, d_0^*\lambda\cdot d_2^*\lambda)$ which depends on $\mu$. Finally, there is a $C_0(G,\lambda)-C_0(G\times G, m^*\lambda)$-correspondence on the trivial $C_0(G\times G, m^*\lambda)$-Hilbert module $C_0(G\times G, m^*\lambda)$ with left $C_0(G, \lambda)$-action
   induced by the multiplication on $G$. Composing these four correspondences, we obtain the $C_0(G, \lambda)-C_0(G,\lambda)\otimes C_0(G, \lambda)$-correspondence 
    \[
    \mathcal{T}_{\lambda, \mu}:=C_0(G\times G, m^*\lambda)\otimes_{C_0(G\times G, m^*\lambda)}\mathcal{D}\otimes_{C_0(G\times G, d_1^*\lambda)}\mathcal{M}_\mu\otimes_{C_0(G\times G, d_0^*\lambda \cdot d_2^*\lambda)}\mathcal{E},
    \]
    which encodes the tensor product. Pulling these correspondences further back to $G\times G\times G$, we obtain two different $C_0(G,\lambda)-C_0(G,\lambda)^{\otimes 3}$-correspondences, namely 
    \[
    \mathcal{T}_{\lambda, \mu}\otimes_{C_0(G,\lambda)^{\otimes 2}}(\mathcal{T}_{\lambda, \mu}\otimes C_0(G,\lambda))
    \]
    and
      \[
      \mathcal{T}_{\lambda, \mu}\otimes_{C_0(G,\lambda)^{\otimes 2}}( C_0(G,\lambda)\otimes \mathcal{T}_{\lambda, \mu}).
    \]
Pulling back all the relevant cocycles in these expressions to the common cover $Y_3\to G^3$ and making the canonical refinement identifications, the former correspondence is controlled by $d_1^*\mu\cdot d^*_3\mu$, and the latter is controlled by $d_2^*\mu\cdot d^*_4\mu$. Therefore $\Omega$ defines a unitary intertwiner
\[
\alpha_\Omega = \mathcal{T}_{\lambda, \mu}\otimes_{C_0(G,\lambda)^{\otimes 2}}(\mathcal{T}_{\lambda, \mu}\otimes C_0(G,\lambda))\to \mathcal{T}_{\lambda, \mu}\otimes_{C_0(G,\lambda)^{\otimes 2}}( C_0(G,\lambda)\otimes \mathcal{T}_{\lambda, \mu}).
\]
The condition for $\Omega$ on $Y_4$ is exactly the pentagon equation, after pulling back all the relevant correspondences and intertwiners to the common refinement $Y_4\to G^4$.

We can construct a unit for this continuous tensor category as follows. Let $Y_e$ be the preimage of the unit $e\in G$ under $Y_1\to G$. Then, taking the counting measure on the finite set $Y_e$, we define $H:=\ell^2(Y_e)$. We define an action of $C_0(G, \lambda)$ on $H$ by letting a function $f\in C_c(Y^{[2]})$ act on $\phi\in \ell^2(Y_e)$ by
\[
(f\cdot\phi)(y) = \sum\limits_{z\in Y_e} f(y, z)\lambda(y, z, y_e)\phi(z),
\]
for a fixed $y_e\in Y_e$. Picking a different element $y_e\in Y_e$ yields an isomorphic representation. Next, the left unitor is a unitary intertwiner
\[
l: \Tcal_{\lambda, \mu}\otimes_{C_0(G, \lambda)^{\otimes 2}} (C_0(G, \lambda)\otimes H)\to C_0(G, \lambda).
\]
The left-hand side is the pullback of $\Tcal_{\lambda, \mu}$ along the degeneracy map $s_1:G\to G\times G$, $g\mapsto (g,e)$. The fact that $\mu$ is normalized implies that $s_1^*\mu$ is trivial, hence providing a canonical unitary intertwiner between the $C_0(G, \lambda)-C_0(G, \lambda)$-correspondences $\Tcal_{\lambda, \mu}\otimes_{C_0(G, \lambda)^{\otimes 2}} (C_0(G, \lambda)\otimes H)  = s_1^*\Tcal_{\lambda, \mu}$ and $C_0(G, \lambda)$. The right unitor can be constructed similarly. The triangle identities follow from the fact that $\Omega$ is normalized.

Taking a different locally finite simplicial cover or a different representative for the class $\omega$ yields an equivalent continuous tensor category. Indeed, suppose that we have another locally finite cover $Z_\bullet\to \mathscr{B}G$ and a triple $(\lambda', \mu', \Omega')$ of functions on $Z_1^{[3]} $, $Z_2^{[2]}$, and $Z_3$ respectively representing the same class $\omega$. Then, denoting by $W_\bullet:= Y_\bullet\times_{\mathscr{B}G}Z_\bullet$ the pullback cover, there exists a pair of functions $\kappa: W_1^{[2]}\to U(1)$ and $\nu: W_2\to U(1)$ satisfying
\begin{align*}
\check{d}\kappa &= p_Y^*\lambda\cdot (p_Z^*\lambda')^{-1}\\
\check{d}\nu  &= d_0^*\kappa\cdot d_1^*\kappa^{-1}\cdot d_2^*\kappa\cdot   p_Y^*\mu\cdot (p_Z^*\mu')^{-1}\\ d_0^*\nu\cdot d_2^*\nu\cdot p_Y^*\Omega &= d_1^*\nu\cdot d_3^*\nu\cdot p_Z^*\Omega' , 
\end{align*}
for $p_Y, p_Z$ the obvious projections in each equation. The function $\kappa$ provides a canonical Morita equivalence $\mathcal{S}_\kappa$ between $C_0(G, \lambda')$ and $C_0(G, \lambda)$, and $\nu$ provides a unitary intertwiner $s_\nu$ between $\mathcal{S}_\kappa\otimes_{C_0(G, \lambda)}\mathcal{T}_{\lambda, \mu}$ and $\mathcal{T}_{\lambda', \mu'}\otimes_{C_0(G, \lambda')^{\otimes 2}}(\mathcal{S}_\kappa\otimes\mathcal{S}_\kappa)$. The compatibility for $\nu$, $\Omega$ and $\Omega'$ on $W_3$ implies that $s_\nu$ is a tensorator for $\mathcal{S}_\kappa$.
\end{example}

\begin{remark}
     Assume $G$ is a Lie group. Then, the class $\omega\in H^3_{SM}(\Bcal G;U(1))$ from the previous example provides a multiplicative $U(1)$-gerbe on $G$, see \cite{waldorf}. The data of a multiplicative gerbe on $G$ defines a manifold tensor category in the sense of \cite{CW}. The continuous tensor categories associated to a locally compact group $G$ and a class in Segal-Mitchison cohomology in Example \ref{ex: HilbOmega} are analogues of manifold tensor categories constructed in \cite{CW, string}.
\end{remark}

\begin{example}
    Example \ref{ex: HilbOmega} can be substantially simplified when the cocycle $[\lambda]$ in the \v{C}ech cohomology of $G$ is trivial. Assume that the class in $H^3_{SM}(\Bcal G; U(1))$ is represented by a triple $(1, \mu, \omega)$. Then, the underlying $\mathrm{C}^*$-algebra of the continuous tensor category associated to this data is commutative, and it is equivalent in $\MorC$ to $C_0(G)$. We can then provide a $C_0(G)-C_0(G\times G)$-correspondence witnessing the tensor product as follows. The compatibility condition of $\mu$ on $Y_2^{[3]}$ is exactly that $\mu$ produces a line bundle $\mathcal{L}_\mu$ on $G\times G$. Then, the vector space $\Gamma_c(\mathcal{L}_\mu)$ of compactly-supported continuous sections of $\mathcal{L}_\mu$ can be suitably completed to obtain a right $C_0(G\times G)$-Hilbert module. In addition, $C_0(G)$ acts on the left via pullback along the multiplication $m:G\times G\to G$. The function $\omega$ provides an isomorphism of line bundles $d_0^*\mathcal{L}_\mu\otimes d_2^*\mathcal{L}_\mu\cong d_1^*\mathcal{L}_\mu\otimes d_3^*\mathcal{L}_\mu$ on $G\times G\times G$, which gives the associator of the continuous tensor category.
\end{example}

\begin{example}[$\Rep(G)$]\label{Ex: GroupC*Algebra}
    Let $G$ be a locally compact group. Let $\mu$ be its (left) Haar measure and $L^1(G)$ the vector space of $\mu$-integrable functions on $G$. Recall that $L^1(G)$ is an algebra with multiplication given by
    \[
    (f*g)(x):=\int_G f(y)g(y^{-1}x)d\mu(y).
    \]
    Given $f\in L^1(G)$, we define
    \[
    ||f|| : = \sup\{||\pi(f)|| \ |\ \text{$\pi$ is a representation of $L^1(G)$}\}.
    \]
    This algebra admits an involution given by $f^*(x) = \Delta(x^{-1})\overline{f(x^{-1})}$, for $\Delta$ the modular function on $G$. The completion of $L^1(G)$ with respect to this norm is called the (full) group $\mathrm{C}^*$-algebra of $G$ and is denoted $\mathrm{C}^*(G)$. Then, $\Rep(\mathrm{C}^*(G))\cong \Rep(G)$, the category of strongly continuous unitary representations of $G$ \cite[II.10.2.4]{OAbook}. Using the fact that $\mathrm{C}^*(G)\otimes \mathrm{C}^*(G)\cong \mathrm{C}^*(G\times G)$, one can endow $\mathrm{C}^*(G)$ with a structure of a $\mathrm{C}^*$-Hopf algebra (whose comultiplication lies in the multiplier algebra of $\mathrm{C}^*(G\times G)$), see \cite{HopfC*}. This map gives an action of $\mathrm{C}^*(G)$ on $\mathrm{C}^*(G\times G)$, which produces a canonical $\mathrm{C}^*(G)-\mathrm{C}^*(G\times G)$-correspondence on the trivial $\mathrm{C}^*(G)\otimes \mathrm{C}^*(G)$-Hilbert module $\mathrm{C}^*(G)\otimes \mathrm{C}^*(G)$. The canonical associator upgrades the data above to a continuous tensor category whose image under $\Ffrak$ is equivalent to $\Rep(G)$ as a $\mathrm{W}^*$-tensor category.
\end{example}

\section{Continuous Tambara-Yamagami tensor categories}

This section is devoted to the definition and classification of continuous Tambara-Yamagami categories. After providing a definition of these, we present a construction that takes as input a locally compact abelian group $G$, a continuous symmetric nondegenerate bicharacter $\chi: G\times G\to U(1)$, and a sign $\xi\in\{\pm1\}$ and produces a continuous Tambara-Yamagami category. We then define Tambara-Yamagami $\mathrm{W}^*$-tensor categories as those $\mathrm{W}^*$-tensor categories which have a tensor product structure that mimics the finite Tambara-Yamagami fusion rules. We show that Tambara-Yamagami $\mathrm{W}^*$-tensor categories are always equivalent to the image of a continuous Tambara-Yamagami category under $\mathfrak{F}$ and that both continuous and $\mathrm{W}^*$ Tambara-Yamagami tensor categories are classified by a locally compact group, a continuous symmetric nondegenerate bicharacter and a sign via the construction above. This can be thought of as an automatic continuity of the associators of a Tambara-Yamagami $\mathrm{W}^*$-tensor category.

\subsection{Definition and construction}
\label{sec: ctsTYcats}
In this section, we define continuous Tambara-Yamagami tensor categories. Fix a locally compact abelian group $G$ (whose group operation we write additively) and let $A:=C_0(G)\oplus\CC$. We construct an $A-A\otimes A$-correspondence that generalizes the finite Tambara-Yamagami fusion rules. The maximal tensor product $A\otimes A$ is canonically isomorphic to $C_0(G^2)\oplus C_0(G)\oplus C_0(G)\oplus \CC$ under the isomorphism
\[
\begin{array}{ccc}
(C_0(G)\oplus\CC)\otimes(C_0(G)\oplus\CC) &\to & C_0(G\times G)\oplus C_0(G)\oplus C_0(G)\oplus \CC\\
(\phi_1, c_1)\otimes (\phi_2, c_2)&\mapsto &(p_1^*\phi_1p_2^*\phi_2, c_2\phi_1, c_1\phi_2, c_1c_2),
\end{array}
\]
for $p_i: G^2\to G$ the projections. Hence, we can construct an $A-A\otimes A$-correspondence from the following pieces.
\begin{enumerate}
    \item The $C_0(G)-C_0(G\times G)$-correspondence on the trivial $C_0(G\times G)$-Hilbert module $C_0(G\times G)$ with $C_0(G)$-action given by $(\phi f)(x,y) = \phi(x+y)f(x,y)$ for $\phi\in C_0(G)$ and $f\in C_0(G\times G)$.
\item The  $\CC-C_0(G)$-correspondence which is the trivial $C_0(G)$-Hilbert module $C_0(G)$.
\item The $C_0(G)-\CC$-correspondence given by the regular representation $L^2(G)$ of $C_0(G)$.
\end{enumerate}
This data, using (\textrm{ii}) twice, produces an $A-A\otimes A$-correspondence which has underlying $(A\otimes A)$-Hilbert module the direct sum of Hilbert modules
\[
\TYcal_G: = C_0(G\times G)\oplus C_0(G)\oplus C_0(G)\oplus L^2(G)
\]
with the obvious left $(C_0(G)\oplus \CC)$-action induced componentwise by projecting onto $C_0(G)$ or $\CC$ and acting as above for every summand of $\TYcal_G$. We denote this $A-A\otimes A$-correspondence again by $\TYcal_G$. The $\mathrm{C}^*$-algebra $A$ also comes with a canonical unit morphism given by the 1-dimensional $(C_0(G)\oplus \CC)$-representation defined by
\[
(\phi, a)\cdot b : = \phi(e)b,
\]
for $\phi\in C_0(G)$, $a,b\in \CC$ and $e\in G$ the identity element.

\begin{remark}\label{Rk: TYcalG as functor}
    The definition of the correspondence $\TYcal_G$ implies that the tensor product it induces on
    \[
    \Rep(C_0(G)\oplus \CC)
    \]
     is a $\Z/2$-graded tensor product with respect to the grading $\Rep(C_0(G)\oplus \CC)\cong \Rep(C_0(G))\oplus \Hilb\cdot \tau$, where we pick a simple object $\tau$ of $\Hilb= \Rep(\CC)$. On the trivial component $\Rep(C_0(G))$, the tensor product is that of $\Hilb\,G$. The square of $\tau$ is the $C_0(G)$-representation $L^2(G).$ The middle two summands of $\TYcal_G$ define the left and right module structures of $\Hilb$ over $\Rep(C_0(G))$ induced by the functor that forgets the $C_0(G)$-action.
\end{remark}

\begin{definition}\label{def: ctsTY}
    A \emph{continuous Tambara-Yamagami tensor category for $G$} is a continuous tensor category of the form $(C_0(G)\oplus \CC, \TYcal_G, \alpha)$ for some associator $\alpha$, and whose unit data is given by the canonical unit data of $(C_0(G)\oplus\CC, \TYcal_G)$.
\end{definition}

Let us provide a class of examples of continuous Tambara-Yamagami tensor categories. Actually, as we shall see later, these exhaust all possibilities up to equivalence. Let $\chi: G\times G\to U(1)$ be a continuous symmetric bicharacter which is nondegenerate in the sense that it induces an isomorphism
\[
\begin{array}{cccc}
j_\chi: &G&\xrightarrow{\cong} &\hat{G}\\
&y&\mapsto & \big(x\mapsto \chi(x,y)\big),
\end{array}
\]
and let $\xi\in \{1, -1\}$. We show that this data provides an associator for $(A: = C_0(G)\oplus \CC, \TYcal_G)$. Just as we have done for $\TYcal_G$, and in line with Remark \ref{Rk: TYcalG as functor}, we can decompose the associator in eight pieces, one for each ordered triple $(X,Y,Z)$ for $X,Y,Z\in \{C_0(G), \CC\}.$ Note that some of the associators will be given by isomorphisms of $C_0(G)-C_0(G)$-correspondences between the composition $L^2(G)\otimes_{\CC} C_0(G)$ and itself. One can compute that the underlying vector space of this composition is $C_0(G)\otimes_\varepsilon L^2(G)$, the injective tensor product of the underlying Banach spaces of $C_0(G)$ and $L^2(G)$. This space can be identified with $C_0(G, L^2(G))$, and we will use this identification throughout. Recall that we write $\mathcal{F}: L^2(\hat{G})\to L^2(G)$ for the unitary Fourier transform. We set $c_\chi\in \R_{>0}$ as the unique positive number such that $(j_\chi)_*\mu = c_\chi\cdot \hat{\mu}$, and write $J_\chi: L^2(G)\to L^2(\hat{G})$ for the unitary $J_\chi(f)(\eta) = \sqrt{c_\chi}
\cdot f\big(j_\chi^{-1}(\eta)\big)$. The eight intertwiners that comprise the associator~are
\begin{enumerate}
    \item for $(C_0(G), C_0(G), C_0(G))$, 
    \[
    \id : C_0(G^3)\to C_0(G^3)
    \]
    as $C_0(G)-C_0(G^3)$-correspondences.
    \item for $(\CC, C_0(G), C_0(G))$ and $( C_0(G), C_0(G), \CC)$, 
    \[
    \id: C_0(G^2)\to C_0(G^2)
    \]
    as $\CC-C_0(G^2)$-correspondences.
    \item for $(C_0(G), \CC, C_0(G))$, 
\[
\begin{array}{ccc}
C_0(G^2)&\to&C_0(G^2)\\
f&\mapsto &\chi f
\end{array}
\]
as $\CC-C_0(G^2)$-correspondences.
\item for $(C_0(G), \CC, \CC)$, 
\[
\begin{array}{ccc}
    C_0(G)\otimes_\varepsilon L^2(G) &\to  &C_0(G)\otimes_\varepsilon L^2(G) \\
     f(x,y)&\mapsto & f(x, x+y)
\end{array}
\]
as $C_0(G)-C_0(G)$-correspondences.
\item for $(\CC, \CC, C_0(G))$,
\[
\begin{array}{ccc}
    C_0(G)\otimes_\varepsilon L^2(G) &\to  &C_0(G)\otimes_\varepsilon L^2(G) \\
     f(x,y)&\mapsto & f(x, -x+y)
\end{array}
\]
as $C_0(G)-C_0(G)$-correspondences.

\item for $(\CC, C_0(G), \CC)$, 
\[
\begin{array}{ccc}
    C_0(G)\otimes_\varepsilon L^2(G) &\to  &C_0(G)\otimes_\varepsilon L^2(G) \\
     f&\mapsto & \chi f
\end{array}
\]
as $C_0(G)-C_0(G)$-correspondences.
\item for $(\CC, \CC, \CC)$,
\[
L^2(G)\xrightarrow{J_\chi} L^2(\hat{G})\xrightarrow{\xi\cdot \Fcal}L^2(G),
\]
as $\CC-\CC$-correspondences.
\end{enumerate}
The data above defines an intertwiner
\[
\alpha_{\chi, \xi}: \TYcal_G\otimes_{A\otimes A}(\TYcal_G\otimes A)\xrightarrow{\cong}\TYcal_G\otimes_{A\otimes A}(A\otimes \TYcal_G)
\]
which provides an associator.
\begin{proposition}\label{prop: ctsTYexist}
    The data above defines a continuous Tambara-Yamagami category $\TYcal(G, \chi, \xi)$.
\end{proposition}
\begin{proof}
    The triangle equations are all trivial except for the case $(\tau, \tau)$. There, the triangle equation holds as $\chi(e, x) = \chi( x, e) = 1$ for all $x\in G$.

    It remains only to argue that the associator $\alpha_{\chi, \xi}$ satisfies the pentagon equation. We can decompose the pentagon equation into 16 equations, one for each ordered tuple $(X,Y,Z,W)$ for $X,Y,Z,W\in\{C_0(G), \CC\}$. Denote each such equation by the set of positions with the entry $\CC$. For example, $\{1,3\}$ denotes the equation corresponding to $(\CC, C_0(G), \CC, C_0(G))$. We now list the reasons for commutativity of the nontrivial pentagon diagrams. $\{2\}$, $\{3\}$, $\{1,3\}$, $\{1,4\}$ and $\{2,4\}$ commute by multiplicativity of the bicharacter. Next, $\{1,2,3\}$ and $\{1,3,4\}$ commute by Proposition \ref{prop: PropertiesF}\ref{prop: PropertiesF1} and $\{2,3,4\}, \{1,2,4\}$ commute by Proposition \ref{prop: PropertiesF}\ref{prop: PropertiesF2}. Finally, $\{1,2,3,4\}$ is equivalent to the equality
    \begin{equation}\label{eq: tautautautau}
\xi^2(\mathcal{F}J_\chi\otimes \id)\Big(\chi(x,y)\cdot (\mathcal{F} J_\chi\otimes\id)(f)(x,y)\Big) = f(y-x, y)
    \end{equation}
    for all $f\in L^2(G\times G).$ By the fact that $\xi^2 = 1$, the frequency shifting property of the Fourier transform and the Inverse Fourier Transform Theorem \cite[Thm. 1.5.1]{rudin}, Equation \eqref{eq: tautautautau} holds.
\end{proof}

We may sometimes refer to $\tau$ as the (unique) simple non-invertible object. We note that $\tau$ is invertible if $G$ is trivial. This case is also treated in our discussion although the terminology is adapted to the more interesting case $G\neq\{e\}.$

For completeness, in the remainder of this section, we spell out the data of the image of $\TYcal(G, \chi, \xi)$ as a $\mathrm{W}^*$-tensor category under the assignment in Proposition \ref{Prop: ForgetE1}, which we also denote by $\TYcal(G, \chi, \xi)$. The underlying $\mathrm{W}^*$-category is $\Rep(C_0(G)\oplus\CC)\cong \Rep(C_0(G))\oplus \Hilb\cdot\tau$. The simple objects of $\Rep(C_0(G))$ are parametrized by the topological space $G$. Given $x\in G$, the associated irreducible representation of $C_0(G)$ has underlying Hilbert space $\CC$ and a function $f\in C_0(G)$ acts by multiplication by $f(x)$. We denote this representation by $\delta_x\in \Rep(C_0(G))$. For simplicity, we work on the following full subcategory of $\Rep(C_0(G))$. Given a locally compact Hausdorff second countable space $T$, a continuous map $p:T\to G$ and a positive Radon measure $\upsilon$ on $T$, we obtain an object, denoted $\int^{\oplus}_{t\in T}\delta_{p(t)}\mathrm{d}\upsilon(t)$, of $\Rep(C_0(G))$ as follows. The underlying Hilbert space of $\int^{\oplus}_{t\in T}\delta_{p(t)}\mathrm{d}\upsilon(t)$ is $L^2(T, \upsilon)$ and the $C_0(G)$-action is given by multiplication after pulling back along $p$. We also denote this object by $\upsilon$ or $(p, \upsilon)$ when the rest of the data is clear from the context. Whenever we write $\int^{\oplus}_{t\in T}\delta_{p(t)}\mathrm{d}\upsilon(t)\in \Rep(C_0(G))$ we mean that $T$ is a locally compact Hausdorff second countable space, $p:T\to G$ is a continuous map, $\upsilon$ is a positive Radon measure on $T$, and $\int^{\oplus}_{t\in T}\delta_{p(t)}\mathrm{d}\upsilon(t)$ is constructed as above.
\begin{proposition}\label{Prop: FullSubcat}
	Let $X$ be a locally compact Hausdorff second countable space. Then, the full subcategory of $\Rep\big(C_0(X)\big)$ on the objects of the form $\int^{\oplus}_{t\in T}\delta_{p(t)}\mathrm{d}\upsilon(t)$ is equivalent to $\Rep\big(C_0(X)\big)$. 
\end{proposition}
\begin{proof}
Let $\mathcal{D}$ be the full subcategory of $\Rep\big(C_0(X)\big)$ on objects of the form $\int^{\oplus}_{t\in T}\delta_{p(t)}\mathrm{d}\upsilon(t)$. It is enough to show that the inclusion functor  $\mathcal{D}\to \Rep\big(C_0(X)\big)$ is essentially surjective. Let $K\in \Rep\big(C_0(X)\big)$. Then, there is a countable collection of cyclic subrepresentations $\{K_i\subset K\}_{i\in I}$, pairwise orthogonal and such that
\[
\bigoplus\limits_{i\in I} K_i= K.
\]
By the GNS construction, each $K_i$ is induced by a state $\rho_i$ on $C_0(X)$, and by the Riesz Representation Theorem, each state $\rho_i$ is given by a Baire probability measure $\upsilon_i$ on X. Hence, there are measures $\{\upsilon_i\}_{i\in I}$ on $X$ such that
\[
K_i \cong L^2(X, \upsilon_i).
\]
Let $T : = \bigsqcup_{i\in I} X$ and $p :T\to X$ be the map whose restriction to every copy of $X$ is the identity. Then, $\upsilon:= \sum_{i\in I}\upsilon_i$ is a positive Radon measure on $T$ and 
\[
K \cong L^2(T, \upsilon)
\]
with $C_0(X)$-action on the right-hand side given by multiplication after pulling back along~$p$.
\end{proof}

Given $\int^{\oplus}_{t\in T}\delta_{p(t)}\mathrm{d}\upsilon(t)$ and $\int^{\oplus}_{s\in S}\delta_{n(s)}\mathrm{d}\nu(s)$ objects of $\Rep(C_0(G))$, the object
\begin{equation}\label{eq: productHilbG}
\int^{\oplus}_{(t,s)\in T\times S}\delta_{p(t)+n(s)}\mathrm{d}(\upsilon\times \nu)(t,s)
\end{equation}
is also denoted by $\upsilon\times\nu$ or $(p+n, \upsilon\times\nu)$, in line with the notation described above. We note that, denoting by $\mu$ the Haar measure on $G$, the object $\int^{\oplus}_{x\in G}\delta_x\mathrm{d}\mu(x)$ is the representation $L^2(G)$ of $C_0(G)$. The object in Equation \eqref{eq: productHilbG} is the tensor product of $\int^{\oplus}_{t\in T}\delta_{p(t)}\mathrm{d}\upsilon(t)$ and $\int^\oplus_{s\in S}\delta_{n(s)}\mathrm{d}\nu(s)$ in $\Hilb\,G$.

We can apply Proposition \ref{Prop: ForgetE1} to $\TYcal(G, \chi, \xi)$ to obtain a $\mathrm{W}^*$-tensor category also denoted $\TYcal(G, \chi, \xi).$ Given a Hilbert space $H\in \Hilb$ and a representation $K\in \Rep(C_0(G))$, we write $H\cdot K\in \Rep(C_0(G))$ for the representation $K$ with multiplicity $H$. Note that the underlying Hilbert space of $H\cdot K$ is the Hilbert space tensor product $H\otimes K$. 
\begin{corollary}\label{cor: CTYgivesWTY}
    Using the notation above, the following data defines a $\mathrm{W}^*$-tensor category $\TYcal(G, \chi, \xi)$.
\begin{enumerate}
	\item The underlying $\mathrm{W}^*$-category is $\Rep(C_0(G)\oplus \CC)\cong \Rep(C_0(G))\oplus \Hilb\cdot \tau$;
	\item the tensor functor is given by
	\[
	\upsilon\otimes \nu = \upsilon\times \nu,\hspace{1cm} \upsilon\otimes \tau= L^2(\upsilon)\cdot \tau,\hspace*{1cm} \tau\otimes \upsilon = L^2(\upsilon)\cdot\tau, \hspace{1cm} \tau\otimes \tau = \mu, 
	\]
    for objects $\int^{\oplus}_{t\in T}\delta_{p(t)}\mathrm{d}\upsilon(t)$ and $\int^{\oplus}_{s\in S}\delta_{n(s)}\mathrm{d}\nu(s)$ of $\Rep(C_0(G))$;
	\item the associators are as follows, where we use objects $\int^{\oplus}_{t\in T}\delta_{p(t)}\mathrm{d}\upsilon(t),\, \int^{\oplus}_{s\in S}\delta_{n(s)}\mathrm{d}\nu(s)$ and $\int^{\oplus}_{r\in R}\delta_{z(r)}\mathrm{d}\eta(r)$ of $\Rep(C_0(G)),$
	\[
	\begin{array}{cccc}
		\alpha_{\upsilon, \nu, \eta}:& \upsilon\times\nu\times\eta&\xrightarrow{\id}& \upsilon\times\nu\times \eta\vspace{.3cm}\\
		\alpha_{\tau, \upsilon,\nu} = \alpha_{\upsilon, \nu, \tau} :& L^2(\upsilon\times \nu)\cdot \tau& \xrightarrow{\id} &L^2(\upsilon\times\nu )\cdot \tau\vspace{.3cm}\\
		\alpha_{\upsilon,\tau,\nu}  :& L^2(\upsilon\times \nu)\cdot \tau& \to &L^2(\upsilon\times\nu )\cdot \tau\\ 
		&f(t,s)&\mapsto& \chi\big(p(t), n(s)\big)f(t,s)\vspace{.3cm}\\
		\alpha_{\upsilon,\tau,\tau}:&L^2( \upsilon)\cdot \mu& \to &\upsilon \times\mu\\
		&f(t, x)&\mapsto& f\big(t, p(t)+x\big)\vspace{.3cm}\\
		\alpha_{\tau,\tau, \upsilon}:& \mu\times\upsilon& \to &L^2(\upsilon)\cdot \mu\\
		&f(x,t)&\mapsto& f\big(t,-p(t)+x\big)\vspace{.3cm}\\
		\alpha_{\tau,\upsilon,\tau}:& L^2(\upsilon)\cdot \mu &\to & L^2(\upsilon)\cdot \mu\\
		&f(t,x)&\mapsto &\chi\big(p(t), x\big)\cdot f(t,x)
	\end{array}
	\]
	and 
	\[	
	\alpha_{\tau,\tau,\tau} : L^2(G,\mu)\cdot \tau\xrightarrow{J_\chi}L^2(\hat{G},\hat{\mu})\cdot \tau\xrightarrow{\xi\cdot\mathcal{F}}L^2(G,\mu)\cdot \tau;
	\]
	\item the unitors are the evaluation-on-$e$ maps $\delta_e\times \upsilon\to\upsilon$ and $\upsilon\times \delta_e\to \upsilon$.
\end{enumerate}
\end{corollary}

\subsection[Continuity of Tambara-Yamagami W*-tensor categories]{Continuity of Tambara-Yamagami $\mathrm{W}^*$-tensor categories}
\label{sec: Automaticcontinuity} In this section we will prove that the associators of a $\mathrm{W}^*$-tensor category with a Tambara-Yamagami-like tensor product are automatically continuous. Let $G$ be a locally compact abelian group. 
\begin{definition}\label{def: WTY}
	A \emph{Tambara-Yamagami $\mathrm{W}^*$-tensor category for $G$} is a $\mathrm{W}^*$-tensor category whose underlying $\mathrm{W}^*$-category is $\Rep(C_0(G))\oplus\Hilb\cdot\tau$ and is such that it
    \begin{enumerate}
        \item admits unitary natural isomorphisms
          \[
        \upsilon\otimes\nu \overset{[\upsilon,\nu]}{\cong}  \upsilon\times \nu\hspace{1cm}
	\upsilon\otimes \tau \overset{[\upsilon,\tau]}{\cong} L^2(\upsilon)\cdot\tau\overset{[\tau,\upsilon]}{\cong}  \tau\otimes \upsilon,
        \]
        for any objects $\int^{\oplus}_{t\in T}\delta_{p(t)}\mathrm{d}\upsilon(t),\, \int^{\oplus}_{s\in S}\delta_{n(s)}\mathrm{d}\nu(s)\in\Rep(C_0(G))$;
        \item satisfies that $\tau\otimes\tau\neq 0$ and $\Hom_{\Rep(C_0(G))\oplus \Hilb\cdot \tau}(\tau\otimes \tau, \tau) = 0$.
    \end{enumerate}
\end{definition}
By Proposition \ref{Prop: ForgetE1}, every continuous Tambara-Yamagami category for $G$ induces a Tambara-Yamagami $\mathrm{W}^*$-tensor category for $G$ through $\Ffrak$, by taking the morphisms $[\upsilon,\nu], [\upsilon, \tau]$, and $[\tau, \upsilon]$ to be identities. In a general Tambara-Yamagami $\mathrm{W}^*$-tensor category, the associators are not required to be continuous in the sense that they are not required to be of the form $\Ffrak(u)$ for $u$ a 2-morphism in $\MorC$, and $\tau\otimes\tau$ is not required to be $\int^\oplus_{x\in G}\delta_x\mathrm{d}\mu(x)\in \Rep(C_0(G))$. We will prove, however, that every Tambara-Yamagami $\mathrm{W}^*$-tensor category is equivalent to one coming from a continuous Tambara-Yamagami tensor~category.
\begin{remark}
    Note that the natural isomorphisms $[\upsilon,\nu]$, $[\upsilon, \tau]$ and $[\tau, \upsilon]$ are not required to be compatible with the associators of the induced tensor structure on $\Rep(C_0(G))\oplus \Hilb\cdot \tau$.
\end{remark}

Let us fix a Tambara-Yamagami $\mathrm{W}^*$-tensor category $\Ccal$. Let us denote by $\Ccal_0 : = \Rep(C_0(G))$ the full subcategory of $\Rep(C_0(G))\oplus\Hilb\cdot\tau$ on direct integrals of the invertible simple~objects.

\begin{proposition}\label{Prop: Tau2isL2G}
    There is a nonzero Hilbert space $H$ and an isomorphism of $C_0(G)$-representations
    \[
    [\tau]: \tau\otimes \tau \xrightarrow{\cong} H\cdot \int^\oplus_{x\in G}\delta_{x}\mathrm{d}\mu(x),
    \]
    where the right-hand side denotes the $C_0(G)$-representation $\int^\oplus_{x\in G}\delta_{x}\mathrm{d}\mu(x)$ with multiplicity $H$.
\end{proposition}
\begin{proof}
    By the hypothesis that $\tau\otimes \tau\neq 0$ and  $\Hom_{\Ccal}(\tau\otimes \tau, \tau) = 0$, we have that $\tau\otimes\tau\in \Rep(C_0(G))$ is a nonzero representation. Since the underlying Hilbert space of $\tau\otimes\tau$ is separable, its spectral measure class admits a finite regular Borel measure. Equivalently, there is a nonzero finite regular Borel measure $\nu$ on $G$ and a $\nu$-measurable Hilbert bundle $K = \{K_x\}_{x\in G}$ (which we may assume satisfies that $K_x\neq 0$ $\nu$-almost everywhere) such that
    \begin{equation}\label{eq: tautau}
    \tau\otimes\tau\cong \int^\oplus_{x\in G}K_x\mathrm{d}\nu(x)\neq 0,
    \end{equation}
    where the right-hand side denotes the space of square-integrable sections of $K$ with $C_0(G)$-action given by pointwise multiplication. By associativity, for every element $y\in G$ we have
    \[
    \int^\oplus_{x\in G} K_x\mathrm{d}\nu(x) = (\delta_y\otimes \tau)\otimes \tau\cong \delta_y\otimes (\tau\otimes\tau) = \int^\oplus_{x\in G}K_{x-y}\mathrm{d}(L_y)_*\nu(x)
    .\]
    for $L_y: G\to G$ the map given by $L_y(x) = x+y$. Equivalently, for every $y\in G$, we have that 
    \begin{enumerate}
   \item $(L_y)_*\nu$ and $\nu$ are equivalent measures; 
   \item$\dim(K_x) = \dim(K_{x-y})$ for almost every $x\in G$. 
   \end{enumerate}
   We first show that $\nu$ is equivalent to the Haar measure $\mu$. Let $h\in L^1(G, \mu)\cap L^\infty(G, \mu)$ be a function taking values in $(0, \infty)$. Then, we may define a new measure $\bar{\nu}$ on $G$ by
    \[
    \bar{\nu}(A) := \int_{x\in G}(L_x)_*\nu(A)h(x)\mathrm{d}\mu(x) = \int_{y\in A}\Big(\int_{x\in G} h(y-x)\mathrm{d}\nu(x)\Big)\mathrm{d}\mu(y)
    \]
    for any Borel $A\subset G$. Note that we have used Tonelli's Theorem in the equality above. Therefore, 
    \[
    \mathrm{d}\bar{\nu}(y) = \Big(\int_{x\in G}h(y-x)\mathrm{d}\nu(x)\Big)\mathrm{d}\mu(y),
    \]
    and the integral $\int_{x\in G}h(y-x)\mathrm{d}\nu(x)$ is strictly positive because $\nu$ is nonzero and $h$ is strictly positive. Therefore, $\bar{\nu}$ is equivalent to the Haar measure $\mu$. On the other hand, the conclusion $(\mathrm{i})$ above implies that $\bar{\nu}$ and $\nu$ are equivalent. Indeed, if $A\subset G$ is a subset with $\nu(A) = 0$, then any translate of ${\nu}$ vanishes on $A$, and hence $\bar{\nu}(A) = 0$. If $\nu(A)>0$, then $(L_x)_*\nu(A)>0$ for every $x\in G$, and integration against the strictly positive $h$ gives $\bar{\nu}(A)>0$. Therefore, $\nu$ is equivalent to the Haar measure $\mu$, and we can replace $\nu$ by $\mu$.

    Now, for every $n\in\mathbb{N}\cup\{\infty\}$, we write $A_n : = \{x\in G\,|\, \dim(K_x) = n\}$. By the consequence $(\mathrm{ii})$ above, for any $n\in\mathbb{N}\cup\{\infty\}$ and any $x\in G$, it holds that $\mu((A_n + x)\Delta A_n) = 0$, for $\Delta$ the symmetric difference. Any measurable subset of $G$ invariant under the action of $G$ on itself up to $\mu$-null subsets is either $\mu$-null or $\mu$-conull. Hence, each $A_n$ is either null or conull with respect to $\mu$. Since the subsets $A_n$ form a countable measurable partition of $G$, there exists exactly one $n_0\in \mathbb{N}\cup\{\infty\}$ for which
    \[
    \dim(K_x) = n_0
    \]
    for almost every $x\in G$. Since $\tau\otimes \tau\neq 0$, we have that $n_0\neq 0$. Let $H$ be any separable Hilbert space of dimension $n_0$. Since any measurable field of separable Hilbert spaces of constant dimension over a standard measure space is measurably trivial, we have that
    \[
    K_x\cong H
    \]
    for almost every $x\in G$. Thus, the right-hand side of Equation \eqref{eq: tautau} is
    \[
    H\cdot \int^{\oplus}_{x\in G}\delta_x\mathrm{d}\mu(x).
    \]
\end{proof}
Throughout this section, we use objects $\int^{\oplus}_{t\in T}\delta_{p(t)}\mathrm{d}\upsilon(t),\, \int^{\oplus}_{s\in S}\delta_{n(s)}\mathrm{d}\nu(s),\, \int^{\oplus}_{r\in R}\delta_{z(r)}\mathrm{d}\eta(r)\in \Rep(C_0(G))$, and denote them by $\upsilon,\,\nu$ and $\eta$ respectively. We reserve the notation $L^2(\upsilon)$ for the object of $\Hilb$ (without the action of $C_0(G)$) and use the notation $K\cdot \upsilon = K\cdot \int^{\oplus}_{t\in T}\delta_{p(t)}\mathrm{d}\upsilon(t)$ to denote the  $C_0(G)$-representation $\int^{\oplus}_{t\in T}\delta_{p(t)}\mathrm{d}\upsilon(t)$ with multiplicity $K\in \Hilb$. 

\begin{remark}
    By Proposition \ref{Prop: FullSubcat}, the definition of a Tambara-Yamagami $\mathrm{W}^*$-tensor category constrains the tensor product of all objects in the category. In particular, $\upsilon \otimes \Big(H\cdot \int^{\oplus}_{x\in G}\delta_x\mathrm{d}\mu(x)\Big) \cong H\cdot \int^{\oplus}_{(t,x)\in T\times G}\delta_{p(t) + x}\mathrm{d}(\upsilon\times\mu)(t,x)$ via $[\upsilon, \mu]$ for every $\int^{\oplus}_{t\in T}\delta_{p(t)}\mathrm{d}\upsilon(t)\in \Rep(C_0(G))$. 
\end{remark}

Choose a pair of a Hilbert space $H$ and an isomorphism $\tau\otimes\tau\xrightarrow[\cong]{[\tau]} H\cdot \int^{\oplus}_{x\in G}\delta_{x}\mathrm{d}\mu(x)$, which exists by Proposition \ref{Prop: Tau2isL2G}. We will show that $G$ is self-Pontryagin dual, that $H\cong \CC$  and that $\Ccal \cong \TYcal(G,\chi,\xi)$ for some continuous symmetric nondegenerate bicharacter $\chi$ on $G$ and a sign $\xi\in \{1,-1\}$. To do so, we will find suitable bases of the various morphism spaces involved to reshape the associators of $\Rep(C_0(G))\oplus\Hilb\cdot\tau$.  By definition, there are natural isomorphisms% https://q.uiver.app/#q=WzAsNixbMywwLCJcXENjYWxfMCJdLFs1LDAsIlxcSGlsYiJdLFswLDAsIlxcQ2NhbF8wXFx0aW1lc1xcQ2NhbF8wIl0sWzIsMCwiXFxDY2FsXzAiXSxbNiwwLCJcXENjYWxfMCJdLFs4LDAsIlxcSGlsYiJdLFswLDEsIlxcSG9tX3tcXENjYWx9KFxcdGF1LCBcXHRhdVxcb3RpbWVzLSkiLDAseyJjdXJ2ZSI6LTN9XSxbMCwxLCJGIiwyLHsiY3VydmUiOjN9XSxbMiwzLCItXFxvdGltZXMtIiwwLHsiY3VydmUiOi0zfV0sWzIsMywiLVxcdGltZXMgLSIsMix7ImN1cnZlIjozfV0sWzQsNSwiXFxIb21fe1xcQ2NhbH0oXFx0YXUsIC1cXG90aW1lc1xcdGF1KSIsMCx7ImN1cnZlIjotM31dLFs0LDUsIkYiLDIseyJjdXJ2ZSI6M31dLFs2LDcsIltcXHRhdSwtXSIsMCx7InNob3J0ZW4iOnsic291cmNlIjoyMCwidGFyZ2V0IjoyMH19XSxbMTAsMTEsIlstLFxcdGF1XSIsMCx7InNob3J0ZW4iOnsic291cmNlIjoyMCwidGFyZ2V0IjoyMH19XSxbOCw5LCJbLSwtXSIsMCx7InNob3J0ZW4iOnsic291cmNlIjoyMCwidGFyZ2V0IjoyMH19XV0=
\[\begin{tikzcd}
	{\Ccal_0\times\Ccal_0} && {\Ccal_0} & {\Ccal_0} && \Hilb & {\Ccal_0} && \Hilb,
	\arrow[""{name=0, anchor=center, inner sep=0}, "{-\otimes-}", curve={height=-18pt}, from=1-1, to=1-3]
	\arrow[""{name=1, anchor=center, inner sep=0}, "{-\times -}"', curve={height=18pt}, from=1-1, to=1-3]
	\arrow[""{name=2, anchor=center, inner sep=0}, "{\Hom(\tau, \tau\otimes-)}", curve={height=-18pt}, from=1-4, to=1-6]
	\arrow[""{name=3, anchor=center, inner sep=0}, "\text{Forget}"', curve={height=18pt}, from=1-4, to=1-6]
	\arrow[""{name=4, anchor=center, inner sep=0}, "{\Hom(\tau, -\otimes\tau)}", curve={height=-18pt}, from=1-7, to=1-9]
	\arrow[""{name=5, anchor=center, inner sep=0}, "\text{Forget}"', curve={height=18pt}, from=1-7, to=1-9]
	\arrow["{[-,-]}", shorten <=5pt, shorten >=5pt, Rightarrow, from=0, to=1]
	\arrow["{[\tau,-]}", shorten <=5pt, shorten >=5pt, Rightarrow, from=2, to=3]
	\arrow["{[-,\tau]}", shorten <=5pt, shorten >=5pt, Rightarrow, from=4, to=5]
\end{tikzcd}\]
where $\text{Forget}: \Rep(C_0(G))\to \Hilb$ is the functor that forgets the $C_0(G)$-action. We will refer to these, together with $[\tau]$, as \emph{coordinates}. Using the coordinates above we can define the following maps at the bottom of every square (we drop the tensor product symbols for~readability).

\[\begin{tikzcd}[column sep=small]
	{(\upsilon\nu)\eta} && {\upsilon(\nu\eta)} && {(\tau\tau)\tau} && {\tau(\tau\tau)} \\
	{(\upsilon\times \nu)\eta} && {\upsilon(\nu\times \eta)} && {(H\cdot \mu)\tau} && {\tau (H\cdot \mu)} \\
	{\upsilon\times\nu\times \eta} && {\upsilon\times\nu\times\eta} && {L^2(\mu, H)\cdot\tau} && {L^2(\mu, H)\cdot\tau} \\
	\\
	{(\tau\upsilon)\nu} && {\tau(\upsilon\nu)} && {(\upsilon\tau)\nu} && {\upsilon(\tau\nu)} \\
	{L^2(\upsilon)\cdot(\tau\nu)} && {\tau(\upsilon\times\nu)} && {L^2(\upsilon)\cdot(\tau\nu)} && {L^2(\nu)\cdot(\upsilon\tau)} \\
	{L^2(\upsilon)\otimes L^2(\nu)\cdot\tau} &&&& {L^2(\upsilon)\otimes L^2(\nu)\cdot \tau} && {L^2(\upsilon)\otimes L^2(\nu)\cdot \tau} \\
	{L^2(\upsilon\times\nu)\cdot \tau} && {L^2(\upsilon\times\nu)\cdot \tau} && {L^2(\upsilon\times\nu)\cdot \tau} && {L^2(\upsilon\times\nu)\cdot \tau} \\
	\\
	{(\upsilon\nu)\tau} && {\upsilon(\nu\tau)} && {(\upsilon\tau)\tau} && {\upsilon(\tau\tau)} \\
	{(\upsilon\times\nu)\tau} && {L^2(\nu)\cdot(\upsilon\tau)} && {L^2(\upsilon)\cdot(\tau\tau)} && {H\cdot (\upsilon\mu)} \\
	&& {L^2(\upsilon)\otimes L^2(\nu)\cdot\tau} && {L^2(\upsilon)\cdot(H\cdot \mu)} \\
	{L^2(\upsilon\times\nu)\cdot \tau} && {L^2(\upsilon\times\nu)\cdot\tau} && {L^2(\upsilon, H)\cdot \mu} && {H\cdot (\upsilon\times\mu)}
	\arrow["{\alpha_{\upsilon, \nu, \eta}}", from=1-1, to=1-3]
	\arrow["{[\upsilon, \nu]\id}"', from=1-1, to=2-1]
	\arrow["{\id[\nu, \eta]}", from=1-3, to=2-3]
	\arrow["{\alpha_{\tau,\tau,\tau}}", from=1-5, to=1-7]
	\arrow["{[\tau]\id}"', from=1-5, to=2-5]
	\arrow["{\id[\tau]}", from=1-7, to=2-7]
	\arrow["{[\upsilon\times\nu, \eta]}"', from=2-1, to=3-1]
	\arrow["{[\upsilon, \nu\times\eta]}", from=2-3, to=3-3]
	\arrow["{[\mu,\tau]}"', from=2-5, to=3-5]
	\arrow["{[\tau,\mu]}", from=2-7, to=3-7]
	\arrow["{\alpha(\upsilon, \nu, \eta)}"', from=3-1, to=3-3]
	\arrow["{\gamma\cdot\id}"', from=3-5, to=3-7]
	\arrow["{\alpha_{\tau, \upsilon, \nu}}", from=5-1, to=5-3]
	\arrow["{[\tau, \upsilon]\id}"', from=5-1, to=6-1]
	\arrow["{\id[\upsilon, \nu]}", from=5-3, to=6-3]
	\arrow["{\alpha_{\upsilon, \tau, \nu}}", from=5-5, to=5-7]
	\arrow["{[\upsilon,\tau]\id}"', from=5-5, to=6-5]
	\arrow["{\id[\tau,\nu]}", from=5-7, to=6-7]
	\arrow["{[\tau, \nu]}"', from=6-1, to=7-1]
	\arrow["{[\tau, \upsilon\times\nu]}", from=6-3, to=8-3]
	\arrow["{[\tau,\nu]}"', from=6-5, to=7-5]
	\arrow["{[\upsilon,\tau]}", from=6-7, to=7-7]
	\arrow[from=7-1, to=8-1]
	\arrow[from=7-5, to=8-5]
	\arrow[from=7-7, to=8-7]
	\arrow["{\alpha_1(\upsilon, \nu)\cdot\id}"', from=8-1, to=8-3]
	\arrow["{\alpha_2(\upsilon,\nu)\cdot\id}"', from=8-5, to=8-7]
	\arrow["{\alpha_{\upsilon, \nu, \tau}}", from=10-1, to=10-3]
	\arrow["{[\upsilon, \nu]\id}"', from=10-1, to=11-1]
	\arrow["{\id[\nu, \tau]}", from=10-3, to=11-3]
	\arrow["{\alpha_{\upsilon, \tau, \tau}}", from=10-5, to=10-7]
	\arrow["{[\upsilon, \tau]\id}"', from=10-5, to=11-5]
	\arrow["{\id[\tau]}", from=10-7, to=11-7]
	\arrow["{[\upsilon\times\nu, \tau]}"', from=11-1, to=13-1]
	\arrow["{[\upsilon, \tau]}", from=11-3, to=12-3]
	\arrow["{[\tau]}"', from=11-5, to=12-5]
	\arrow["{[\upsilon, \mu]}", from=11-7, to=13-7]
	\arrow[from=12-3, to=13-3]
	\arrow[from=12-5, to=13-5]
	\arrow["{\alpha_3(\upsilon, \nu)\cdot \id}"', from=13-1, to=13-3]
	\arrow["{\beta_1(\upsilon)}"', from=13-5, to=13-7]
\end{tikzcd}\]

% https://q.uiver.app/#q=WzAsMTUsWzQsMCwiKFxcdGF1XFx0YXUpXFx1cHNpbG9uIl0sWzYsMCwiXFx0YXUoXFx0YXVcXHVwc2lsb24pIl0sWzQsMSwiSFxcY2RvdCAoXFxtdVxcdXBzaWxvbikiXSxbNiwxLCJMXjIoXFx1cHNpbG9uKVxcY2RvdChcXHRhdVxcdGF1KSJdLFs0LDMsIkhcXGNkb3QgKFxcbXVcXHRpbWVzXFx1cHNpbG9uKSJdLFs2LDIsIkxeMihcXHVwc2lsb24pXFxjZG90IChIXFxjZG90IFxcbXUpIl0sWzYsMywiTF4yKFxcdXBzaWxvbiwgSClcXGNkb3QgXFxtdSJdLFswLDAsIihcXHRhdVxcdXBzaWxvbilcXHRhdSJdLFsyLDAsIlxcdGF1KFxcdXBzaWxvblxcdGF1KSJdLFswLDEsIkxeMihcXHVwc2lsb24pXFxjZG90KFxcdGF1XFx0YXUpIl0sWzIsMSwiTF4yKFxcdXBzaWxvbilcXGNkb3QoXFx0YXVcXHRhdSkiXSxbMCwyLCJMXjIoXFx1cHNpbG9uKVxcY2RvdCAoSFxcY2RvdCBcXG11KSJdLFsyLDIsIkxeMihcXHVwc2lsb24pXFxjZG90IChIXFxjZG90IFxcbXUpIl0sWzAsMywiTF4yKFxcdXBzaWxvbiwgSClcXGNkb3QgXFxtdSJdLFsyLDMsIkxeMihcXHVwc2lsb24sIEgpXFxjZG90IFxcbXUiXSxbMCwxLCJcXGFscGhhX3tcXHRhdSxcXHRhdSxcXHVwc2lsb259Il0sWzAsMiwiW1xcdGF1XVxcaWQiLDJdLFsxLDMsIlxcaWRbXFx0YXUsXFx1cHNpbG9uXSJdLFsyLDQsIltcXG11LFxcdXBzaWxvbl0iLDJdLFszLDUsIltcXHRhdV0iXSxbNCw2LCJcXGJldGFfMyhcXHVwc2lsb24pIiwyXSxbNSw2XSxbNyw4LCJcXGFscGhhX3tcXHRhdSwgXFx1cHNpbG9uLFxcdGF1fSJdLFs3LDksIltcXHRhdSxcXHVwc2lsb25dXFxpZCIsMl0sWzgsMTAsIlxcaWRbXFx1cHNpbG9uLFxcdGF1XSJdLFs5LDExLCJbXFx0YXVdIiwyXSxbMTAsMTIsIltcXHRhdV0iXSxbMTEsMTNdLFsxMiwxNF0sWzEzLDE0LCJcXGJldGFfMihcXHVwc2lsb24pIiwyXV0=
\[\begin{tikzcd}[column sep=2em]
	{(\tau\upsilon)\tau} && {\tau(\upsilon\tau)} && {(\tau\tau)\upsilon} && {\tau(\tau\upsilon)} \\
	{L^2(\upsilon)\cdot(\tau\tau)} && {L^2(\upsilon)\cdot(\tau\tau)} && {H\cdot (\mu\upsilon)} && {L^2(\upsilon)\cdot(\tau\tau)} \\
	{L^2(\upsilon)\cdot (H\cdot \mu)} && {L^2(\upsilon)\cdot (H\cdot \mu)} &&&& {L^2(\upsilon)\cdot (H\cdot \mu)} \\
	{L^2(\upsilon, H)\cdot \mu} && {L^2(\upsilon, H)\cdot \mu} && {H\cdot (\mu\times\upsilon)} && {L^2(\upsilon, H)\cdot \mu}
	\arrow["{\alpha_{\tau, \upsilon,\tau}}", from=1-1, to=1-3]
	\arrow["{[\tau,\upsilon]\id}"', from=1-1, to=2-1]
	\arrow["{\id[\upsilon,\tau]}", from=1-3, to=2-3]
	\arrow["{\alpha_{\tau,\tau,\upsilon}}", from=1-5, to=1-7]
	\arrow["{[\tau]\id}"', from=1-5, to=2-5]
	\arrow["{\id[\tau,\upsilon]}", from=1-7, to=2-7]
	\arrow["{[\tau]}"', from=2-1, to=3-1]
	\arrow["{[\tau]}", from=2-3, to=3-3]
	\arrow["{[\mu,\upsilon]}"', from=2-5, to=4-5]
	\arrow["{[\tau]}", from=2-7, to=3-7]
	\arrow[from=3-1, to=4-1]
	\arrow[from=3-3, to=4-3]
	\arrow[from=3-7, to=4-7]
	\arrow["{\beta_2(\upsilon)}"', from=4-1, to=4-3]
	\arrow["{\beta_3(\upsilon)}"', from=4-5, to=4-7]
\end{tikzcd}\]

A priori, these are all unitary maps of Hilbert spaces. Chasing the $C_0(G)$-actions along the diagrams, we realize that, for any $\phi\in C_0(G)$,
\begin{align}
    \alpha(\upsilon,\nu,\eta)\Big(\phi\big(p(t)+n(s)+z(r)\big)f(t,s,r)\Big) &= \phi\big(p(t)+n(s)+z(r)\big)\alpha(\upsilon,\nu,\eta)\big(f(t,s,r)\big)\\
	\beta_1(\upsilon)\big(\phi(x)f(t,x)\big) &= \label{eq: ChaseC0ActionFirst} \phi\big(p(t)+x\big)\beta_1(\upsilon)\big(f(t,x)\big)\\
	\beta_2(\upsilon)\big(\phi(x)f(t,x)\big) &= \phi(x)\beta_2(\upsilon)\big(f(t,x)\big)\\
	\beta_3(\upsilon)\big(\phi(x+p(t))f(x,t)\big) &= \phi(x)\beta_3(\upsilon)\big(f(x,t)\big).\label{eq: ChaseC0ActionLast}
\end{align}
These relations, together with the naturality of the tensor product, constrain the form of the operators involved. For an object $\int^{\oplus}_{c\in C}\delta_{q(c)}\mathrm{d}\rho(c)\in \Rep(C_0(G))$ and $K$ a Hilbert space, we define the shift operators
\[
\begin{array}{cccc}
	\sigma_L: &L^2(\rho\times \mu, K) &\to &L^2(\rho\times \mu, K)\\
	&f(c,x)&\mapsto &f\big(c,x+q(c)\big)
\end{array}\hspace{2cm}
\begin{array}{cccc}
	\sigma_R: &L^2(\mu\times \rho, K) &\to &L^2(\mu \times\rho, K)\\
	&f(x, c)&\mapsto &f\big(x+q(c), c\big).
\end{array}
\]
Whenever it is clear from the context, we will drop the subscripts $L$ or $R$ from the notation, but they are relevant whenever $\int^{\oplus}_{c\in C}\delta_{q(c)}\mathrm{d}\rho(c) = \int^{\oplus}_{x\in G}\delta_{x}\mathrm{d}\mu(x)$. We need the following well-known~result.

\begin{lemma}\label{lemm: CommutantL2}
 Let $X$ be a topological space with a Borel measure $\upsilon$, and $K$ a Hilbert space. Then, on $L^2(X,\upsilon, K)$, we have $C_0(X)' =  L^\infty(X,\upsilon,  B(K))$ and $C_0(X)'' = L^\infty(X,\upsilon).$
 \end{lemma}

By the previous lemma, given $p:T\to X$ a continuous map and $\upsilon$ a measure on $T$, it holds that, as operators on $L^2(\upsilon)$, we have inclusions $C_0(X)''\subset C_0(T)'' = L^\infty(\upsilon).$

\begin{lemma}\label{lemm: CharacterizationCommuting}
	Let $\int^{\oplus}_{t\in T}\delta_{p(t)}\mathrm{d}\upsilon(t),\, \int^{\oplus}_{s\in S}\delta_{n(s)}\mathrm{d}\nu(s),\, \int^{\oplus}_{r\in R}\delta_{z(r)}\mathrm{d}\eta(r)\in \Ccal_0$. Then, there exist functions
	\begin{align*}
		a(\upsilon, \nu, \eta)&\in L^\infty(T\times S\times R, \upsilon\times\nu\times \eta, U(1))\\
		a_1(\upsilon, \nu), a_2(\upsilon,\nu), a_3(\upsilon, \nu)&\in L^\infty(T\times S, \upsilon\times\nu, U(1))\\
		b_1(\upsilon), b_2(\upsilon), b_3(\upsilon)&\in L^\infty\big(T\times G, \upsilon\times \mu, U(H)\big)
	\end{align*}
	such that $\alpha(\upsilon,\nu,\eta), \alpha_i(\upsilon, \nu),\sigma^{-1}\circ\beta_1(\upsilon), \beta_2(\upsilon), \sigma\circ\beta_3(\upsilon)$ are given by multiplication by the corresponding function.
\end{lemma}
\begin{proof}
	By naturality of the coordinates and the associators, the operators $\alpha(\upsilon,\nu,\eta)$ and $\alpha_i(\upsilon, \nu)$ and $\sigma^{-1}\circ\beta_1(\upsilon), \beta_2(\upsilon), \sigma\circ\beta_3(\upsilon)$ lie in the double commutant of $C_0(G^3)$, $C_0(G^2)$ or $C_0(G)$, respectively. In addition, by Equations \eqref{eq: ChaseC0ActionFirst} to \eqref{eq: ChaseC0ActionLast}, $\sigma^{-1}\circ\beta_1(\upsilon), \beta_2(\upsilon), \sigma\circ\beta_3(\upsilon)$ lie in the commutant of $C_0(G)$ acting on the entry with the Haar measure. Hence, all these operators are given by multiplication by an $L^\infty$-function. Since the associators are required to be unitary by definition of a $\mathrm{W}^*$-tensor category, they take values in $U(1)$ or $U(H)$.
\end{proof}

In terms of the morphisms defined in Lemma \ref{lemm: CharacterizationCommuting}, the pentagon equations read as follows. The following are equalities of $L^\infty$-functions with respect to the obvious measures, or equalities of operators on $L^2$-spaces. In order to avoid cluttering the notation, we write $x+t$ in place of $x+p(t)$ for $x\in G$ and $t\in T$. We also identify $\mathbb{\CC}$ with its image in $B(H)$. We take another object $\int^\oplus_{c\in C}\delta_{q(c)}\mathrm{d}{\rho}(c)\in \Rep(C_0(G))$.

\begin{align}
%1
\label{eq: Pentagon0}
&a(\nu, \eta, \rho)(s,r,v)\cdot
a(\upsilon, \nu\times\eta, \rho)(t,s,r,v)\cdot
a(\upsilon,\nu,\eta)(t,s,r) \nonumber\\
&\qquad=
a(\upsilon,\nu,\eta\times\rho)(t,s,r,v)\cdot
a(\upsilon\times\nu,\eta,\rho)(t,s,r,v)
\\[0.3em]
%2
\label{eq: Pentagon1.4}
&a_3(\nu,\eta)(s,r)\cdot
a_3(\upsilon,\nu\times\eta)(t,s,r)\cdot
a(\upsilon,\nu,\eta)(t,s,r)=
a_3(\upsilon,\nu)(t,s)\cdot
a_3(\upsilon\times\nu,\eta)(t,s,r)
\\[0.3em]
%3
\label{eq: Pentagon1.1}
&a(\upsilon,\nu,\eta)(t,s,r)\cdot
a_1(\upsilon\times\nu,\eta)(t,s,r)\cdot
a_1(\upsilon,\nu)(t,s)=
a_1(\upsilon,\nu\times\eta)(t,s,r)\cdot
a_1(\nu,\eta)(s,r)
\\[0.3em]
%4
\label{eq: Pentagon1.2}
&a_2(\upsilon,\eta)(t,r)\cdot
a_2(\upsilon,\nu)(t,s)
=
a_2(\upsilon,\nu\times\eta)(t,s,r)
\\
%5
\label{eq: Pentagon1.3}
&a_2(\nu,\eta)(s,r)\cdot
a_2(\upsilon,\eta)(t,r)
=
a_2(\upsilon\times\nu,\eta)(t,s,r)
\\[0.3em]
%6
\label{eq: Pentagon2.34}
&a(\upsilon,\nu,\mu)(t,s,-s-t+x)\circ
b_1(\upsilon\times\nu)(t,s,x)\nonumber\\
&\qquad=
b_1(\nu)(s,-t+x)\circ
b_1(\upsilon)(t,x)\circ
a_3(\upsilon,\nu)(t,s)
\\[0.3em]
%7
\label{eq: Pentagon2.12}
&a_1(\upsilon,\nu)(t,s)\circ
b_3(\nu)(s,t+x)\circ
b_3(\upsilon)(t,x)=
b_3(\upsilon\times\nu)(t,s,x)\circ
a(\mu,\upsilon,\nu)(x,t,s)
\\
%8
\label{eq: Pentagon2.24}
&b_2(\nu)(s,x)\circ
a_2(\upsilon,\nu)(t,s)
=
b_2(\nu)(s,t+x)
\\
%9
\label{eq: Pentagon2.13}
&a_2(\upsilon,\nu)(t,s)\circ
b_2(\upsilon)(t,x-s)
=
b_2(\upsilon)(t,x)
\\[0.3em]
%10
\label{eq: Pentagon2.23}
&b_3(\nu)(s,-t+x)\circ
a(\upsilon,\mu,\nu)(t,-t+x,s)\circ
b_1(\upsilon)(t,x)=
b_1(\upsilon)(t,x+s)\circ
b_3(\nu)(s,x)
\\
%11
\label{eq: Pentagon2.14}
&a_3(\upsilon,\nu)(t,s)\circ
b_2(\upsilon\times\nu)(t,s,x)\circ
a_1(\upsilon,\nu)(t,s)=
b_2(\upsilon)(t,x)\circ
b_2(\nu)(s,x)
\\[0.3em]
%12
\label{eq: Pentagon3.1}
&\big[\id\otimes\gamma\big]\circ
\big[\alpha_3(\upsilon,\mu)\big]\circ
\beta_1(\upsilon)
=
\alpha_2(\upsilon,\mu)\circ
\big[\id\otimes\gamma\big]\\
%13
\label{eq: Pentagon3.4}
&\beta_3(\upsilon)\circ
\alpha_1(\mu,\upsilon)\circ
\big[\id\otimes\gamma\big]
=
\big[\id\otimes\gamma\big]\circ
\alpha_2(\mu,\upsilon)\\
%14
\label{eq: Pentagon3.2}
&\beta_1(\upsilon)\circ
\big[\id\otimes\gamma\big]\circ
\beta_2(\upsilon)
=
\alpha_1(\upsilon,\mu)\circ
\big[\id\otimes\gamma\big]
\\
%15
\label{eq: Pentagon3.3}
&\beta_2(\upsilon)\circ
\big[\id\otimes\gamma\big]\circ
\beta_3(\upsilon)
=
\big[\id\otimes\gamma\big]\circ
\alpha_3(\mu,\upsilon).
\end{align}
Denoting by $\sigma_H$ the endomorphism of $L^2(\mu\times \mu, H\otimes H)$ given by postcomposition with the swap map 
\[
\begin{array}{ccc}
    H\otimes H & \to &H\otimes H  \\
    h_1\otimes h_2&\mapsto &h_2\otimes h_1,
\end{array}
\]
the missing pentagon equation reads, for any $f\in L^2(\mu\times\mu, H\otimes H)$,
\begin{align}
    	%16
	\label{eq: Pentagon4}
	\big[\gamma\otimes \id \big]&\Big([\id \otimes b_2(\mu)(x,y)]\circ \big[\gamma\otimes \id\big](f)(x,y)\Big)\\ & \nonumber =\sigma_H\circ [b_3(\mu)(x, -x+y)\otimes \id]\circ[\id\otimes b_1(\mu)(-x+y, y)]f(-x+y,y).
\end{align}

We now describe the effect of changing coordinates in the functions $a, a_i$ and $b_i$ for $i  =1,2,3$. Define new natural isomorphisms  $[\upsilon, \nu]', [\upsilon, \tau]',  [\tau, \upsilon]', [\tau]' $ given by
\begin{align*}
	[-, -]' = & \theta \circ [-, -]\\
	[-, \tau]' = & \varphi\circ [-,\tau]\\ 
	[\tau, -]'  = & \psi\circ [\tau, -]\\
	[\tau]' = & \omega\circ  [\tau],
\end{align*}
where we have introduced natural isomorphisms $\theta: -\times-\overset{\cong}{\to}-\times-$ and $\varphi,\psi: \text{Forget}\xrightarrow{\cong}\text{Forget}$, as well as $\omega: L^2(\mu, H)\xrightarrow{\cong}L^2(\mu, H)$. The new primed morphisms induce primed operators $\alpha', \alpha_i', \beta_i', \gamma'$ and new functions $a', a_i', b_i'$. Let the transformations $\theta$, $\phi$, $\psi$ be given by multiplication by a $U(1)$-valued $L^\infty$-function which we denote by the same symbol. Also, $\omega$ is given by multiplication by an $L^\infty$-function $G\to U(H)$. Then, the coefficients of the associators get changed to

\begin{align}
	%17
	\label{eq: ChangeCoordinates0}
	\theta(\nu, \eta)(s,r)\cdot \theta(\upsilon , \nu\times \eta)&(t, s, r)\cdot a(\upsilon, \nu, \eta)(t,s,r) \nonumber\\ 	 = a'(\upsilon, \nu, \eta)(t,s,r)&\cdot \theta(\upsilon, \nu)(t,s)\cdot \theta(\upsilon\times \nu, \eta)(t,s,r)\\
	%18
	\label{eq: ChangeCoordinates1.1}
	\theta(\upsilon, \nu)(t,s)\cdot \psi(\upsilon\times\nu)(t,s)\cdot a_1(\upsilon, \nu)(t,s) & = a_1'(\upsilon, \nu)(t,s)\cdot \psi(\upsilon)(t)\cdot \psi(\nu)(s)\\
	%19
	\label{eq: ChangeCoordinates1.2}
	\alpha_2(\upsilon, \nu) &= \alpha_2'(\upsilon, \nu)\\
	%20
	\label{eq: ChangeCoordinates1.3}
	\big[\varphi(\nu)\circ p_2\big]\circ \big[\varphi(\upsilon)\circ p_1\big]\circ \alpha_3(\upsilon, \nu) &= \alpha_3'(\upsilon, \nu)\circ\theta(\upsilon, \nu)\circ \varphi(\upsilon\times\nu)\\
	%21
	\label{eq: ChangeCoordinates2.1}
	\big[\omega\circ p_2\big]\circ \theta(\upsilon,\mu)\circ \beta_1(\upsilon)&=\beta_1'(\upsilon)\circ \big[\varphi(\upsilon)\circ p_1\big]\circ\big[\omega\circ p_2\big]\\
	%22
	\label{eq: ChangeCoordinates2.2}
	\big[\varphi(\upsilon)\circ p_1\big]\circ \big[\omega\circ p_2\big]\circ \beta_2(\upsilon)&=\beta_2'(\upsilon)\circ \big[\omega\circ p_2\big]\circ \big[\psi(\upsilon)\circ p_1\big]\\
	%23
	\label{eq: ChangeCoordinates2.3}
	\big[\psi(\upsilon)\circ p_1\big]\circ \big[\omega\circ p_2\big]\circ \beta_3(\upsilon) &= \beta_3'(\upsilon)\circ \big[\omega\circ p_2\big]\circ \theta(\mu, \upsilon)\\
	%24
	\label{eq: ChangeCoordinates3}
	\omega\circ \psi(\mu)\circ \gamma &= \gamma'\circ \omega\circ \varphi(\mu).
\end{align}

The rest of this section is devoted to showing that we can change the coordinates so that the operators defining the associators of $\Ccal$ have the same form as one of the categories $\TYcal(G, \chi,\xi)$ defined above. To do this, we have to carefully pick the functions $\theta, \psi, \varphi, \omega$. We will do this in different steps. Our proof strategy is similar to that in \cite{TY}.  First, we pick a non-trivial $\theta$ and trivial $\varphi, \psi$, and $\omega$ and compute the form of some of the new associators given the new set of coordinates, in Proposition \ref{prop: FirstStep}. We then start from the new set of coordinates and modify them again by picking a non-trivial function $\omega$ and trivial functions for the rest of $\theta, \varphi$, and $\psi$. In Proposition \ref{prop: SecondStep} we describe the new form of the associators after this second change of coordinates. The new form of the associators allows us to prove in Theorem \ref{thm: HisC} that the Hilbert space $H$ is one-dimensional, and also to identify the continuous symmetric nondegenerate bicharacter $\chi$. We then perform a last step by making a new change of coordinates by choosing a non-trivial $\psi$, but trivial $\theta, \varphi$, and $\omega$. This allows us to identify the sign $\xi\in\{\pm1\}$, and in Theorem \ref{thm: ThirdStep}, we show that this last change of coordinates produces exactly the associators of $\TYcal(G, \chi,\xi)$. The main difference from the finite setting is that, in most cases, we work with $L^\infty$-functions on non-Dirac measures. Therefore, we cannot make pointwise arguments, and we have to find alternative ways to obtain the needed changes of coordinates. Continuity comes into the picture by the fact that measurable group homomorphisms are necessarily continuous \cite{sasvari}. We defer some technical results to Appendix \ref{App A} in order not to overcrowd this section.

Let us perform the first change of coordinates, as described above.

\begin{proposition}\label{prop: FirstStep}
	There exists a change of coordinates for which, for all $\int^{\oplus}_{t\in T}\delta_{p(t)}\mathrm{d}\upsilon(t)$, $\int^{\oplus}_{s\in S}\delta_{n(s)}\mathrm{d}\nu(s)$ and $\int^{\oplus}_{r\in R}\delta_{z(r)}\mathrm{d}\eta(r)$ in $\Ccal_0$,
	\[
	a(\upsilon,\nu, \eta)\equiv 1, \hspace{1cm} a_3(\upsilon, \nu)\equiv 1.
	\]
\end{proposition}
\begin{proof}
	We can pick the change of coordinates
\[
\theta(\upsilon, \nu) := a_3(\upsilon, \nu) \hspace{1cm} \varphi(\upsilon) := \id, \hspace{1cm}\psi(\upsilon) := \id, \hspace*{1cm}\omega := \id.
\]
By Equation \eqref{eq: ChangeCoordinates1.3}, we obtain $\alpha_3'(\upsilon, \nu) = 1$ and by Equations \eqref{eq: Pentagon1.4} and \eqref{eq: ChangeCoordinates0}, we obtain $a'(\upsilon, \nu, \eta) = 1$. The claim follows.
\end{proof}

The second step consists of the following result.

\begin{proposition}\label{prop: SecondStep}
	There is a change of coordinates for which, for all $\int^{\oplus}_{t\in T}\delta_{p(t)}\mathrm{d}\upsilon(t)$, $\int^{\oplus}_{s\in S}\delta_{n(s)}\mathrm{d}\nu(s)$ and $\int^{\oplus}_{r\in R}\delta_{z(r)}\mathrm{d}\eta(r)$ in $\Ccal_0$,
	\begin{enumerate}
		\item$a(\upsilon, \nu, \eta)$ and $a_3(\upsilon, \nu)$ are trivialized,
		\item $\beta_1(\upsilon) = \sigma_L$ is the shift operator,
		\item $ a_2(\mu, \mu)$ can be represented by a continuous bicharacter $\chi$.
	\end{enumerate}
\end{proposition}

The proof of Proposition \ref{prop: SecondStep} will depend on the following lemma, which characterizes $b_1$ and $a_2$ on products of a measure with the Haar measure.

\begin{lemma}\label{lemma: beta1a2}
	Let $\int^{\oplus}_{t\in T}\delta_{p(t)}\mathrm{d}\upsilon(t)\in \Ccal_0$. For $\upsilon\times\mu\times \mu$-almost every $(t,x,y)\in T\times G\times G$,
	$$b_1(\upsilon\times \mu)(t,x,y) = b_1(\mu)(p(t)+x,y)$$
		\[
	a_2(\upsilon\times \mu, \mu)(t,x,y) = a_2(\mu, \mu)(p(t)+x, y),\hspace*{1cm} a_2(\mu, \upsilon\times \mu)(t,x,y) = a_2(\mu, \mu)(x, p(t)+y).
	\]
\end{lemma}
\begin{proof}
	We prove the first equality, the other two follow from similar arguments. A function $f\in L^2(\upsilon)$ induces a morphism 
	\[
	\begin{array}{cccc}
		\sigma_f:&L^2(\mu)&\to&L^2(\upsilon\times\mu)\\
		&g&\mapsto& f(t)g\big(p(t)+x\big).
	\end{array}
	\]
	By naturality of the associators and the coordinates, the first square in
	% https://q.uiver.app/#q=WzAsNixbMCwwLCJMXjIoXFxtdVxcdGltZXNcXG11LCBIKSJdLFsyLDAsIkxeMihcXG11XFx0aW1lc1xcbXUsIEgpIl0sWzMsMCwiTF4yKFxcbXVcXHRpbWVzXFxtdSwgSCkiXSxbMCwxLCJMXjIoKFxcdXBzaWxvblxcdGltZXNcXG11KVxcdGltZXNcXG11LCBIKSJdLFsyLDEsIkxeMigoXFx1cHNpbG9uXFx0aW1lcyBcXG11KVxcdGltZXNcXG11LCBIKSJdLFszLDEsIkxeMigoXFx1cHNpbG9uXFx0aW1lc1xcbXUpXFx0aW1lc1xcbXUsIEgpIl0sWzAsMSwiXFxiZXRhXzEoXFxtdSkiXSxbMSwyLCJcXHNpZ21hX0xeey0xfSJdLFszLDQsIlxcYmV0YV8xKFxcdXBzaWxvblxcdGltZXNcXG11KSIsMl0sWzQsNSwiXFxzaWdtYV9MXnstMX0iLDJdLFswLDMsIlxcc2lnbWFfZlxcb3RpbWVzXFxpZCIsMl0sWzEsNCwiXFxzaWdtYV9mXFxvdGltZXNcXGlkIiwyXSxbMiw1LCJcXHNpZ21hX2ZcXG90aW1lc19cXGlkIl1d
\[\begin{tikzcd}
	{L^2(\mu\times\mu, H)} && {L^2(\mu\times\mu, H)} & {L^2(\mu\times\mu, H)} \\
	{L^2((\upsilon\times\mu)\times\mu, H)} && {L^2((\upsilon\times \mu)\times\mu, H)} & {L^2((\upsilon\times\mu)\times\mu, H)}
	\arrow["{\beta_1(\mu)}", from=1-1, to=1-3]
	\arrow["{\sigma_f\otimes\id}"', from=1-1, to=2-1]
	\arrow["{\sigma_L^{-1}}", from=1-3, to=1-4]
	\arrow["{\sigma_f\otimes\id}"', from=1-3, to=2-3]
	\arrow["{\sigma_f\otimes_\id}", from=1-4, to=2-4]
	\arrow["{\beta_1(\upsilon\times\mu)}"', from=2-1, to=2-3]
	\arrow["{\sigma_L^{-1}}"', from=2-3, to=2-4]
\end{tikzcd}\]
	commutes. By definition, the second square also commutes. Hence, the outer rectangle commutes, meaning that
    \[
    f(t)b_1(\mu)(p(t)+x,y) = f(t)b_1(\upsilon\times \mu)(t, x, y)
    \]
    for almost every $(t,x,y)\in T\times G\times G$. Since $f\in L^2(\upsilon)$ was arbitrary, the claim follows.
\end{proof}
We are now ready to prove Proposition \ref{prop: SecondStep}.

\begin{proof}[Proof of Proposition \ref{prop: SecondStep}]
	 By Proposition \ref{prop: FirstStep}, we can assume that the choice of coordinates is such that $a$ and $a_3$ have been trivialized. Let us denote $b:=b_1(\mu)$. Equation \eqref{eq: Pentagon2.34}, together with Lemma \ref{lemma: beta1a2} imply that
	 \begin{equation}\label{Eq b1mu}
	 	b(y, -x+z)\cdot  b(x,z) = b(x+y, z)
	 \end{equation}
	 as functions in $L^\infty(\mu\times\mu\times\mu, U(H))$. Hence, by Lemma \ref{lemm: SolutionToEquationL2}, we can pick a function  $B\in L^\infty(\mu, U(H))$ such that
	 \[
	 b_1(\mu)(x,y) = b(x,y) = B(-x+y)^{-1}\cdot B(y).
	 \]
	 We define $\omega(x) := B(x)\in L^\infty(\mu, U(H))$. Then, by Equation \eqref{eq: ChangeCoordinates2.1}, 
	 \[
	 b_1'(\mu)(x,x+y) = \omega(y)\cdot b(x,x+y)\cdot  {\omega(x+y)}^{-1} = B(y)\cdot B(y)^{-1}\cdot B(x+y)\cdot B(x+y)^{-1} = \id,
	 \]
	 and hence $\beta'_1(\mu) = \sigma$ is just a shift operator. By Equation \eqref{eq: Pentagon2.34} and Lemma \ref{lemma: beta1a2}, $b_1(\upsilon) = \id$ for any $\upsilon$. This proves $(\mathrm{i})$ and $(\mathrm{ii})$.
	 
	 We proceed similarly for $a_2$. Equations \eqref{eq: Pentagon1.2} and \eqref{eq: Pentagon1.3}, and Lemma \ref{lemma: beta1a2} imply that $a_2(\mu, \mu)$ admits a representative $a: G\times G\to U(1)$ such that
     \begin{align*}
        a(x, y_1+y_2) &= a(x_, y_1)a(x, y_2)\\
        a(x_1+x_2, y) & = a(x_1, y)a(x_2, y)
     \end{align*}
     for almost every $(x, y_1,y_2)\in G^3$ and almost every $(x_1, x_2, y)\in G^3$. In order to show that $a_2(\mu, \mu)$ can be represented by a continuous bicharacter, we want to represent elements of $G$ and $\hat{G}$ as vectors of a separable Hilbert space. Since $G$ is a locally compact Hausdorff second countable group, there exists a finite Borel measure $\lambda$ equivalent to the Haar measure $\mu$. We define the maps
     \[
     \begin{array}{cccc}
         \iota:&\hat{G}&\to &L^2(G, \lambda)\\
         & \eta&\mapsto& \big(x\mapsto\eta(x)\big)
     \end{array}\hspace{2cm}
        \begin{array}{cccc}
         A:&{G}&\to &L^2(G, \lambda)\\
         & y&\mapsto& \big(x\mapsto a(x,y)\big).
     \end{array}
     \]
     The map $\iota$ is continuous by dominated convergence, and it is injective as $\lambda$ has full support and characters are continuous. Since both $\hat{G}$ and $L^2(G, \lambda)$ are Polish spaces, $\iota(\hat{G})$ is Borel by the Lusin-Souslin theorem, and $\iota^{-1}: \iota(\hat{G})\to \hat{G}$ is Borel. The map $A$ is also Borel by Fubini's theorem.

     By hypothesis on $a$, for almost every $y\in G$, the measurable map $x\in G\mapsto a(x,y)\in U(1)$ is a homomorphism almost everywhere. By \cite[Cor. 5.3]{ramsay}, it agrees almost everywhere with a genuine measurable homomorphism $G\to U(1)$, which is continuous by \cite{sasvari}. Hence, $A(y)\in \iota(\hat{G})$ for almost every $y\in G$. Let $Y:=A^{-1}(\iota(\hat{G}))\subset G$, which is a Borel subset, and write $\Phi: G\to \hat{G}$ for the Borel map
     \[\Phi(y) = 
     \begin{cases}
        \iota^{-1}(A(y)),\quad &\text{if $y\in Y$}\\
        1,\quad &\text{otherwise},
     \end{cases}
     \]
     for $1\in\hat{G}$ the trivial character. Again by hypothesis, for almost every $(y_1,y_2)\in G\times G$, $\Phi(y_1+y_2) = \Phi(y_1)\Phi(y_2)$, where we use that the set of pairs $(y_1,y_2)\in Y\times Y$ for which $y_1+y_2\in Y$ is conull. Hence, by \cite[Cor. 5.3]{ramsay}, it agrees almost everywhere with a genuine measurable homomorphism $G\to \hat{G}$, which is continuous by \cite{sasvari}.
     
      For almost every $(x,y)\in G\times Y$, it holds that $\chi(x,y):=\Phi(y)(x) = a(x,y)$ by construction. It holds that $\chi$ is a continuous bicharacter and it agrees with $a$ almost everywhere. Hence, the claim follows.
\end{proof}

Before defining the last change of coordinates $\psi$, we show that the Hilbert space $H$ is necessarily one-dimensional. Let $\int^\oplus_{t\in T}\delta_{p(t)}\mathrm{d}\upsilon(t)\in \Ccal_0$. Equation \eqref{eq: Pentagon2.24} implies
\[
b_2(\upsilon)(t,x)\cdot a_2(\mu, \upsilon)(y,t) = b_2(\upsilon)(t,x+y),
\]
and by Equation \eqref{eq: Pentagon1.3} and Lemma \ref{lemma: beta1a2} it holds that
\[
a_2(\mu, \upsilon)(x,t)\cdot a_2(\mu, \upsilon)(y,t) = a_2(\mu, \upsilon)(x+y,t).
\]
Hence, $\phi:=b_2(\upsilon)$ and $\Pi:= a_2(\mu, \upsilon)$ satisfy the hypotheses of Lemma \ref{Lemm: Function P}, and there exists a function $P(\upsilon)\in L^\infty(T, \upsilon, U(H))$ such that
\[
b_2(\upsilon)(t,x) = a_2(\mu, \upsilon)(x,t)\cdot P(\upsilon)(t).
\]
In the lemma, $P(\upsilon)$ is induced by the function (essentially constant in the variable $x$) $(t,x)\in T\times G\mapsto b_2(\upsilon)(t, x)\cdot a_2(\mu, \upsilon)(x,t)^{-1}$. We prove that the family $P(-)$ is natural, meaning that it defines a unitary natural automorphism of the functor
\[
\begin{array}{ccc}
    \mathcal{C}_0 &\to & \Hilb  \\
     \int^{\oplus}_{t\in T}\delta_{p(t)}\mathrm{d}\upsilon(t)&\mapsto & L^2(T, \upsilon)\otimes H \cong L^2(T, \upsilon, H). 
\end{array}
\]

Let $\int^\oplus_{s\in S}\delta_{n(s)}\mathrm{d}\nu(s)\in\Ccal_0$ be another object and $F: \int^\oplus_{t\in T}\delta_{p(t)}\mathrm{d}\upsilon(t)\to \int^\oplus_{s\in S}\delta_{n(s)}\mathrm{d}\nu(s)$ a morphism of $C_0(G)$-representations. Under the canonical equivalence of Hilbert spaces $L^2(T\times G, \upsilon\times\mu, H)\cong L^2(T, \upsilon)\otimes L^2(G, \mu, H)$, the operator
\[
\begin{array}{cccc}
 Q(\upsilon): L^2(T\times G, \upsilon\times\mu, H)   &\to &L^2(T\times G, \upsilon\times\mu, H)  \\
    f &\mapsto & (t,x)\mapsto a_2(\mu, \upsilon)(x, t)^*\cdot b_2(\upsilon)(t, x)(f) 
\end{array}
\]
is given by $P(\upsilon)\otimes \id_{L^2(G, \mu, H)}$ by definition of $P(\upsilon)$. The analogous equation holds for $P(\nu)$ using the analogous operator $Q(\nu)$. We write $\tilde{F}:L^2(T, \upsilon)\otimes H\otimes L^2(G, \mu)\to L^2(S, \nu)\otimes H\otimes L^2(G, \mu)$ for the operator $\tilde{F}:=F\otimes \id_H\otimes \id_{L^2(G, \mu)}.$ By naturality of the associator, it holds that $\tilde{F}\circ Q(\upsilon) = Q(\nu)\circ \tilde{F}$, implying that
\[
\big((F\otimes \id_H)\circ P(\upsilon) \big)\otimes \id_{L^2(G, \mu)} = \big(P(\nu)\circ (F\otimes \id_H)\big)\otimes \id_{L^2(G, \mu)}.
\]
Since tensoring with the identity operator on $L^2(G,\mu)$ is faithful, the family $P(-)$ is natural.

\begin{theorem}\label{thm: HisC}
Let $\Ccal$ be a Tambara-Yamagami $\mathrm{W}^*$-tensor category for a locally compact abelian group $G$. Let $\tau$ be the simple non-invertible object of $\Ccal$. Then,
\[
\tau\otimes \tau\cong \int^\oplus_{x\in G}\delta_x\mathrm{d}\mu(x).
\]
\end{theorem}
\begin{proof}
By Proposition \ref{Prop: Tau2isL2G}, there is a Hilbert space $H$ such that
\[
\tau\otimes \tau\cong H\cdot\int^\oplus_{x\in G}\delta_x\mathrm{d}\mu(x).
\]
After having performed the change of coordinates in Proposition \ref{prop: SecondStep}, defining the family of functions $P(-)$ as above, by Equations \eqref{eq: Pentagon2.14} and \eqref{eq: Pentagon1.2}, we obtain
\[
a_1(\upsilon,\nu)(t,s) = P(\upsilon\times\nu)(t,s)^{-1}P(\upsilon)(t)P(\nu)(s)\in L^\infty(\upsilon\times\nu, U(1)).
\]
Using Equations \eqref{eq: Pentagon3.4} and \eqref{eq: Pentagon3.2} for $\int^\oplus_{t\in T}\delta_{p(t)}\mathrm{d}\upsilon(t) = \int^\oplus_{x\in G}\delta_x\mathrm{d}\mu(x)$, and letting $P$ be a representative of $P(\mu)$,
\begin{equation}\label{eq:bintermsofP}
b_3(\mu)(x,y) = P(x)^{-1}\cdot\frac{P(\mu\times\mu)(x,y)^{-1}P(x)P(y)}{P(\mu\times\mu)(y, x)^{-1}P(y)P(x)}
\end{equation}
for almost every $(x,y)\in G^2$. Let $\rho(x,y):=\frac{P(\mu\times\mu)(x,y)^{-1}P(x)P(y)}{P(\mu\times\mu)(y, x)^{-1}P(y)P(x)}\in L^\infty(\mu\times\mu, U(1)).$ Using this, the pentagon Equation \eqref{eq: Pentagon4} reads
\begin{align}\label{Eq: Pentagon4Bis}
[\gamma\otimes \id]\Big(\chi(y, x)\cdot [\id\otimes P(x)]\circ [\gamma\otimes &\id](f)(x,y)\Big) \\ & \nonumber= \sigma_H\circ [P(x)^{-1}\otimes \id]\Bigg(\rho(x,-x+y)f(-x+y, y)\Bigg)
\end{align}
for any $f\in L^2(\mu\times\mu, H\otimes H)$. Let us write $L$ and $R$ for the operators in $U\big(L^2(G\times G, H\otimes H)\big)$ defining the left- and right-hand sides of Equation \eqref{Eq: Pentagon4Bis} respectively. Fix an operator $T\in B(H)$ and write $T_1,T_2\in B(L^2(G\times G, H\otimes H))$ for the operators $T_1f(x,y) = (T\otimes \id)f(x,y)$ and $T_2f(x,y) = (\id\otimes  T)f(x,y)$ for all $f\in L^2(G\times G, H\otimes H).$ It is clear that
\begin{align*}
R\circ T_2 = \sigma_H\circ [P(x)^{-1}&\otimes \id]\big(\rho(x,-x+y)\cdot -\big) \circ T_2 \\&= \sigma_H\circ T_2\circ [P(x)^{-1}\otimes \id]\big(\rho(x,-x+y)\cdot -\big)  = T_1\circ R.
\end{align*}
Similarly, since $T_2$ commutes with $\gamma\otimes \id$ and with multiplication by $\chi(y,x)$,
\[
L\circ T_2 = [\gamma\otimes\id]\circ (\id\otimes P(x)TP(x)^* )\circ [\gamma\otimes \id]^*\circ L.
\]
Since $L = R$ by hypothesis, we obtain that
\[
T_1 = [\gamma\otimes \id]\circ (\id\otimes P(x) TP(x)^*)\circ [\gamma^*\otimes \id],
\]
and conjugating by $\gamma^*\otimes \id$,
\begin{equation}\label{eq: middlestep}
\gamma^*T\gamma\otimes \id = \id\otimes P(x) TP(x)^*.
\end{equation}
The left-hand side of Equation \eqref{eq: middlestep} acts trivially on the second $L^2(G, H)$-factor, and hence it commutes with $S_2 = (\id\otimes S)\in B\big(L^2(G\times G, H\otimes H)\big)$ for every $S\in B(H)$. Therefore, the right-hand side of Equation \eqref{eq: middlestep} also commutes with any such $S_2$, implying that $P(x)TP(x)^*$ commutes with any operator $S\in B(H)$ for almost every $x\in G$. Since $H$ is separable, we obtain that $P(x)TP(x)^*$ is a scalar multiple of the identity for almost every $x\in G$. This implies that $T$ is a scalar multiple of the identity. Since $T\in B(H)$ was arbitrary, we obtain that $B(H) = \mathbb{C}$ and $H$ has dimension~1.
\end{proof}

We can now finish the description of the tensor product structure on $\Rep(C_0(G))\oplus \Hilb\cdot \tau$. Recall that we call a bicharacter $\chi: G\times G\to U(1)$ nondegenerate if the map $j_\chi :G\to \hat{G}$, $j_\chi(y)(x) = \chi(x,y)$ is an isomorphism, and that we write $\mathcal{F}: L^2(\hat{G})\to L^2(G)$ for the unitary Fourier~transform. If $\chi$ is nondegenerate, we write $J_\chi: L^2(G, \mu)\to L^2(\hat{G}, \hat{\mu})$ for the unitary $J_\chi(f)(\eta) = \sqrt{c_\chi}\cdot f\big(j_\chi^{-1}(\eta)\big)$, for $c_\chi\in \R_{>0}$ the unique positive number such that $(j_\chi)_*\mu = c_\chi\hat{\mu}$.

\begin{theorem}\label{thm: ThirdStep} Let $\Ccal$ be a Tambara-Yamagami $\mathrm{W}^*$-tensor category for $G$. Let $\tau$ be the non-invertible simple object. Then $\tau\otimes \tau\cong \int^\oplus_{x\in G}\delta_x\mathrm{d}\mu(x)$ and there is a choice of coordinates and a sign $\xi\in \{\pm 1\}$ for which
		\begin{enumerate}
		\item$a, a_3, a_1$ are trivialized, 
		\item $\beta_1 = \sigma$ and $\beta_3 = \sigma^{-1}$,
		\item $\chi:= a_2(\mu, \mu)$ is a nondegenerate  continuous symmetric bicharacter, and
		\begin{align*}
		a_2(\upsilon,\nu)(t,s) &= \chi\big(p(t), n(s)\big)\\
		b_2(\upsilon)(t,x) & = \chi\big(p(t), x\big),
		\end{align*}
        almost everywhere for all objects $\int^\oplus_{t\in T}\delta_{p(t)}\mathrm{d}\upsilon(t),\,\int^\oplus_{s\in S}\delta_{n(s)}\mathrm{d}\nu(s)\in \mathcal{C}_0$,
		\item the isomorphism  $\gamma$ is given by the composition
		\[
		L^2(G)\xrightarrow{J_\chi}L^2(\hat{G})\xrightarrow{\xi\cdot \mathcal{F}} L^2(G).
		\]
	\end{enumerate}
\end{theorem}
\begin{proof}
	We pick the change of coordinates described in Proposition \ref{prop: SecondStep}. By Theorem \ref{thm: HisC}, we can pick a unitary isomorphism $H\cong\CC$. Then, $b_2$ and $b_3$ are $L^\infty$-functions with values in $U(1)$. Let $\int^\oplus_{t\in T}\delta_{p(t)}\mathrm{d}\upsilon(t),\,\int^\oplus_{s\in S}\delta_{n(s)}\mathrm{d}\nu(s)\in \mathcal{C}_0$. Then, using Equations \eqref{eq: Pentagon2.24} and \eqref{eq: Pentagon2.13}, we obtain
	 \begin{equation}\label{eq: almostsymmetric}
	 	a_2(\upsilon, \nu)(t,s) = {b_2(\nu)(s, x)}^{-1}\cdot b_2(\nu)(s,t+x) = a_2(\nu,\upsilon)(s,t),
	 \end{equation}
	 hence the continuous bicharacter $\chi$ is symmetric. Since $P(\upsilon)\in L^\infty(\upsilon, U(1))$ and the family $P(-)$ is natural, we can set $\psi(\upsilon) := P(\upsilon)$. By Equation \eqref{eq: ChangeCoordinates2.2}, we obtain $$b_2'(\upsilon)(t,x) = a_2(\mu,\upsilon)(x,t) = a_2(\upsilon, \mu)(t,x),$$ 
     and by Equations \eqref{eq: ChangeCoordinates2.3} and \eqref{eq:bintermsofP},
     $$b_3'(\mu)(x,y)  = 1,$$
     where we have also used that $P(\mu\times\mu)(x,y) = P(\mu)(x+y)$ by naturality. Similarly to how we have argued for $b_1$ in the proof of Proposition \ref{prop: SecondStep}, this implies that $b'_3(\upsilon) = 1$ for any object $\int^\oplus_{t\in T}\delta_{p(t)}\mathrm{d}\upsilon(t)$ of $\Ccal_0$. Therefore, $\beta'_3(\upsilon) = \sigma^{-1}$ is the inverse of the shift operator. The triviality of $a_1$ follows straightforwardly from Equations \eqref{eq: Pentagon1.2} and \eqref{eq: Pentagon2.14}.
	
	Let us now characterize $a_2$. By Equations \eqref{eq: Pentagon1.3} and Lemma \ref{lemma: beta1a2},
	\[
	a_2(\upsilon, \mu)(t,x) = \frac{\chi\big(p(t)+y, x\big)}{\chi(y,x)} = \chi\big(p(t), x\big)
	\]
	and by Equation \eqref{eq: Pentagon1.2} and Lemma \ref{lemma: beta1a2}, 
	\[
	a_2(\mu, \upsilon)(x,t) = \frac{\chi\big(x, p(t)+y\big)}{\chi(x,y)} = \chi\big(x, p(t)\big)	
	\]
	for almost every $y\in G$. Similarly, using naturality of the associator, we obtain
	\[
	a_2(\upsilon, \nu\times\mu)(t,s,x) = a_2(\upsilon, \mu)\big(t,x+n(s)\big) = \frac{\chi\big(p(t)+y, x+n(s))}{\chi\big(y, x+n(s)\big)} = \chi\big(p(t),x+n(s)\big)
	\]
	for almost every $y\in G$. This last equality, together with Equation \eqref{eq: Pentagon1.2}, yields
	\[
	a_2(\upsilon, \nu)(t,s) = \frac{\chi\big(p(t), x+n(s)
		\big)}{\chi\big(p(t), x\big)} = \chi\big(p(t), n(s)\big)
	\]
	for almost every $x\in G$, as needed. Next, note that  Equations \eqref{eq: Pentagon3.2} and \eqref{eq: Pentagon3.3} now read, for any $g(y)\in L^2(G)$ and almost every $x\in G$,
\begin{align}\label{Eq: 1415rephrased}
	[\id\otimes \gamma]\big(\chi(x,y)g(y)\big) = \gamma(g)(-x+y)\hspace{1cm}  [\id\otimes \gamma]\big(g(x+y)\big) &= \chi(x,y)\gamma(g).
\end{align}
    We will apply Theorem \ref{thm: Schur}. In order to do so, we need to argue that $\chi$ induces an injective map $G\hookrightarrow \hat{G}$. Note that, applying Equation \eqref{eq: Pentagon4} to a product function $f(x)g(y)\in L^2(\mu\times\mu)$ of functions in $L^2(\mu)$, we obtain
    \begin{equation}\label{eq: Product}
    \gamma\big(\chi(x,y)\cdot\gamma(f)\big)(x)g(y) = f(-x+y)g(y).
    \end{equation}
     Given $y\in G$, we write $R_y\in B(L^2(G))$ for the operator $R_yf(x) = f(-x+y)$, and $M_y\in B(L^2(G))$ for the operator $M_yf(x) = \chi(x,y)f(x)$. Since $f$ and $g$ in Equation \eqref{eq: Product} are arbitrary, it follows that 
     \begin{equation}\label{eq: Pentagon4Simplified}
     \gamma\circ M_y\circ\gamma = R_y
     \end{equation}
     for almost every $y\in G$. Let $\{f_n\}_{n\in \mathbb{N}}$ be a countable dense subset of $L^2(G)$ and, for every $n\in\mathbb{N}$ write $Y_n\subset G$ for the subset of full measure $Y_n = \{y\in G\,|\, \gamma\circ M_y\circ\gamma(f_n) = R_y(f_n)\}$. Then, $Y:=\bigcap\limits_{n\in\mathbb{N}} Y_n$ is of full measure. For every $y\in Y$, Equation \eqref{eq: Product} holds for any function $f$ on the dense subset $\{f_n\}_{n\in \mathbb{N}}\subset L^2(G)$. By continuity, it holds for all $f\in L^2(G)$. The maps $y\in G\mapsto M_y\in B(L^2(G))$ and $y\in G\mapsto R_y\in B(L^2(G))$ are strongly continuous by continuity of $\chi$, dominated convergence, and continuity of the translation operator of $G$ on $L^2(G)$. Hence, it holds that for all $y\in G$,
     \[
     \gamma\circ M_y\circ\gamma = R_y.
     \]
     Therefore, we can apply Lemma \ref{Lemm: nondegenerate} to deduce that the bicharacter $\chi$ indeed induces an injective homomorphism $j_\chi:G\hookrightarrow\hat{G}.$ Thus, the hypotheses of Theorem \ref{thm: Schur} are satisfied and it follows that $\gamma$ is given by
\[
L^2(G)\xrightarrow{J_\chi}L^2(\hat{G})\xrightarrow{\xi\cdot \mathcal{F}} L^2(G),
\]
for some $\xi\in U(1)$. In addition, $j_\chi:G\to \hat{G}$ is an isomorphism. Finally, Equation \eqref{eq: Pentagon4Simplified} now reads, using the frequency shifting property of the Fourier transform, 
\[
\xi^2(\mathcal{F}\circ J_\chi)^2(f)(x-y) = f(-x+y)
\]
for any $f\in L^2(\mu)$. Since the square of $\mathcal{F}\circ J_\chi$ is the parity operator by the Inverse Fourier Theorem \cite[Thm. 1.5.1]{rudin}, we obtain $\xi = \pm1$.  
\end{proof}

We can now prove the main classification result for Tambara-Yamagami $\mathrm{W}^*$-tensor~categories. The group $\text{Aut}(G)$ of continuous automorphisms of $G$ acts on the set of continuous symmetric nondegenerate bicharacters on $G$ by $\phi\cdot\chi = \chi\circ(\phi^{-1}\times\phi^{-1})$ for all $\phi\in \text{Aut}(G)$ and $\chi:G\times G\to U(1)$ a continuous symmetric nondegenerate bicharacter. Therefore, $\text{Aut}(G)$ acts on the set of pairs $(\chi, \xi)$ of a continuous symmetric nondegenerate bicharacter on $G$ and a sign $\xi\in\{\pm1\}$.

\begin{theorem}\label{thm: modAut(G)}
    Let $G$ be a locally compact abelian group. There is a bijection 
\[
\nicefrac{\left\{\begin{array}{l} (\chi, \xi)\ |\ 
    \text{$\chi: G\times G\to U(1)$ a}\\
    \text{continuous symmetric} \\\text{nondegenerate bicharacter}\\\text{ and $\xi\in\{\pm1\}$}
  \end{array}\right\}}{\text{Aut}(G)}\xrightarrow{\TYcal(G,-,-)}\nicefrac{\left\{\begin{array}{l}
    \text{Tambara-Yamagami $\mathrm{W}^*$-}\\
    \text{tensor categories for $G$}
  \end{array}\right\}}{\mathrm{W}^*-\otimes\text{ equiv.}} 
\]
\end{theorem}
\begin{proof}
We first show that the map is well-defined. Let $\chi:G\times G\to U(1)$ be a continuous symmetric nondegenerate bicharacter and $\xi\in \{\pm1\}$ a sign. Fix $\phi: G\to G$ a continuous group automorphism and define $\chi':=\chi\circ (\phi^{-1}\times\phi^{-1})$. We need to show that $\TYcal(G, \chi, \xi)\cong \TYcal(G, \chi', \xi)$ as $\mathrm{W}^*$-tensor categories. Let $F: \TYcal(G, \chi, \xi)\to \TYcal(G, \chi',\xi)$ be the $\mathrm{W}^*$-functor with $F(\tau) = \tau$ and $F\big(\int^{\oplus}_{t\in T }\delta_{p(t)}\mathrm{d}\upsilon(t)\big) = \int^{\oplus}_{t\in T }\delta_{\phi\circ p(t)}\mathrm{d}\upsilon(t)$ for all $\int^{\oplus}_{t\in T }\delta_{p(t)}\mathrm{d}\upsilon(t)\in \Rep(C_0(G))$, extended trivially to morphisms. We write $a_{\phi^{-1}}\in \R_{>0}$ for the unique positive real number such that $(\phi^{-1})_*\mu = a_{\phi^{-1}}\cdot \mu$, and we denote by $U_\phi: \int^{\oplus}_{x\in G}\delta_x\mathrm{d}\mu(x)\to \int^{\oplus}_{x\in G}\delta_{\phi\circ x}\mathrm{d}\mu(x)$ the unitary equivalence of $C_0(G)$-representations given by $f\in L^2(G)\mapsto \sqrt{a_{\phi^{-1}}}\cdot f\circ \phi\in L^2(G)$. We provide $F$ with a tensorator $s$ whose only non-trivial tensor data is
\[\hspace{-.3cm}
s_{\tau, \tau}: F(\tau)\otimes F(\tau)
= \tau\otimes \tau
= \int^{\oplus}_{x\in G}\delta_x\,\mathrm{d}\mu(x)
\xrightarrow{U_\phi}
\int^{\oplus}_{x\in G}\delta_{\phi(x)}\,\mathrm{d}\mu(x)
= F\!\left(\int^{\oplus}_{x\in G}\delta_x\,\mathrm{d}\mu(x)\right)
= F(\tau\otimes \tau).
\]
Compatibility with the associators is straightforward to check. Hence, the map in the statement of the theorem is well defined.

We show that the map is surjective. Let $\Ccal = \Rep(C_0(G))\oplus \Hilb\cdot \tau$ be a Tambara-Yamagami $\mathrm{W}^*$-tensor category for $G$. After having performed the change of coordinates described in Theorem \ref{thm: ThirdStep}, we obtain a collection of natural isomorphisms $[-,-], [\tau, -], [-,\tau]$ and an isomorphism $[\tau]$, as well as a continuous symmetric nondegenerate bicharacter $\chi$ and a sign $\xi\in\{\pm1\}$. Then, the identity functor $\Ccal = \Rep(C_0(G))\oplus \Hilb\cdot \tau$, together with $[-,-]^{-1},[\tau,-]^{-1},[-,\tau]^{-1}, [\tau]^{-1}$ is an equivalence of $\mathrm{W}^*$-tensor categories between $\Ccal$ and $\TYcal(G, \chi, \xi)$. Note that the definition of Tambara-Yamagami $\mathrm{W}^*$-tensor category already implies that $\delta_e$ is the unit, for $e\in G$ the identity element. There is no need to argue compatibility with the unitors, see \cite[Prop. 2.4.3]{EGNO}.

It remains only to show that the map is injective. Let $\chi, \chi': G\times G\to U(1)$ be two bicharacters and $\xi, \xi'\in \{\pm1\}$ be two signs. Let $(F, s): \TYcal(G, \chi, \xi)\to \TYcal(G, \chi', \xi')$ be a $\mathrm{W}^*$-tensor equivalence. Since $F(\tau)\in \TYcal(G, \chi', \xi')$ is a non-invertible simple object, we can pick a unitary isomorphism $\pi_\tau: F(\tau)\xrightarrow{\cong}\tau$. Denoting by $\mu\in \Rep(C_0(G))$ the regular representation $L^2(G) = \int^\oplus_{x\in G}\delta_x\mathrm{d}\mu(x)$, we have a sequence of unitary equivalences in $\TYcal(G, \chi', \xi')$
    \[
   \Phi:  F(\mu) = F(\tau\otimes\tau)\xrightarrow[\cong]{s_{\tau, \tau}^{-1}} F(\tau)\otimes F(\tau) \xrightarrow[\cong]{\pi_\tau\otimes\pi_\tau}\tau\otimes\tau = \mu.
    \]
   We identify $L^\infty(G) = \End(\mu)$ by sending a function $a\in L^\infty(G)$ to the multiplication-by-$a$ bounded map $\mu\to \mu$. Let us denote by $C_\Phi: \End(F(\mu))\cong \End(\mu) = L^\infty(G)$ the von Neumann algebra isomorphism given by conjugation by the unitary $\Phi$. Then, we have a normal $*$-automorphism
   \begin{equation}\label{eq: VNAutomorphism}
   L^\infty(G) = \End(\mu)\xrightarrow{F}\End(F\mu)\xrightarrow{C_\Phi}\End(\mu) = L^\infty(G).
   \end{equation}
   Therefore, there exists a non-singular measurable bijection $\phi: G\to G$ (with non-singular inverse) such that the automorphism \eqref{eq: VNAutomorphism} is given by precomposition with $\phi^{-1}$. Throughout, we fix representatives of $\phi$ and $\phi^{-1}$ and denote them by the same symbols. Let $r_\phi = \frac{\mathrm{d} (\phi^{-1})_*\mu}{\mathrm{d}\mu}$. Since $\phi$ and $\phi^{-1}$ are non-singular, we have $0<r_\phi(x)<\infty$ for almost every $x\in G$. Then, there is a unitary equivalence of $C_0(G)$-representations
\[
\begin{array}{cccc}
   U_\phi:  & \int^{\oplus}_{x\in G} \delta_x\mathrm{d}\mu(x)&\to &\int^\oplus_{x\in G}\delta_{\phi(x)}\mathrm{d}\mu(x)\\
     & f &\mapsto & \sqrt{r_\phi}\cdot (f\circ \phi). 
\end{array}
\]
In addition, for any $\varphi\in L^\infty(G) = \End(\mu)$, it holds that the diagram
% https://q.uiver.app/#q=WzAsNixbMCwwLCJGKFxcbXUpIl0sWzIsMCwiXFxtdSJdLFswLDEsIkYoXFxtdSkiXSxbMywwLCJcXGludF57XFxvcGx1c31fe3hcXGluIFh9XFxkZWx0YV97XFxwaGkoeCl9XFxtYXRocm17ZH1cXG11KHgpIl0sWzIsMSwiXFxtdSJdLFszLDEsIlxcaW50XntcXG9wbHVzfV97eFxcaW4gWH1cXGRlbHRhX3tcXHBoaSh4KX1cXG1hdGhybXtkfVxcbXUoeCkiXSxbMCwxLCJcXFBoaSJdLFswLDIsIkYoXFx2YXJwaGlcXGNkb3QgLSkiLDJdLFsxLDMsIlBfXFxwaGkiXSxbMSw0LCIoXFx2YXJwaGlcXGNpcmMgXFxwaGleey0xfSlcXGNkb3QgLSIsMl0sWzIsNCwiXFxQaGkiLDJdLFs0LDUsIlBfXFxwaGkiLDJdLFszLDUsIlxcdmFycGhpXFxjZG90IC0iXV0=
\[\begin{tikzcd}
	{F(\mu)} && \mu & {\int^{\oplus}_{x\in G}\delta_{\phi(x)}\mathrm{d}\mu(x)} \\
	{F(\mu)} && \mu & {\int^{\oplus}_{x\in G}\delta_{\phi(x)}\mathrm{d}\mu(x)}
	\arrow["\Phi", from=1-1, to=1-3]
	\arrow["{F(\varphi\cdot -)}"', from=1-1, to=2-1]
	\arrow["{U_\phi}", from=1-3, to=1-4]
	\arrow["{(\varphi\circ \phi^{-1})\cdot -}"', from=1-3, to=2-3]
	\arrow["{\varphi\cdot -}", from=1-4, to=2-4]
	\arrow["\Phi"', from=2-1, to=2-3]
	\arrow["{U_\phi}"', from=2-3, to=2-4]
\end{tikzcd}\]
commutes. Indeed, the left diagram commutes by the fact that the automorphism \eqref{eq: VNAutomorphism} is given by precomposition by $\phi^{-1}$, and the right diagram commutes trivially. Hence, the isomorphism $\pi_\mu:= U_\phi\circ \Phi$ is natural with respect to $L^\infty(G) = \End(\mu)$.  We denote $\mu_\phi:=\int^\oplus_{x\in G}\delta_{\phi(x)}\mathrm{d}\mu(x)$. We define unitary operators $L,R:L^2(G)\to L^2(G)$ by
% https://q.uiver.app/#q=WzAsMTIsWzAsMCwiRihcXG11KVxcb3RpbWVzIEYoXFx0YXUpIl0sWzIsMCwiRihcXG11XFxvdGltZXNcXHRhdSkiXSxbMiwyLCJMXjIoRylcXGNkb3QgXFx0YXUiXSxbMiwxLCJMXjIoRylcXGNkb3QgRihcXHRhdSkiXSxbMCwxLCJcXG11XFxvdGltZXNcXHRhdSJdLFswLDIsIkxeMihHKVxcY2RvdCBcXHRhdSJdLFszLDAsIkYoXFx0YXUpXFxvdGltZXMgRihcXG11KSJdLFs1LDAsIkYoXFx0YXVcXG90aW1lc1xcbXUpIl0sWzUsMiwiTF4yKEcpXFxjZG90IFxcdGF1Il0sWzUsMSwiTF4yKEcpXFxjZG90IEYoXFx0YXUpIl0sWzMsMSwiXFx0YXVcXG90aW1lc1xcbXUiXSxbMywyLCJMXjIoRylcXGNkb3QgXFx0YXUiXSxbMCwxLCJzX3tcXG11LCBcXHRhdX0iXSxbMywyLCJcXHBpX1xcdGF1Il0sWzAsNCwiXFxwaV9cXG11XFxvdGltZXMgXFxwaV9cXHRhdSIsMl0sWzQsNSwiPSIsMl0sWzEsMywiPSJdLFs1LDIsIkwiLDJdLFs2LDcsInNfe1xcdGF1LCBcXG11fSJdLFs5LDgsIlxccGlfXFx0YXUiXSxbNiwxMCwiXFxwaV9cXHRhdVxcb3RpbWVzIFxccGlfXFxtdSIsMl0sWzEwLDExLCI9IiwyXSxbNyw5LCI9Il0sWzExLDgsIlIiLDJdXQ==
\[\begin{tikzcd}
	{F(\mu)\otimes F(\tau)} && {F(\mu\otimes\tau)} & {F(\tau)\otimes F(\mu)} && {F(\tau\otimes\mu)} \\
	{\mu\otimes\tau} && {L^2(G)\cdot F(\tau)} & {\tau\otimes\mu} && {L^2(G)\cdot F(\tau)} \\
	{L^2(G)\cdot \tau} && {L^2(G)\cdot \tau} & {L^2(G)\cdot \tau} && {L^2(G)\cdot \tau}.
	\arrow["{s_{\mu, \tau}}", from=1-1, to=1-3]
	\arrow["{\pi_\mu\otimes \pi_\tau}"', from=1-1, to=2-1]
	\arrow["{=}", from=1-3, to=2-3]
	\arrow["{s_{\tau, \mu}}", from=1-4, to=1-6]
	\arrow["{\pi_\tau\otimes \pi_\mu}"', from=1-4, to=2-4]
	\arrow["{=}", from=1-6, to=2-6]
	\arrow["{=}"', from=2-1, to=3-1]
	\arrow["{\pi_\tau}", from=2-3, to=3-3]
	\arrow["{=}"', from=2-4, to=3-4]
	\arrow["{\pi_\tau}", from=2-6, to=3-6]
	\arrow["L"', from=3-1, to=3-3]
	\arrow["R"', from=3-4, to=3-6]
\end{tikzcd}\]
Since all the morphisms involved are natural with respect to $\End(\mu)$, we have that $L$ and $R$ are given by multiplication by some functions $l,r\in L^\infty(G, U(1))$ respectively. Then, the condition of the tensorator $s$ on $(\mu,\tau,\mu)$ implies that, for all $f\in L^2(G\times G)$,
\[
l(x)r(y)\chi'\big(\phi(x),\phi(y)\big)f(x,y) = l(x)r(y)\chi(x,y) f(x,y)
\]
as functions in $L^2(G\times G)$. Therefore, $\chi' \big(\phi(x),\phi(y)\big) = \chi(x,y)$, for almost every $(x,y)\in G\times G$. We can now show that $\phi: G\to G$ is a continuous group automorphism. We have that, for almost every $(x_1,x_2,y)\in G\times G\times G$,
\begin{align*}
\chi'(\phi(x_1+x_2), \phi(y)) &= \chi(x_1+x_2, y) \\&= \chi(x_1, y)\cdot \chi(x_2, y)\\& = \chi'(\phi(x_1), \phi(y))\cdot \chi'(\phi(x_2), \phi(y))\\
& = \chi'(\phi(x_1)+\phi(x_2), \phi(y)).
\end{align*}
Therefore, for almost every $z \in G$, it holds that 
\begin{equation}\label{eq: middle2}
\chi'(\phi(x_1+x_2), z) = \chi'(\phi(x_1) + \phi(x_2), z)
\end{equation}
for almost every $(x_1,x_2)\in G\times G$. Since both sides of Equation \eqref{eq: middle2} are continuous with respect to $z$ and the Haar measure has full support, Equation \eqref{eq: middle2} holds for all $z\in G$. Since $\chi'$ is continuous and nondegenerate, $\phi$ is a measurable group homomorphism almost everywhere. Therefore, by \cite[Cor. 5.3]{ramsay}, it can be represented by a genuine measurable homomorphism, and by \cite{sasvari}, it is continuous. Applying the same arguments to $\phi^{-1}$ we show that $\phi$ is a continuous group automorphism. Therefore, we have
\[
\chi'\circ(\phi\times\phi) = \chi
\]
as continuous functions on $G\times G$. Note that $\phi$ being a group homomorphism implies that $(\phi^{-1})_*\mu$ is a Haar measure, hence a multiple of $\mu$, and $r_\phi$ is essentially constant. The diagram for $s$ on $(\tau,\mu,\tau)$ implies that 
\[
\chi'(x,y)\cdot l(x) =\chi(\phi^{-1}(x), \phi^{-1}(y)) \cdot r(x)
\]
for almost every $(x,y)\in G\times G$, implying that $l = r$ as functions in $L^\infty(G)$. Let us write $\mathcal{F}_\chi:= \mathcal{F}\circ J_\chi$ and $\mathcal{F}_{\chi'}:=\mathcal{F}\circ J_{\chi'}$ for the unitary Fourier transforms $L^2(G)\to L^2(G)$ induced by the bicharacters $\chi,\chi'$. Then, the diagram for $(\tau,\tau,\tau)$ implies that
\[
\xi'\cdot U_\phi\circ \mathcal{F}_{\chi'} = \xi\cdot \mathcal{F}_{\chi}\circ U_\phi.
\]
We claim that $U_\phi\circ \mathcal{F}_{\chi'} = U_\phi\circ \mathcal{F}_{\chi\circ (\phi^{-1}\times\phi^{-1})} = \mathcal{F}_{\chi}\circ U_\phi$. Indeed, applied to a compactly supported continuous function $f\in C_c(G)$, we have
\[
\mathcal{F}\circ J_\chi(f)(x) = \sqrt{c_\chi}\int_{\eta\in \hat{G}}\eta(-x) f(j_{{\chi}}^{-1}(\xi))\mathrm{d}\hat{\mu}(\eta) = \frac{1}{\sqrt{c_\chi}}\int_{y\in G}\chi(-x,y)f(y)\mathrm{d}\mu(y),
\]
and similarly for $\mathcal{F}\circ J_{\chi'}(f)(x) = \frac{1}{\sqrt{c_{\chi'}}}\int_{y\in G}\chi'(-x,y)f(y)\mathrm{d}\mu(y)$. Hence, we find
\[
U_\phi\circ \mathcal{F}_{\chi'}(f)(x) = \sqrt{\frac{r_\phi}{c_{\chi'}}}\int_{y\in G}\chi'(-\phi(x), y)f(y)\mathrm{d}\mu(y),
\]
and
\begin{align*}
\mathcal{F}_{\chi}\circ U_\phi(f)(x) &= \sqrt{\frac{r_\phi}{c_{\chi}}}\int_{y\in G}\chi(-x, y)f(\phi(y))\mathrm{d}\mu(y)\\
& = \sqrt{\frac{r_\phi}{c_{\chi}}}\int_{y\in G}\chi'(-\phi(x), \phi(y))f(\phi(y))\mathrm{d}\mu(y)\\& = \sqrt{\frac{1}{c_\chi r_\phi}}\int_{y\in G}\chi'(-\phi(x), y)f(y)\mathrm{d}\mu(y).
\end{align*}
Since $c_{\chi'}/c_\chi = r_\phi^2$ and $C_c(G)$ is dense in $L^2(G)$, we conclude that $\xi = \xi'$, as needed.
\end{proof}

\subsection{Classification of continuous Tambara-Yamagami tensor categories}

Relying on the proof of Theorem \ref{thm: modAut(G)}, we classify continuous Tambara-Yamagami tensor categories. Let $G$ be a locally compact abelian group. Given a continuous Tambara-Yamagami category for $G$ $(A:=C_0(G)\oplus\CC, \TYcal_G, \alpha)$, we obtain a Tambara-Yamagami $\mathrm{W}^*$-tensor category $\mathfrak{F}(A, \TYcal_G, \alpha)$, where all the coordinates can be taken to be identities. In the previous section, we have constructed a sequence of changes of these coordinates which provide a continuous symmetric nondegenerate bicharacter $\chi$ on $G$ and a sign $\xi\in\{\pm1\}$ as well as, as discussed in the proof of Theorem \ref{thm: modAut(G)}, an equivalence of $\mathrm{W}^*$-categories between $\mathfrak{F}(A, \TYcal_G, \alpha)$ and the $\mathrm{W}^*$-tensor category $\TYcal(G, \chi, \xi)$, which is the image of the continuous tensor category $\TYcal(G,\chi, \xi)$ under $\mathfrak{F}$. In this section we shall prove that, actually, this equivalence of $\mathrm{W}^*$-tensor categories comes from an equivalence of continuous tensor categories between $(A, \TYcal_G, \alpha)$ and $\TYcal(G,\chi, \xi).$ To obtain this, it is enough to show that the changes of coordinates described in the previous section come from intertwiners between the relevant $\mathrm{C}^*$-correspondences. Recall that these changes of coordinates are given by multiplication by some functions $\theta, \varphi, \psi$, and $\omega$. These functions are defined using the form of the associators in the original coordinates, that is the original functions $a, a_i$, and $b_i$. Hence, we first need to understand the functions $a, a_i$, and $b_i$ for the particular case when the Tambara-Yamagami $\mathrm{W}^*$-tensor category comes from a continuous Tambara-Yamagami tensor category and the coordinates are taken to be identities.

Let $A:=C_0(G)\oplus\CC$ and $\TYcal_G$ be the $A-A\otimes A$-correspondence defined in Section \ref{sec: ctsTYcats}. Let $\alpha$ be an associator for $(A, \TYcal_G)$, that is, a unitary intertwiner 
\[
\TYcal_G\otimes_{A\otimes A}\big(\TYcal_G\otimes A\big)\xrightarrow{\cong} \TYcal_G\otimes_{A\otimes A}\big(A\otimes \TYcal_G \big)
\]
as $A-A^{\otimes 3}$-correspondences. By definition of $\TYcal_G$, both of the correspondences above decompose, as $A^{\otimes 3}$-Hilbert modules, as a direct sum of eight pieces, one for each ordered triple $(X,Y,Z)$ for $X,Y,Z\in \{C_0(G), \CC\}$, as in Section \ref{sec: ctsTYcats}. Hence, the data of $\alpha$ is the data of eight unitary intertwiners. For example, for $(C_0(G), C_0(G), C_0(G))$, we obtain a unitary intertwiner
\[
\alpha_0: C_0(G^3)\to C_0(G^3)
\]
as $C_0(G)-C_0(G^3)$-correspondences, where $C_0(G)$ acts by pullback along the group operation $m_{123}:G^3\to G$. Therefore, $\alpha_0$ is given by a unitary operator such that
\[
\alpha_0(\phi(x+y+z) f) = \phi(x+y+z) \alpha_0(f),\hspace{1cm} \alpha_0(ff') = \alpha_0(f)f'
\]
for $\phi\in C_0(G)$ and $f,f'\in C_0(G^3)$, that is, $\alpha_0$ is given by multiplication by a continuous function
\[
\mathrm{a}: G^3\to U(1).
\] 
Using the same arguments, we characterize all the pieces of the associator $\alpha$:
\begin{enumerate}
    \item for $(C_0(G), C_0(G), C_0(G))$ the associator is given by multiplication by some $\mathrm{a}: G^3\to U(1)$ continuous,
    \item for $(\CC, C_0(G), C_0(G))$ and $( C_0(G), C_0(G), \CC)$ the corresponding pieces are given by multiplication by some $\mathrm{a}_1: G^2\to U(1)$ and $\mathrm{a}_3: G^2\to U(1)$ continuous, respectively,
    \item for $(C_0(G), \CC, C_0(G))$, the associator is given by multiplication by some $\mathrm{a}_2: G^2\to U(1)$ continuous,
\item for $(C_0(G), \CC, \CC)$, the associator is a unitary $\beta_1: C_0(G)\otimes_\varepsilon L^2(G) \to  C_0(G)\otimes_\varepsilon L^2(G) $ such that 
$$\beta_1(\phi(y)f) = \phi(x+y)\beta_1(f), \hspace{1cm}\beta_1(f\rho(x)) = \beta_1(f)\rho(x)$$ for all $\phi, \rho\in C_0(G)$ and $f\in C_0(G)\otimes_\varepsilon L^2(G)\cong C_0(G, L^2(G)).$

\item for $(\CC, \CC, C_0(G))$, the associator is a unitary $\beta_3: C_0(G)\otimes_\varepsilon L^2(G)\to C_0(G)\otimes_\varepsilon L^2(G)$ such that $$\beta_3(\phi(x+y)f) = \phi(y)\beta_3(f), \hspace{1cm}\beta_3(f\rho(x)) = \beta_3(f)\rho(x)$$ for all $\phi, \rho\in C_0(G)$ and $f\in C_0(G)\otimes_\varepsilon L^2(G).$

\item for $(\CC, C_0(G), \CC)$,  the associator $C_0(G)\otimes_\varepsilon L^2(G)\to C_0(G)\otimes_\varepsilon L^2(G)$ is given by a strongly continuous field $x\in G\mapsto \mathrm{b}_{2}(x)\in U(L^2(G))$, where each $ \mathrm{b}_{2}(x)$ is given by multiplication by an $L^\infty(G, U(1))$ function also denoted $\mathrm{b}_{2}(x)\in L^\infty(G, U(1))$. We use that the commutant of the image of $C_0(G)$ in $B(L^2(G))$ is $L^\infty(G)$. 

\item for $(\CC, \CC, \CC)$, the associator is a unitary $\gamma: L^2(G)\to L^2(G)$.
\end{enumerate}

We can express $\beta_1$ and $\beta_3$ also in terms of functions as follows. Let $\sigma$ be the translation operator
\[
\begin{array}{cccccc}
\sigma :& C_0(G)\otimes_\varepsilon L^2(G)&\cong C_0(G, L^2(G))&\to &C_0(G, L^2(G))\cong&C_0(G)\otimes_\varepsilon L^2(G)\\
&&f(x,y)&\mapsto & f(x,x+y)&.
\end{array}
\]
Note that $\sigma^{-1}\circ \beta_1$ is given by a strongly continuous field $x\in G\mapsto \mathrm{b}_1(x)\in U(L^2(G))$, so that each $\mathrm{b}_1(x)$ is obtained by multiplication by a function also denoted $\mathrm{b}_1(x)\in L^\infty(G, U(1))$. Similarly, $\beta_3\circ \sigma$ is given by a strongly continuous field $x\in G\mapsto \mathrm{b}_3(x)\in U(L^2(G))$, with each $\mathrm{b}_3(x)$ induced by multiplication by a function $\mathrm{b}_3(x)\in L^\infty(G, U(1))$. The argument above allows us to describe the form of the functions $a, a_i, b_i$ introduced in Lemma \ref{lemm: CharacterizationCommuting} when the Tambara-Yamagami $\mathrm{W}^*$-tensor category is the image of the continuous tensor category $(A, \TYcal_G, \alpha)$ under $\Ffrak$, and the coordinates have been taken to be identities. Let $\int^\oplus_{t\in T}\delta_{p(t)}\mathrm{d}\upsilon(t),\, \int^\oplus_{s\in S}\delta_{n(s)}\mathrm{d}\nu(s),\, \int^\oplus_{r\in R}\delta_{z(r)}\mathrm{d}\eta(r)\in \Rep(C_0(G))$. Then,
\begin{align}
a(\upsilon, \nu, \eta)(t,s,r)&=\mathrm{a}(p(t), n(s), z(r)) \nonumber\\
a_i(\upsilon, \nu)(t,s) &= \mathrm{a}_i(p(t), n(s))\nonumber\\
b_i(\upsilon)(t, x) &= \mathrm{b}_i(p(t), x).\nonumber
\end{align}

We can produce the main classification of Tambara-Yamagami continuous tensor categories.

\begin{theorem}\label{thm: finalTY2}Let $G$ be a locally compact abelian group. Then, there is a commutative diagram of bijections   
% https://q.uiver.app/#q=WzAsMyxbMCwwLCJcXG5pY2VmcmFje1xcbGVmdFxce1xcYmVnaW57YXJyYXl9e2x9IChcXGNoaSwgXFx4aSlcXCB8XFwgICAgICBcXHRleHR7JFxcY2hpOiBHXFx0aW1lcyBHXFx0byBVKDEpJCBhIGNvbnRpbnVvdXMsfVxcXFwgICAgIFxcdGV4dHtzeW1tZXRyaWMgLCBub25kZWdlbmVkYXJlIGJpY2hhcmFjdGVyfVxcXFwgXFx0ZXh0e2FuZCAkXFx4aVxcaW5cXHtcXHBtMVxcfSR9ICAgXFxlbmR7YXJyYXl9XFxyaWdodFxcfX17XFxBdXQoRyl9Il0sWzIsMCwiXFxuaWNlZnJhY3tcXGxlZnRcXHtcXGJlZ2lue2FycmF5fXtsfVxcdGV4dHtUYW1iYXJhLVlhbWFnYW1pIH1XXipcXHRleHR7LX1cXFxcICAgICBcXHRleHR7dGVuc29yIGNhdGVnb3JpZXMgZm9yICRHJH0gICBcXGVuZHthcnJheX1cXHJpZ2h0XFx9fXtXXipcXHRleHR7LX1cXG90aW1lc1xcdGV4dHsgZXEufX0iXSxbMCwxLCJcXG5pY2VmcmFje1xcbGVmdFxce1xcYmVnaW57YXJyYXl9e2x9XFx0ZXh0e0NvbnRpbnVvdXMgVGFtYmFyYS1ZYW1hZ2FtaX1cXFxcICAgICBcXHRleHR7dGVuc29yIGNhdGVnb3JpZXMgZm9yICRHJH0gICBcXGVuZHthcnJheX1cXHJpZ2h0XFx9fXtcXHRleHR7Y29udGludW91cyAkXFxvdGltZXMkIGVxLn19Il0sWzAsMSwiXFxUWWNhbChHLCAtLC0pIl0sWzAsMiwiXFxUWWNhbChHLCAtLC0pIl0sWzIsMSwiXFxGZnJhayIsMl1d
\[\hspace{-.8cm}\begin{tikzcd}[column sep=small]
    \begin{array}{c}\nicefrac{\left\{\begin{array}{l}(\chi, \xi)\ |\ \text{$\chi\phantom{,}$: $G\times G\to U(1)$ a continuous}\\     \text{symmetric nondegenerate bicharacter} \\ \text{and $\xi\in\{\pm1\}$}   \end{array}\right\}}{\text{Aut}(G)}\end{array}
    &&
    \begin{array}{c} \nicefrac{\left\{\begin{array}{l}\text{Tambara-Yamagami }\mathrm{W}^*\text{-}\\     \text{tensor categories for $G$}   \end{array}\right\}}{\mathrm{W}^*-\otimes\ \text{equiv.}} \end{array}
    \\
    \begin{array}{c} \nicefrac{\left\{\begin{array}{l}\text{continuous Tambara-Yamagami}\\     \text{tensor categories for $G$}   \end{array}\right\}}{\text{continuous $\otimes$ equiv.}} \end{array}
	\arrow["{\TYcal(G, -,-)}", from=1-1, to=1-3]
	\arrow["{\TYcal(G, -, -)}"', from=1-1, to=2-1]
	\arrow["\Ffrak"', from=2-1, to=1-3]
\end{tikzcd}\]
\end{theorem}
\begin{proof}
Let $\chi: G\times G\to U(1)$ be a continuous symmetric nondegenerate bicharacter and $\xi\in\{\pm1\}$ be a sign. Let $\phi: G\to G$ be a continuous group automorphism and define $\chi':= \chi\circ (\phi^{-1}\times \phi^{-1})$. Then, the equivalence of $\mathrm{W}^*$-tensor categories $\TYcal(G, \chi, \xi)\xrightarrow{\cong}\TYcal(G, \chi',\xi)$ produced in the proof of Theorem \ref{thm: modAut(G)} clearly comes from an equivalence of continuous tensor categories. Therefore, the vertical arrow is well defined.

The diagonal arrow is well defined by Proposition \ref{Prop: ForgetE1} and taking identities as the isomorphisms needed in the definition of Tambara-Yamagami $\mathrm{W}^*$-tensor categories. The diagram commutes by definition, and the horizontal arrow is an isomorphism by Theorem \ref{thm: modAut(G)}. Hence the vertical arrow is injective. 

It remains only to show that the vertical arrow is surjective. To do so, we show that any continuous Tambara-Yamagami tensor category for $G$ is equivalent to some $\TYcal(G, \chi, \xi)$. 

Let $(A:=C_0(G)\oplus\CC, \TYcal_G,\alpha)$ be a continuous Tambara-Yamagami category. We know that there exists a symmetric nondegenerate bicharacter $\chi$ on $G$ and a sign $\xi\in\{\pm 1\}$ such that $\Ffrak(A, \TYcal_G, \alpha)\cong \Ffrak(\TYcal(G, \chi, \xi))$ as $\mathrm{W}^*$-tensor categories. Hence, it is enough to show that the $\mathrm{W}^*$-tensor equivalence produced in the proof of Theorem \ref{thm: modAut(G)} actually comes from an equivalence of continuous Tambara-Yamagami tensor categories for $G$ between $(A, \TYcal_G,\alpha)$ and $\TYcal(G, \chi, \xi)$. Actually, it is enough to show that the changes of coordinates described in Section \ref{sec: Automaticcontinuity} come from unitary intertwiners between the relevant $\mathrm{C}^*$-correspondences. More explicitly, we need to argue that
\begin{enumerate}
    \item $\theta$ comes from an intertwiner $C_0(G\times G)\to C_0(G\times G)$ of $C_0(G)-C_0(G\times G)$-correspondences,
    \item $\omega$ comes from an intertwiner $L^2(G)\to L^2(G)$ of $C_0(G)$-representations,
    \item $\psi$ comes from an intertwiner $C_0(G)\to C_0(G)$ of $\CC - C_0(G)$-correspondences.
\end{enumerate}
There is no need to argue about $\varphi$ because we defined it to be the constant function 1.

In Proposition \ref{prop: FirstStep}, we defined $\theta(\upsilon, \nu) = a_3(\upsilon,\nu)$. This change of coordinates comes from the intertwiner $C_0(G\times G)\to C_0(G\times G)$ as $C_0(G)-C_0(G\times G)$-correspondences given by multiplication by the continuous function $\mathrm{a}_3$. Next, in Proposition \ref{prop: SecondStep}, we defined $\omega(x) : = B(x)$ for some function $B\in L^\infty(G, U(1))$. Hence, this comes from the intertwiner $L^2(G)\to L^2(G)$ between $C_0(G)-\CC$-correspondences given by multiplication by $B$. 

It is left to argue that $\psi(x) = P(\mu)(x)$ as defined in Theorem \ref{thm: ThirdStep} comes from a unitary intertwiner between the relevant correspondences. Given $x\in G$, we write $M_x\in U(L^2(G))$ for the unitary $f\mapsto \big(y\mapsto \chi(x,y)f(y)\big)$ for every $f\in L^2(G)$. Then, the map $x\in G\mapsto M_x\in U(L^2(G))$ is strongly continuous by dominated convergence. Lemma \ref{Lemm: Function P} provides a measurable function $P: G\to U(1)$ such that $\mathrm{b}_{2}(x) = P(x)M_x$ for almost every $x\in G$. Fix a unit vector $f\in L^2(G)$ and write $\mathcal{P}: G\to \mathbb{C}$ for the map $\mathcal{P}(x) = \langle M_x^*\mathrm{b}_2(x)f, f\rangle$, which is $U(1)$-valued almost everywhere. Since $x\mapsto M_x$ and $x\mapsto b_2(x)$ are strongly continuous, $\mathcal{P}$ is continuous, and it agrees with $P$ almost everywhere. Since $|\mathcal{P}| = 1$ almost everywhere and the Haar measure has full support, we have that $|\mathcal{P}| = 1$ everywhere. The change of coordinate $\psi$ is induced by multiplication by $\mathcal{P}$ as a map $C_0(G)\to C_0(G)$ of $\mathbb{C}-C_0(G)$-correspondences. Hence, the claim follows.
\end{proof}

\begin{remark}
    If the locally compact group $G$ is not isomorphic to its Pontryagin dual $\hat{G}$, there exist no nondegenerate bicharacters on $G$, and hence there exist no Tambara-Yamagami $\mathrm{W}^*$-tensor categories or continuous Tambara-Yamagami tensor categories for $G$.
\end{remark}

\appendix
\section{Technical proofs}
\renewcommand{\thesection}{A} 
\label{App A}

\setcounter{definition}{0}

We present in this appendix some of the technical results needed in the proofs of Section~\ref{sec: Automaticcontinuity}. Throughout the appendix, $G$ denotes a locally compact abelian group and $\mu$ its Haar measure. We use additive notation for the group operation on $G$. In addition, $K$ denotes a Hilbert space. Recall that all topological spaces are locally compact, paracompact, Hausdorff, and second countable, all measure spaces are standard measure spaces and all Hilbert spaces are separable. The following lemma is needed in the proof of Proposition \ref{prop: SecondStep}.

\begin{lemma}\label{lemm: SolutionToEquationL2}
 Let $\phi\in L^\infty(\mu\times\mu, U(K))$ be a function such that 
	\begin{equation}\label{eq: L2equation}
	\phi(x+y, z) = \phi(y,-x+z)\cdot  \phi(x,z)
	\end{equation}
	as functions in $L^\infty(\mu\times\mu\times \mu, U(K))$. Then, there is a function $\Phi\in L^\infty(\mu, U(K))$ such that
	\[
	\phi(x,y) = \Phi(-x+y)^{-1}\cdot \Phi(y).
	\]
\end{lemma}
\begin{proof}
Pick a representative of the class of $\phi\in L^\infty(G\times G, U(K))$, which we still denote by $\phi$. Then, there is a subset $E\subset G^3$ of full measure such that
\[
\phi(x+y, z) = \phi(y,-x+z)\cdot  \phi(x,z)
\]
if $(x,y,z)\in E$. Define the subsets of $G$
\begin{align*}
C&:=\{w \in G\ |\ \text{for almost every $(x,y)\in G\times G$, $(-y+w,x, w)\in E$} \}\\
D&:=\{w\in G\ |\ \text{the function $x\mapsto \phi(-x+w, w)$ is measurable}\}.
\end{align*}
Since $E$ is of full measure, $C$ is of full measure, and by Fubini's Theorem, $D$ is of full measure. Hence, $C\cap D$ is of full measure and, in particular, not empty. Let $w\in C\cap D$ and define 
\[
\Phi(x):= \phi(-x+w, w)^{-1}.
\]
Then, since $w\in D$, we obtain a function $\Phi\in L^\infty(G;U(K))$. Also, since $w\in C$, for almost every pair $(x,y)\in G^2$, we have $(-y+w,x,w)\in E$, which implies that 
\[
\Phi(-x+y)^{-1}\cdot \Phi(y)= \phi(x-y+w, w)\cdot \phi(-y+w, w)^{-1} = \phi(x,y).
\]

\end{proof}

The following lemma is needed to prove Theorem \ref{thm: HisC} and Theorem \ref{thm: ThirdStep}. 

\begin{lemma}\label{Lemm: Function P}
    Let $T$ be a topological space and $\upsilon$ a measure on $T$. Let $\phi\in L^\infty(T\times G, \upsilon\times\mu, U(K))$ and $\Pi\in L^\infty(G\times T, \mu\times\upsilon, U(1))$ be functions such that
    \[
    \phi(t,x)\cdot \Pi(y,t) = \phi(t,x+y) \hspace{1cm}\Pi(x,t)\cdot\Pi(y,t) = \Pi(x+y,t)
    \]
   almost everywhere. Then, there exists a unique function $P\in L^\infty(T, \upsilon, U(K))$ such that
    \[
    \phi(t,x) = \Pi(x,t) \cdot P(t).
    \]
\end{lemma}
\begin{proof}
We take representatives of $\phi$ and $\Pi$, which we continue to denote $\phi$ and $\Pi$. Define, for all $t\in T$ and $x\in G$, the unitary $p(t,x) = \phi(t,x)\cdot \Pi(x,t)^{-1}$. Since $\Pi$ is scalar-valued, we find that for almost every $(x,y,t)\in G\times G\times T$,
\begin{align*}
    p(t,x+y) & = \phi(t,x+y)\cdot \Pi(x+y,t)^{-1}\\ & = \phi(t, x+y)\cdot \Pi(x,t)^{-1}\cdot \Pi(y, t)^{-1}\\ & = \phi(t, x)\cdot \Pi(x, t)^{-1}\\ & = p(t,x),
\end{align*}
implying that $p$ is essentially constant in the second variable. This gives a unique measurable function $P: T\to U(K)$ satisfying the required condition. 
\end{proof}

The next two results finish the proof of Theorem \ref{thm: ThirdStep}. Given $y\in G$, let $\sigma_y:L^2(G)\to L^2(G)$ denote the operator given by $\sigma_y(f)(x) = f(x+y)$.

\begin{lemma}\label{Lemm: nondegenerate}
 Let $\chi: G\times G\to U(1)$ be a continuous symmetric bicharacter on $G$. Assume that there exists a unitary operator $\gamma: L^2(G)\to L^2(G)$ such that, for every function $f\in L^2(G)$ and every $y\in G$, it holds that
 \[
 \gamma\Big(\chi(x,y)\cdot \gamma(f)\Big)(x) = \sigma_y(f)(-x).
 \]
 Then, the homomorphism $G\to \hat{G}$ given by $y\mapsto \chi(-,y)$ is injective.
\end{lemma}
\begin{proof}
Let $y\in G$ be such that
	\[
	\chi(x,y) = 1
	\]
for all $x\in G$. Denote by $e\in G$ the identity element. Then, for all $f\in L^2(G)$,
    \[
    \sigma_y(f)(-x)  = \gamma^2(f)(x)= \sigma_{e}(f)(-x)
    \]
 and hence $\sigma_y = \sigma_{e}$. It follows that $y = e$. 
\end{proof}

To conclude the proof of Theorem \ref{thm: ThirdStep}, we  need the following result, which characterizes the operator $\gamma$. Let us recall the following notation first. Given a continuous symmetric bicharacter $\chi$ we write $j_\chi: G\to \hat{G}$ for the map $j_\chi(y)(x) = \chi(x,y)$. Recall that we call $\chi$ nondegenerate if $j_\chi :G\to\hat{G}$ is an isomorphism. In this situation, we write $c_\chi\in\R_{>0}$ for the unique positive number such that $(j_\chi)_*\mu = c_\chi \hat{\mu}$, and we define $J_\chi: L^2(G)\to L^2(\hat{G})$ as the unitary $J_\chi(f)(\eta) = \sqrt{c_\chi}\cdot f\big(j_\chi^{-1}(\eta)\big)$. Recall that we write $\Fcal: L^2(\hat{G})\xrightarrow{\cong} L^2(G)$ for the unitary Fourier transform. We write
\[
    \Phi: L^2(G)\xrightarrow{J_\chi}  L^2(\hat{G}) \xrightarrow{\mathcal{F}}L^2(G)
\]
for the unitary defined as the composition of $\mathcal{F}$ and $J_\chi$ whenever $\chi$ is nondegenerate.

\begin{theorem}\label{thm: Schur}
	Let $G$ be a locally compact abelian group and let $\chi: G\times G\to U(1)$ be a continuous symmetric bicharacter such that the map $j_\chi:G\to \hat{G}$ given by $j_\chi(y)(x) = \chi(x,y)$ is injective. Let $\gamma:L^2(G)\to L^2(G)$ be a unitary operator such that, for almost every $x\in G$ and every $g\in L^2(G)$,
	\begin{equation}\label{eq: ConditionsGamma}
	\gamma\big(\chi(x,y)g(y)\big) = \gamma(g)(-x+y)\hspace{1cm}  \gamma\big(g(x+y)\big) = \chi(x,y)\gamma(g).
	\end{equation}
    Then, $\chi$ is nondegenerate and $\gamma: L^2(G)\to L^2(G)$ is given by
    \[
\gamma = \xi\cdot \Phi = \xi\cdot (\mathcal{F}\circ J_\chi)
    \]
	for some $\xi\in U(1).$
\end{theorem}
\begin{proof}
    We first prove that $\chi$ is nondegenerate. For any $x\in G$, we define the following operators on $L^2(G)$
    \[
    \begin{array}{cccc}
        \sigma_x: & L^2(G)&\to& L^2(G) \\
         & f&\mapsto &\big(y\mapsto f(x+y)\big)
    \end{array}
    \hspace{1cm}
        \begin{array}{cccc}
        M_x: & L^2(G)&\to& L^2(G) \\
         & f&\mapsto &\big(y\mapsto \chi(x,y)f(y)\big).
    \end{array}
    \]
 We also write $N_x:  L^2(\hat{G})\to L^2(\hat{G})$  for the operator sending $h\in L^2(\hat{G})$ to $N_x(h)(\eta) = \eta(x)h(\eta)$. Then Equations \eqref{eq: ConditionsGamma} read
    \begin{equation}\label{eq: ConditionsGammaCompact}
        \gamma\circ M_x = \sigma_{-x}\circ \gamma\hspace{2cm} \gamma\circ \sigma_x = M_x\circ \gamma
    \end{equation}
    for almost every $x\in G$. The maps $x\in G\mapsto \sigma_x\in B(L^2(G))$ and $x\in G\mapsto M_x\in B(L^2(G))$ are strongly continuous, the second fact following from continuity of $\chi$ and dominated convergence. By continuity, Equations \eqref{eq: ConditionsGammaCompact} hold for every $x\in G$.
    
    The second equation in \eqref{eq: ConditionsGammaCompact} implies that $\gamma^{-1}\circ M_x\circ \gamma = \sigma_x$, and hence we find
    \begin{equation}\label{eq: conjugate M to get N}
    \mathcal{F}\circ \gamma^{-1}\circ M_x\circ \gamma\circ \mathcal{F}^{-1} = \mathcal{F}\circ \sigma_x\circ \mathcal{F}^{-1} = N_x,
    \end{equation}
      where we also use the time-shifting property of the Fourier transform. We may decompose the operators $M_x$ and $N_x$ over $\hat{G}$ as follows. We define the projection-valued measure $E_\chi$ on $\hat{G}$ by sending any Borel subset $B\subset \hat{G}$ to $E_\chi(B) = \text{Mult}_{\pmb{1}_{j_\chi^{-1}(B)}}\in B(L^2(G))$, the operator given by multiplication by the indicator function of ${j_\chi^{-1}(B)}$. Similarly, we write $E_0(B) = \text{Mult}_{\pmb{1}_B}\in B(L^2(\hat{G}))$, the operator given by multiplication by the indicator function of $B$, which defines another projection-valued measure on $\hat{G}$. Then we find
    \[
    M_x = \int_{\eta\in \hat{G}} \eta(x) \mathrm{d}E_\chi(\eta)\hspace{2cm} N_x = \int_{\eta\in \hat{G}} \eta(x)\mathrm{d}E_0(\eta),
    \]
    which, together with Equation \eqref{eq: conjugate M to get N}, implies that $E_\chi$ and $E_0$ have the same null sets, by uniqueness of the spectral measure associated to a strongly continuous unitary representation of $G$. Therefore, for every Borel subset $B\subset \hat{G}$,
    \[
    \hat{\mu}(B) = 0\iff E_0(B) = 0\iff E_\chi(B) = 0\iff \mu(j_\chi^{-1}(B)) = 0\iff (j_\chi)_*\mu(B) = 0,
    \]
 and $(j_\chi)_*\mu$ and $\hat{\mu}$ are equivalent measures. Now, the subgroup $j_\chi(G)\subset \hat{G}$ is Borel, as $G$ is $\sigma$-compact, and since $(j_\chi)_*\mu$ is supported on $j_\chi(G)$ and $\hat{\mu}$ is equivalent to $(j_\chi)_*\mu$, it holds that $\hat{G}\setminus j_\chi(G)$ is null with respect to $\hat{\mu}$. In particular, $j_\chi(G)$ is a measurable subgroup of positive Haar measure, hence open by the Steinhaus-Weil Theorem. Moreover, 
    \begin{align*}
        j_\chi(G)^\perp:&=\{x\in G\ |\  \eta(x) = 1 \text{  for all $\eta\in j_\chi(G)$}\}\\ & = \{x\in G\ |\ \chi(x,y) = 1 \text{   for all $y\in G$}\}\\ & = \ker(j_\chi)\\
        & = \{e\},
    \end{align*}
    and by Pontryagin duality $j_\chi(G)\subset \hat{G}$ is dense. Since it is also open, we find that $j_\chi(G) =  \hat{G}$ and $j_\chi$ is a continuous bijective group homomorphism. By the open mapping theorem, since $G$ and $\hat{G}$ are locally compact groups, $j_\chi: G\to \hat{G}$ is an isomorphism, as needed.

 We now prove the second part of the statement. For convenience, we will work with $\gamma^{-1}:L^2(G)\to L^2(G)$. We can rewrite Equations \eqref{eq: ConditionsGamma} as
	\[
\gamma^{-1}\Big(g(-x+y)\Big) = \chi(x,y)\gamma^{-1}(g)(y)\hspace{1cm} \gamma^{-1}\Big(\chi(x,y)g(y)\Big) = \gamma^{-1}(g)(x+y).
	\]

    Since $y\in G\mapsto \chi(-,y)\in \hat{G}$ is injective by hypothesis, the subset $j_\chi(G)\subset \hat{G}$ is dense by Pontryagin duality. We will show that the composition $\Phi\circ \gamma^{-1}$ is a morphism between irreducible representations of the standard Heisenberg group $HG$ of $G$ defined by $\chi$. Let us first introduce $HG$. Recall that, for every $x\in G$, we have defined the operator $\sigma_x$ on $L^2(G)$~by
	\[
	\sigma_x(g)(z) = g(x+z).
	\]
	We define the standard Heisenberg group of $G$ by
	\[
	HG : = \{u\sigma_x \chi(y,-)\ |\ u\in U(1), x,y\in G\}\subset B\big(L^2(G)\big)
	\]
	as a subgroup of the unitary operators on $L^2(G)$. Note that 
	\[
	\big(u\sigma_x \chi(y,-)\big)\big(u'\sigma_{x'}\chi(y',-)\big) = \frac{uu'}{\chi(y,x')}\sigma_{x+x'}\chi(y+y', -).
	\]
	By definition, $HG$ acts on $L^2(G)$.
	We claim that the composition $\Phi\circ \gamma^{-1}$ is a morphism of $HG$-representations. Indeed, let $u\sigma_x \chi(y,-)\in HG$ and $g(z)\in L^2(G)$. Then,
	\begin{align*}
    \Phi\circ\gamma^{-1}\big(u\sigma_x \chi(y,-)g(z)\big) &= u\Phi\circ\gamma^{-1}\big(\chi(y,x+z)\cdot g(x+z)\big)\\&=u\Phi\Big(\chi(-x,z)\gamma^{-1}\big(\chi(y,z)\cdot g(z)\big)\Big)\\&=u\Phi\Big(\chi(-x,z)\cdot \gamma^{-1}\big(g\big)(y+z)\Big)\\ &= u \chi(y,x+z )\big[\Phi\circ\gamma^{-1}\big](g)(x+z),
	\end{align*}
	by hypothesis and Proposition \ref{prop: PropertiesF}. It is well known that $L^2(G)$ is an irreducible $HG$-representation, see for example \cite{Amritanshu11}. Hence, by Schur's Lemma \cite[13.1.4]{Dixmier}, the space of $HG$-equivariant bounded maps $L^2(G)\to L^2(G)$ is $\CC\cdot \id_{L^2(G)}$. Therefore, there is a scalar $\xi\in \CC$ such that
	\[
	\gamma(g) = \xi \cdot \Phi(g).
	\]
    As $\gamma$ and $\Phi$ are both unitary, $\xi\in U(1)$.
\end{proof}

\renewcommand{\thesection}{B} 
\section{Example. Tambara-Yamagami $\mathrm{W}^*$-tensor categories for $\mathbb{R}$}

We provide, as an example, the description of all Tambara-Yamagami $\mathrm{W}^*$-tensor categories for $\mathbb{R}$ under addition, up to $\mathrm{W}^*$-tensor automorphism. Any continuous symmetric nondegenerate bicharacter on $\R$ is of the form
\[
\begin{array}{cccc}
    \chi_a:&\R\times\R &\to &U(1)\\
    &(x,y)&\mapsto & e^{iaxy}
\end{array}
\]
for some $a\in \R\setminus\{0\}$. Up to the action of $\text{Aut}(\R) = \R^\times$, there are two continuous symmetric nondegenerate bicharacters on $\R$, namely
\[
\begin{array}{cccc}
    \chi_+:&\R\times\R &\to &U(1)\\
    &(x,y)&\mapsto & e^{ixy}
\end{array}
\hspace{2cm}
\begin{array}{cccc}
    \chi_-:&\R\times\R &\to &U(1)\\
    &(x,y)&\mapsto & e^{-ixy}.
\end{array}
\]
Fix $a\in \{\pm1\}$ and $\xi\in\{\pm1\}$. The underlying $\mathrm{W}^*$-category of $\TYcal(\R, \chi_a, \xi)$ is 
\[
\Rep(C_0(\R))\oplus\Hilb\cdot \tau.
\]
We write $\delta_x\in \Rep(C_0(\R))$ for the irreducible representation induced by $x\in\R$, which has underlying Hilbert space $\CC$ and $f\in C_0(\R)$ acts by multiplication by $f(x)$. More generally, given a continuous map $p:T\to \R$ and a measure $\upsilon$ on $T$, we produce the object $\int^\oplus_{t\in T}\delta_{p(t)}\mathrm{d}\upsilon(t)\in \Rep(C_0(\R))$, whose underlying Hilbert space is $L^2(T, \upsilon)$ and whose $C_0(\R)$-action is given, for $\phi\in C_0(\R)$ and $f\in L^2(T, \upsilon)$, by
\[
(\phi\cdot f)(t) := \phi(p(t))\cdot f(t).
\]
The full subcategory of $\Rep(C_0(\R))$ on objects of this type is equivalent to $\Rep(C_0(\R))$, by Proposition \ref{Prop: FullSubcat}, and hence we describe the tensor product and the associators of $\TYcal(\R, \chi_a ,\xi)$ on this subcategory. We write $\lambda$ for the Lebesgue (that is, Haar) measure on $\R$, and therefore we also write $\lambda$ for the regular representation of $\int^{\oplus}_{x\in\R}\delta_{x}\mathrm{d}\lambda(x)\in \Rep(C_0(\R))$. If $H$ is a Hilbert space, we denote by $H\cdot \upsilon$ the object of $\Rep(C_0(\R))$ whose underlying Hilbert space is the tensor product $H\otimes L^2(T, \upsilon)$ and such that $C_0(\R)$ acts by pointwise multiplication on the factor $L^2(T, \upsilon)$ after pulling back along $p:T\to \R$.

The tensor product in $\TYcal(\R, \chi_a, \xi)$ is given by
\[
\int^\oplus_{t\in T}\delta_{p(t)}\mathrm{d}\upsilon(t)\otimes \int^\oplus_{s\in S}\delta_{n(s)}\mathrm{d}\nu(s) = \int^{\oplus}_{(t,s)\in T\times S}\delta_{p(t)+n(s)}\mathrm{d}(\upsilon\times \nu)(t, s)
\]
and
\begin{align*}
 \upsilon\otimes\tau = L^2(T, \upsilon)\cdot \tau \hspace{2cm}\tau\otimes\upsilon = L^2(T, \upsilon)\cdot \tau\hspace{2cm} \tau\otimes\tau =\int^{\oplus}_{x\in \R}\delta_x\mathrm{d}\lambda(x) = \lambda,
\end{align*}
for objects $\int^\oplus_{t\in T}\delta_{p(t)}\mathrm{d}\upsilon(t),\, \int^\oplus_{s\in S}\delta_{n(s)}\mathrm{d}\nu(s)\in \Rep(C_0(\R))$.

Finally, we describe the associators of $\TYcal(\R, \chi_a, \xi)$. Let $\int^\oplus_{t\in T}\delta_{p(t)}\mathrm{d}\upsilon(t),\, \int^\oplus_{s\in S}\delta_{n(s)}\mathrm{d}\nu(s)$, $\int^{\oplus}_{r\in R}\delta_{z(r)}\mathrm{d}\eta(r)\in \Rep(C_0(\R))$. Then,
\[
\begin{array}{cccc}
		\alpha_{\upsilon, \nu, \eta}:& \upsilon\times\nu\times\eta&\xrightarrow{\id}& \upsilon\times\nu\times \eta\vspace{.3cm}\\
		\alpha_{\tau, \upsilon,\nu} = \alpha_{\upsilon, \nu, \tau} :& L^2(\upsilon\times \nu)\cdot \tau& \xrightarrow{\id} &L^2(\upsilon\times\nu )\cdot \tau\vspace{.3cm}\\
		\alpha_{\upsilon,\tau,\nu}  :& L^2(\upsilon\times \nu)\cdot \tau& \to &L^2(\upsilon\times\nu )\cdot \tau\\ 
		&f(t,s)&\mapsto& e^{i a p(t) n(s)}f(t,s)\vspace{.3cm}\\
\alpha_{\upsilon,\tau,\tau}:&L^2( \upsilon)\cdot \lambda& \to &\upsilon \times\lambda\\
		&f(t, x)&\mapsto& f\big(t, p(t)+x\big)\vspace{.3cm}\\
		\alpha_{\tau,\tau, \upsilon}:& \lambda\times\upsilon& \to &L^2(\upsilon)\cdot \lambda\\
		&f(x,t)&\mapsto& f\big(t,-p(t)+x\big)\vspace{.3cm}\\
\alpha_{\tau,\upsilon,\tau}:& L^2(\upsilon)\cdot \lambda &\to & L^2(\upsilon)\cdot \lambda\\
		&f(t,x)&\mapsto &e^{i a p(t)x} f(t,x)
	\end{array}
	\]
	and 
    \[
	\begin{array}{cccc}
	\alpha_{\tau,\tau,\tau} :& L^2(\R,\lambda)\cdot \tau&\to&L^2(\R,\lambda)\cdot \tau\\
    & f&\mapsto &\Big(x\mapsto  \xi\cdot\sqrt{\frac{1}{2\pi}}\cdot\int_{y\in \R} e^{-i a  xy}f(y)\text{d}\lambda(y)\Big).
	\end{array}
    \]
\bibliographystyle{halpha-abbrv}
\bibliography{TYbiblio}
\end{document}